\documentclass[a4paper]{amsart}
\usepackage[english]{babel}
\usepackage[utf8]{inputenc} 
\usepackage{amssymb,amsthm,amsmath} 
\usepackage{indentfirst}
\usepackage[makeroom]{cancel}
\usepackage{microtype}
\usepackage{color}
\usepackage{enumerate} 
\usepackage{tikz} 
\usetikzlibrary{cd} 
\usepackage{bbm} 
\usepackage{mathdots}
\usepackage[hidelinks,draft=false]{hyperref}
\usepackage[colorinlistoftodos,textsize=small,textwidth=80]{todonotes}
\usepackage[toc]{appendix}
\usepackage{mathtools}
\usepackage{cite}
\usepackage{stmaryrd} 
\usepackage[left=3cm, right=3cm, top=3cm, bottom=3cm]{geometry}
\usepackage{mathabx,epsfig}

\usepackage[all]{xy}

\hypersetup{
    linktoc=all,     
}

\makeatletter
\DeclareFontEncoding{LS1}{}{}
\DeclareFontSubstitution{LS1}{stix}{m}{n}
\DeclareMathAlphabet{\mathscr}{LS1}{stixscr}{m}{n}
\makeatother

\usetikzlibrary{arrows}

\usepackage{tikz}
\usepackage{pgfplots}
\pgfplotsset{compat=newest}
\usetikzlibrary{decorations.markings}

\usepackage{soul}

\newcommand{\Pic}{\operatorname{Pic}}

\newcommand{\K}{\mathbb{K}}
\newcommand{\Cl}{\operatorname{Cl}}
\newcommand{\Vect}{\operatorname{Vect}}
\newcommand{\tr}{\operatorname{tr}}

\newtheorem{theorem}{Theorem}[section]
\newtheorem{corollary}[theorem]{Corollary}
\newtheorem{lemma}[theorem]{Lemma}
\newtheorem{proposition}[theorem]{Proposition}
\theoremstyle{definition}
\newtheorem{definition}[theorem]{Definition}

\newtheorem{remark}[theorem]{Remark}

\numberwithin{equation}{section}

\title{Invertible projective 2-representations from invertible 2d TQFTs with defects}

\author{Domenico Fiorenza}
\address{Sapienza Universit\`a di Roma; Dipartimento di Matematica ``Guido Castelnuovo'', P.le Aldo Moro, 5 - 00185 - Roma, Italy; 
}
\email{fiorenza@mat.uniroma1.it}
\author{Chetan Vuppulury}
\address{Sapienza Universit\`a di Roma; Dipartimento di Matematica ``Guido Castelnuovo'', P.le Aldo Moro, 5 - 00185 - Roma, Italy; }
\email{chetan.vuppulury@uniroma1.it}

\begin{document}

\begin{abstract}
 We investigate invertible projective representations and their 2-categorical analogues using the language of TQFTs with defects. The main result is a freeness property for invertible projective representatios. While trivial in the 1-categorical setting, this result becomes interesting for 2-representations: as an application, only relying only on invertibility of Clifford algebras and Fock bimodules in the Morita 2-category of super vector spaces we recover Ludewig--Roos' result that the Clifford/Fock construction is a projective 2-representation of the category of Lagrangian correspondences.
  \end{abstract}

\maketitle

\setcounter{tocdepth}{1}
\tableofcontents

\section{Introduction}\label{sec:introduction}

The first rule of invertible projective representations is you don't talk of invertible projective representations. There are at least two good reasons for doing so. The first one is that invertible projective representations of a group $G$ are trivial. Indeed, a projective representation of $G$ is a group homomorphism $G\to PGL(V)$, where $V$ is some vector space over some field $\mathbb{K}$; if $V$ is invertible as an object in the monoidal category of vector spaces over $\mathbb{K}$, then $\dim_{\mathbb{K}}V=1$, which implies $PGL(V)$ is the trivial group. The second one is more interesting, and naturally extends to invertible projective representations of arbitrary categories with values in a symmetric monoidal category. 
A projective representation of $G$ can be seen as the datum of a collection of elements $\rho_g$ in $GL(V)$, one for each element $g$ in $G$, that satisfy the equation
\[
\rho_g\rho_h=\alpha_{g,h}\rho_{gh}
\]
for some $\mathbb{K}^\ast$-valued 2-cocycle $\alpha$ on $G$. Generally, i.e., when $\dim_{\mathbb{K}}V>1$, the existence of such a 2-cocycle given the elemens $\rho_g$ is quite a nontrivial condition. Yet, if $V$ is invertible, then one canonically has $GL(V)=\mathbb{K}^\ast$, and $\alpha$ is then obtained from the $\rho_g$ by setting
\[
\alpha_{g,h}=\rho_{gh}^{-1}\rho_g\rho_h.
\]
Notice that in this case the right hand side is automatically an element in $\mathbb{K}^\ast$, no condition is imposed. In other words, an invertible projective representation of $G$ is determined by a free assignment of elements  $\rho_g$ in $\mathbb{K}^\ast$. This freeness property 
 ultimately entirely relies on the invertibility of $V$ as object in the symmetric monoidal category $\Vect_\K$ and on the invertibility of the elements $\rho_{g}$ as morphisms from $V$ to $V$ in $\Vect_\K$. This is an instance of a much more general construction that we present in Section \ref{sec:invertible-1d} below: : we show that if $(\mathcal{V},\otimes)$ is a symmetric monoidal category with duals and $\mathcal{C}$ is an arbitrary 2-category, then freely assigning an invertible object $V_X$ in $\mathcal{V}$ with any object $X$ of $\mathcal{C}$ and an invertible morphism $\rho_f\colon V_X \to V_Y$ with any morphism $f\colon X \to Y$ in $\mathcal{V}$ determines a projective representation $\rho\colon \mathcal{C}\to \mathcal{V}/\!/\mathbf{B}\Pic(\Omega\mathcal{V})$. This is proved by making use of the language of 1-dimensional TQFTs with defects.
\footnote{The rules of topological manipulations of invertible defects appear to be widely known in the TQFT community. At the same time we have been unable to locate a systematic treatment in the literature. For invertible 1d defects, rules are described, e.g., in \cite{three-recipes}, while for non-invertible 1d defects they can be traced back at least to \cite{khovanov-heisenberg}. The rules for invertible 2d defects that we are going to extensively use in Section \ref{sec:invertible-2-rep}, on the other hand, appear not to have been systematically investigated in the literature. Yet, we are confident their use will be self-explanatory to the reader.
 We address the reader to \cite{atiyah,lurie-cobordism} for general background on topological quantum field theories, and to \cite{carqueville-defects} for details on TQFTs with defects.} 
One could argue that the interest in this result is limited, due to aforementioned triviality of invertible projective representations, and that the use of  the language of 1-dimensional TQFTs with defects to establish it, though fancy, is an overkill since equational reasoning would work as well. This is indeed a justified criticism: the main reason for Section \ref{sec:invertible-1d} is to prepare the reader to Section\ref{sec:invertible-2-rep}, where the results are generalized to projective 2-representations, and the freeness theorem for invertible projective 2-representations:  if $(2\mathcal{V},\otimes)$ is a symmetric monoidal 2-category with duals and $\mathcal{C}$ is an arbitrary 2-category, then freely assigning an invertible object $V_X$ in $\mathcal{V}$ with any object $X$ of $\mathcal{C}$ and an invertible morphism $\rho_f\colon V_X \to V_Y$ with any morphism $f\colon X \to Y$ in $\mathcal{V}$ determines a projective 2-representation $\rho\colon \mathcal{C}\to 2\mathcal{V}/\!/\mathbf{B}\Pic(\Omega2\mathcal{V})$. To prove this we make use of the full power of  2-dimensional TQFTs with defects, allowing us to ``draw the algebra''. With one dimension more, the use of surfaces to visualize morphisms turns out to be essential, reducing quite hard to handle diagrammatic or equational proofs to simple geometrical manipulations. More importantly, once one moves from 1-representations to 2-representations, the result is no more trivial in terms of applications: as we show in Section \ref{sec:clifford-fock}, when applied to the Clifford/Fock construction, it reproduces a recent result by Ludewig--Roos realizing the Clifford/Fock construction as a projective 2-representation of the category of Lagrangian correspondences \cite{ludewig-roos}. The defect TQFT approach shows in particular that this result actually does not rely on the properties of Clifford algebras and Fock bimodules, but rather entirely relies on their invertibility as objects and 1-morphisms in the Morita 2-category $2\mathrm{superVect}_\K$, where $\K$ is either $\mathbb{R}$ or $\mathbb{C}$. 

An important consequence of the freeness result on invertible projective 2-representations is that it is very easy to produce them. This offers a possible strategy to produce a linear 2-representation starting with a projective one, as follows: assume $\rho$ is a projective 2-representation of $\mathcal{C}$ with 2-cocycle $l_\rho$; if we can find a category $\tilde{\mathcal{C}}$ together with a functor $p\colon \mathcal{\tilde{C}}\to \mathcal{C}$ and an assignment of invertible objects and morphisms in $2\mathcal{V}$ to objects and morphisms in  $\mathcal{\tilde{C}}$ in such a way that the 2-cocycle $l_\beta$ of the associated projective 2-representation $\beta$ is  equivalent to $(p^\ast l_\rho)^{-1}$, then $J=\beta\otimes \rho$ is a linear 2-representation of $\mathcal{\tilde{C}}$. When $\rho$ takes invertible values there's a trivial way of adopting this strategy: one simply takes $\tilde{\mathcal{C}}=\mathcal{C}$ and the assignment on $\tilde{\mathcal{C}}$ to be the inverse of $\rho$. This produces the trivial 2-representation.
But remarkably, even when $\rho$ takes invertible values there are nontrivial examples of this strategy. For instance, when $\mathcal{C}$ is the category of Lagrangian correspondences, one can take $\tilde{\mathcal{C}}$ to be the category of Lagrangian spans (or generalized Lagrangian correspondences), and recover the fact that by considering twisted Fock bimodules one turns the  Clifford/Fock projective 2-representation of Lagrangian correspondences into a linear 2-representation of Lagrangian spans. An independent proof of this fact has been very recently given by Ludewig in \cite{ludewig}; the statement was already present as the Gluing Lemma 2.3.14 in Stolz and Teichner's \cite{what-is-an-elliptic-object}. 

The terminology and notation we use for 2-representations throughout the article is hopefully self-explanatory. A possible reference is \cite{fiorenza-vuppulury-projective-representations}.

\vskip 1 cm
This article is based on C.V. PhD Thesis \cite{vuppulury}. We thank Matthias Ludewig, Christoph Schweigert, and Jim Stasheff  for very useful comments and suggestions. D.F.  was partially supported by 2023 Sapienza research grant ``Representation Theory and Applications", by 2024 Sapienza research grant ``Global, local and infinitesimal aspects of moduli spaces'', and by PRIN 2022 ``Moduli spaces and special varieties'' CUP B53D23009140006. D.F. is a member of the Gruppo Nazionale per le Strutture Algebriche, Geometriche e le loro Applicazioni.
(GNSAGA-INdAM).

\section{Invertible projective representations and 1d TQFTs with defects}\label{sec:invertible-1d} 

To state the main result of this section we need a preliminary definition.
\begin{definition}
 Let $\mathcal{C}$ be a category and $\mathcal{D}$ be a 1-category. A \emph{wannabe functor}\footnote{It is likely that an established terminology should exist for this kind of association, but we have so far not been able to locate it in the literature.} $\rho\colon \mathcal{C}\dashrightarrow \mathcal{D}$ is a map of simplicial sets $\rho\colon\mathrm{sk}_1(\mathcal{C})\to\mathrm{sk}_1(\mathcal{D})$, where $\mathrm{sk}_1$ is the 1-skeleton of a simplicial set.
 \end{definition}
 \begin{remark}
 Explicitly, a wannabe functor $\rho\colon \mathcal{C}\dashrightarrow \mathcal{D}$ consists
 of the following associations:
 \begin{itemize}
 \item an object $V_{X_i}$ in $\mathcal{D}$ for any object $X_i$ in $\mathcal{C}$    
 \item a morphism $\rho_{f_{ij}}$ in $\mathrm{Hom}_{\mathcal{D}}(V_{X_i},V_{X_j})$ for any morphism $f_{ij}\colon X_i\to X_j$ in $\mathcal{C}$,
 \end{itemize}
 such that $\rho_{\mathrm{id}_{X_i}}=\mathrm{id}_{V_{X_i}}$ for any object $X_i$ in $\mathcal{C}$.
\end{remark}
As it is manifest from the above remark, since $\mathcal{D}$ is a 1-category, a wannabe functor only misses satisfying the equation
\[
\rho_{f_{ik}}=\rho_{f_{jk}}\circ \rho_{f_{ij}}
\]
in order to be an actual functor. This is quite a strict condition, so a general wannabe functor has no hope to be an actual functor. Yet, under invertibility assumptions, any wannabe functor canonically extends to a projective representation. To state this properly let us give the following.
\begin{definition}
    Let $\mathcal{C}$ be a category and $\mathcal{V}$ be a symmetric monoidal 1-category. A wannabe functor $\rho\colon \mathcal{C}\dashrightarrow \mathcal{V}$ is called \emph{invertible} if it factors through the 1-skeleton $\mathrm{sk}_1(\Pic(\mathcal{V}))$ of the 2-group $\Pic(\mathcal{V})$, i.e., if 
    \begin{itemize}
     \item the object $V_{X_i}$ is invertible in $(\mathcal{V},\otimes)$ for any $X_i$ in $\mathcal{C}$;
     \item the morphism $\rho_{f_{ij}}\colon V_{x_i}\to V_{X_j}$ is invertible for every $f_{ij}\colon X_i\to X_j$ in $\mathcal{C}$.
    \end{itemize}
     \end{definition}
Then we have the following.
\begin{proposition}\label{prop:wannabe-is-proj-0}
    Let $\mathcal{V}$ be a symmetric monoidal 1-category, and let $\rho \colon \mathcal{C}\dashrightarrow \mathcal{V}$ be an invertible wannabe functor. Then $\rho$ canonically extends to a projective representation $\rho\colon \mathcal{C}\to \mathcal{V}/\!/\mathbf{B}\Pic(\Omega\mathcal{V})$. More precisely, if we set
\[
\alpha_{\Xi_{ijk}}=\tr(\rho_{f_{ik}}^{-1}\circ \rho_{f_{jk}}\circ\rho_{f_{ij}})
\]
for any 2-simplex in $\mathcal{C}$, then $\alpha$ is a 2-cocycle on $\mathcal{C}$ with values in $\Pic(\Omega\mathcal{V})=\mathrm{Aut}_{\mathcal{V}}(\mathbf{1}_{\mathcal{V}})$ and $\rho$ is a  projective representation with 2-cocycle $\alpha$, i.e., it satisfies  
\[
 \rho_{f_{ik}}\cdot \alpha_{\Xi_{ijk}}=\rho_{f_{jk}}\circ \rho_{f_{ij}}
\]
for any 2-simplex in $\mathcal{C}$.
\end{proposition}
\begin{proof}
The proof is split into the following Lemmas \ref{lem:is-invertible}, \ref{lem:is-2-cocycle} and \ref{lem:satisfies-equation-proj}.
\end{proof}
{
\begin{remark}
In the above Proposition, `tr' is the trace in the symmetric monoidal category $\mathcal{V}$. We implicitly used the fact that an invertible object in a symmetric monoidal category is automatically dualizable, with dual given by its monoidal inverse. See, e.g., \cite{ponto-shulman} for the general theory of duality and traces in symmetric monoidal categories.
\end{remark}
}

\begin{lemma}\label{lem:is-invertible}
The element   $\alpha_{\Xi_{ijk}}$ is an element in  $\Pic(\Omega\mathcal{V})$, i.e., it is an invertible endomorphism of the unit object $\mathbf{1}_{\mathcal{V}}$ 
\end{lemma}
\begin{proof}
     To begin with, we make the graphical associations
\[
V_{X_i}\rightsquigarrow \raisebox{-.4cm}{
\begin{tikzpicture}
   \draw [red,thick,decoration={markings, mark=at position 0.5 with {\arrow{>}} },postaction={decorate},domain=-30:90] (0,-0.5) -- (0,0.5);
  \end{tikzpicture} }\quad;\quad 
  V_{X_j}\rightsquigarrow \raisebox{-.4cm}{
\begin{tikzpicture}
   \draw [blue,thick,decoration={markings, mark=at position 0.5 with {\arrow{>}} },postaction={decorate},domain=-30:90] (0,-0.5) -- (0,0.5);
  \end{tikzpicture} }\quad;\quad 
  V_{X_k}\rightsquigarrow \raisebox{-.4cm}{
\begin{tikzpicture}
   \draw [green,thick,decoration={markings, mark=at position 0.5 with {\arrow{>}} },postaction={decorate},domain=-30:90] (0,-0.5) -- (0,0.5);
  \end{tikzpicture} }
  \]
  \[
\rho_{f_{ij}}\rightsquigarrow \raisebox{-.4cm}{
\begin{tikzpicture}
   \draw [red,thick,decoration={markings, mark=at position 0.5 with {\arrow{>}} },postaction={decorate},domain=-30:90] (0,-0.5) -- (0,0);
   \draw [blue,thick,decoration={markings, mark=at position 0.7 with {\arrow{>}} },postaction={decorate},domain=-30:90] (0,0) -- (0,0.5);
   \node[black] at (0,0){\textbullet};
   \node[black] at (0.3,0){\footnotesize{$+$}};
  \end{tikzpicture} }\quad;\quad 
  \rho_{f_{jk}}\rightsquigarrow \raisebox{-.4cm}{
\begin{tikzpicture}
   \draw [blue,thick,decoration={markings, mark=at position 0.5 with {\arrow{>}} },postaction={decorate},domain=-30:90] (0,-0.5) -- (0,0);
   \draw [green,thick,decoration={markings, mark=at position 0.7 with {\arrow{>}} },postaction={decorate},domain=-30:90] (0,0) -- (0,0.5);
   \node[orange] at (0,0){\textbullet};
   \node[orange] at (0.3,0){\footnotesize{$+$}};
  \end{tikzpicture} }\quad;\quad 
  \rho_{f_{ik}}\rightsquigarrow \raisebox{-.4cm}{
\begin{tikzpicture}
   \draw [red,thick,decoration={markings, mark=at position 0.5 with {\arrow{>}} },postaction={decorate},domain=-30:90] (0,-0.5) -- (0,0);
   \draw [green,thick,decoration={markings, mark=at position 0.7 with {\arrow{>}} },postaction={decorate},domain=-30:90] (0,0) -- (0,0.5);
   \node[violet] at (0,0){\textbullet};
   \node[violet] at (0.3,0){\footnotesize{$+$}};
  \end{tikzpicture} }
  \]
    \[
\rho_{f_{ij}}^{-1}\rightsquigarrow \raisebox{-.4cm}{
\begin{tikzpicture}
   \draw [blue,thick,decoration={markings, mark=at position 0.5 with {\arrow{>}} },postaction={decorate},domain=-30:90] (0,-0.5) -- (0,0);
   \draw [red,thick,decoration={markings, mark=at position 0.7 with {\arrow{>}} },postaction={decorate},domain=-30:90] (0,0) -- (0,0.5);
   \node[black] at (0,0){\textbullet};
   \node[black] at (0.3,0){\footnotesize{$-$}};
  \end{tikzpicture} }\quad;\quad 
  \rho_{f_{jk}}^{-1}\rightsquigarrow \raisebox{-.4cm}{
\begin{tikzpicture}
   \draw [green,thick,decoration={markings, mark=at position 0.5 with {\arrow{>}} },postaction={decorate},domain=-30:90] (0,-0.5) -- (0,0);
   \draw [blue,thick,decoration={markings, mark=at position 0.7 with {\arrow{>}} },postaction={decorate},domain=-30:90] (0,0) -- (0,0.5);
   \node[orange] at (0,0){\textbullet};
   \node[orange] at (0.3,0){\footnotesize{$-$}};
  \end{tikzpicture} }\quad;\quad 
  \rho_{f_{ik}}^{-1}\rightsquigarrow \raisebox{-.4cm}{
\begin{tikzpicture}
   \draw [green,thick,decoration={markings, mark=at position 0.5 with {\arrow{>}} },postaction={decorate},domain=-30:90] (0,-0.5) -- (0,0);
   \draw [red,thick,decoration={markings, mark=at position 0.7 with {\arrow{>}} },postaction={decorate},domain=-30:90] (0,0) -- (0,0.5);
   \node[violet] at (0,0){\textbullet};
   \node[violet] at (0.3,0){\footnotesize{$-$}};
  \end{tikzpicture} }
  \]
  Then we have
\[
\alpha_{\Xi_{ijk}}=\raisebox{-.7cm}{
\begin{tikzpicture}
   \draw [red,thick,decoration={markings, mark=at position 0.5 with {\arrow{>}} },postaction={decorate},domain=-30:90] plot ({0.7*cos(\x)}, {0.7*sin(\x)});
   \draw [blue,thick,,decoration={markings, mark=at position 0.5 with {\arrow{>}} },postaction={decorate},domain=90:210] plot ({0.7*cos(\x)}, {0.7*sin(\x)});
   \draw [green,thick,,decoration={markings, mark=at position 0.5 with {\arrow{>}} },postaction={decorate},domain=210:330] plot ({0.7*cos(\x)}, {0.7*sin(\x)});
   \node[black] at ({0.7*cos(90}, {0.7*sin(90)}){\textbullet};
   \node[violet] at ({0.7*cos(-30)}, {0.7*sin(-30)}){\textbullet};
   \node[orange] at ({0.7*cos(210)}, {0.7*sin(210)}){\textbullet};
   \node[black] at ({0.95*cos(90}, {0.95*sin(90)}){\footnotesize{$+$}};
   \node[violet] at ({0.9*cos(-30)}, {0.9*sin(-30)}){\footnotesize{$-$}};
   \node[orange] at ({0.9*cos(210)}, {0.9*sin(210)}){\footnotesize{$+$}};
\end{tikzpicture}}
\]
We claim that the inverse of $\alpha_{\Xi_{ijk}}$ is
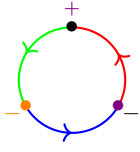
\[
\begin{tikzpicture}
   \draw [red,thick,decoration={markings, mark=at position 0.5 with {\arrow{>}} },postaction={decorate},domain=-30:90] plot ({0.7*cos(\x)}, {0.7*sin(\x)});
   \draw [green,thick,,decoration={markings, mark=at position 0.5 with {\arrow{>}} },postaction={decorate},domain=90:210] plot ({0.7*cos(\x)}, {0.7*sin(\x)});
   \draw [blue,thick,,decoration={markings, mark=at position 0.5 with {\arrow{>}} },postaction={decorate},domain=210:330] plot ({0.7*cos(\x)}, {0.7*sin(\x)});
   \node[violet] at ({0.7*cos(-30)}, {0.7*sin(-30)}){\textbullet};
   \node[black] at ({0.7*cos(90}, {0.7*sin(90)}){\textbullet};
   \node[orange] at ({0.7*cos(210)}, {0.7*sin(210)}){\textbullet};
   \node[black] at ({0.9*cos(-30)}, {0.9*sin(-30)}){\footnotesize{$-$}};
   \node[violet] at ({0.95*cos(90}, {0.95*sin(90)}){\footnotesize{$+$}};
   \node[orange] at ({0.9*cos(210)}, {0.9*sin(210)}){\footnotesize{$-$}};
\end{tikzpicture}
\]
To verify this we compute
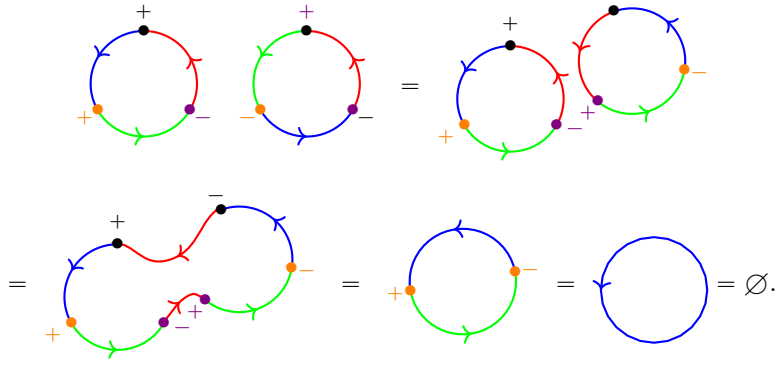
\[
\raisebox{-.7cm}{
\begin{tikzpicture}
   \draw [red,thick,decoration={markings, mark=at position 0.5 with {\arrow{>}} },postaction={decorate},domain=-30:90] plot ({0.7*cos(\x)}, {0.7*sin(\x)});
   \draw [blue,thick,,decoration={markings, mark=at position 0.5 with {\arrow{>}} },postaction={decorate},domain=90:210] plot ({0.7*cos(\x)}, {0.7*sin(\x)});
   \draw [green,thick,,decoration={markings, mark=at position 0.5 with {\arrow{>}} },postaction={decorate},domain=210:330] plot ({0.7*cos(\x)}, {0.7*sin(\x)});
   \node[black] at ({0.7*cos(90}, {0.7*sin(90)}){\textbullet};
   \node[violet] at ({0.7*cos(-30)}, {0.7*sin(-30)}){\textbullet};
   \node[orange] at ({0.7*cos(210)}, {0.7*sin(210)}){\textbullet};
   \node[black] at ({0.95*cos(90}, {0.95*sin(90)}){\footnotesize{$+$}};
   \node[violet] at ({0.9*cos(-30)}, {0.9*sin(-30)}){\footnotesize{$-$}};
   \node[orange] at ({0.9*cos(210)}, {0.9*sin(210)}){\footnotesize{$+$}};
\end{tikzpicture}
\begin{tikzpicture}
   \draw [red,thick,decoration={markings, mark=at position 0.5 with {\arrow{>}} },postaction={decorate},domain=-30:90] plot ({0.7*cos(\x)}, {0.7*sin(\x)});
   \draw [green,thick,,decoration={markings, mark=at position 0.5 with {\arrow{>}} },postaction={decorate},domain=90:210] plot ({0.7*cos(\x)}, {0.7*sin(\x)});
   \draw [blue,thick,,decoration={markings, mark=at position 0.5 with {\arrow{>}} },postaction={decorate},domain=210:330] plot ({0.7*cos(\x)}, {0.7*sin(\x)});
   \node[violet] at ({0.7*cos(-30)}, {0.7*sin(-30)}){\textbullet};
   \node[black] at ({0.7*cos(90}, {0.7*sin(90)}){\textbullet};
   \node[orange] at ({0.7*cos(210)}, {0.7*sin(210)}){\textbullet};
   \node[black] at ({0.9*cos(-30)}, {0.9*sin(-30)}){\footnotesize{$-$}};
   \node[violet] at ({0.95*cos(90}, {0.95*sin(90)}){\footnotesize{$+$}};
   \node[orange] at ({0.9*cos(210)}, {0.9*sin(210)}){\footnotesize{$-$}};
\end{tikzpicture}
}
=
\raisebox{-.9cm}{\begin{tikzpicture}
   \draw [red,thick,decoration={markings, mark=at position 0.5 with {\arrow{>}} },postaction={decorate},domain=-30:90] plot ({0.7*cos(\x)}, {0.7*sin(\x)});
   \draw [blue,thick,,decoration={markings, mark=at position 0.5 with {\arrow{>}} },postaction={decorate},domain=90:210] plot ({0.7*cos(\x)}, {0.7*sin(\x)});
   \draw [green,thick,,decoration={markings, mark=at position 0.5 with {\arrow{>}} },postaction={decorate},domain=210:330] plot ({0.7*cos(\x)}, {0.7*sin(\x)});
   \node[black] at ({0.7*cos(90)}, {0.7*sin(90)}){\textbullet};
   \node[violet] at ({0.7*cos(330)}, {0.7*sin(330)}){\textbullet};
   \node[orange] at ({0.7*cos(210)}, {0.7*sin(210)}){\textbullet};
   \node[black] at (0,1){\footnotesize{$+$}};
   \node[violet] at (0.85,-0.4){\footnotesize{$-$}};
   \node[orange] at (-0.85,-0.5){\footnotesize{$+$}};
   \draw [red,thick,decoration={markings, mark=at position 0.5 with {\arrow{>}} },postaction={decorate},domain=110:230] plot ({1.6+0.7*cos(\x)}, {.5+0.7*sin(\x)});
   \draw [green,thick,,decoration={markings, mark=at position 0.5 with {\arrow{>}} },postaction={decorate},domain=230:350] plot ({1.6+0.7*cos(\x)}, {.5+0.7*sin(\x)});
   \draw [blue,thick,,decoration={markings, mark=at position 0.5 with {\arrow{>}} },postaction={decorate},domain=-10:110] plot ({1.6+0.7*cos(\x)}, {.5+0.7*sin(\x)});
   \node[black] at ({1.6+0.7*cos(110)}, {.5+0.7*sin(110)}){\textbullet};
   \node[violet] at ({1.6+0.7*cos(230)}, {.5+0.7*sin(230)}){\textbullet};
   \node[orange] at ({1.6+0.7*cos(350)}, {.5+0.7*sin(350)}){\textbullet};
   \node[black] at ({1.6+0.9*cos(110)}, {.5+0.9*sin(110)}){\footnotesize{$-$}};
   \node[violet] at ({1.6+0.9*cos(230}, {.5+0.9*sin(230)}){\footnotesize{$+$}};
   \node[orange] at ({1.6+0.9*cos(350)}, {.5+0.9*sin(350)}){\footnotesize{$-$}};
\end{tikzpicture}
}
\]
\[
=
\raisebox{-.9cm}{\begin{tikzpicture}
   \draw [blue,thick,decoration={markings, mark=at position 0.5 with {\arrow{>}} },postaction={decorate},domain=90:210] plot ({0.7*cos(\x)}, {0.7*sin(\x)});
   \draw [green,thick,decoration={markings, mark=at position 0.5 with {\arrow{>}} },postaction={decorate},domain=210:330] plot ({0.7*cos(\x)}, {0.7*sin(\x)});
   \draw [green,thick,decoration={markings, mark=at position 0.5 with {\arrow{>}} },postaction={decorate},domain=230:350] plot ({1.6+0.7*cos(\x)}, {.5+0.7*sin(\x)});
   \draw [blue,thick,decoration={markings, mark=at position 0.5 with {\arrow{>}} },postaction={decorate},domain=-10:110] plot ({1.6+0.7*cos(\x)}, {.5+0.7*sin(\x)});
   \draw [red,thick,decoration={markings, mark=at position 0.5 with {\arrow{>}} },postaction={decorate}] plot [smooth, tension=1] coordinates { ({0.7*cos(330)}, {0.7*sin(330)}) ({0.7*cos(335)}, {0.7*sin(335)}) (.9,0) ({1.6+0.7*cos(225)}, {.5+0.7*sin(225)}) ({1.6+0.7*cos(230)}, {.5+0.7*sin(230)})};
   \draw [red,thick,decoration={markings, mark=at position 0.5 with {\arrow{>}} },postaction={decorate}] plot [smooth, tension=1] coordinates { ({1.6+0.7*cos(110)}, {.5+0.7*sin(110)}) ({1.6+0.7*cos(120)}, {.5+0.7*sin(120)}) (.7,.5) ({0.7*cos(85)}, {0.7*sin(85)}) ({0.7*cos(90)}, {0.7*sin(90)})};
   \node[black] at ({0.7*cos(90)}, {0.7*sin(90)}){\textbullet};
   \node[violet] at ({0.7*cos(330)}, {0.7*sin(330)}){\textbullet};
   \node[orange] at ({0.7*cos(210)}, {0.7*sin(210)}){\textbullet};
   \node[black] at (0,1){\footnotesize{$+$}};
   \node[violet] at (0.85,-0.4){\footnotesize{$-$}};
   \node[orange] at (-0.85,-0.5){\footnotesize{$+$}};
   \node[black] at ({1.6+0.7*cos(110)}, {.5+0.7*sin(110)}){\textbullet};
   \node[violet] at ({1.6+0.7*cos(230)}, {.5+0.7*sin(230)}){\textbullet};
   \node[orange] at ({1.6+0.7*cos(350)}, {.5+0.7*sin(350)}){\textbullet};
   \node[black] at ({1.6+0.9*cos(110)}, {.5+0.9*sin(110)}){\footnotesize{$-$}};
   \node[violet] at ({1.6+0.9*cos(230}, {.5+0.9*sin(230)}){\footnotesize{$+$}};
   \node[orange] at ({1.6+0.9*cos(350)}, {.5+0.9*sin(350)}){\footnotesize{$-$}};
   \end{tikzpicture}
   }
 =  
 \raisebox{-.7cm}{
 \begin{tikzpicture}
   \draw [blue,thick,,decoration={markings, mark=at position 0.5 with {\arrow{>}} },postaction={decorate},domain=10:190] plot ({0.7*cos(\x)}, {0.7*sin(\x)});
   \draw [green,thick,,decoration={markings, mark=at position 0.5 with {\arrow{>}} },postaction={decorate},domain=190:370] plot ({0.7*cos(\x)}, {0.7*sin(\x)});
   \node[orange] at ({0.7*cos(10)}, {0.7*sin(10)}){\textbullet};
   \node[orange] at ({0.9*cos(10)}, {0.9*sin(10)}){\footnotesize{$-$}};
   \node[orange] at ({0.7*cos(190)}, {0.7*sin(190)}){\textbullet};
   \node[orange] at ({0.9*cos(190)}, {0.9*sin(190)}){\footnotesize{$+$}};
\end{tikzpicture}}
 =  
 \raisebox{-.7cm}{
 \begin{tikzpicture}
   \draw [blue,thick,,decoration={markings, mark=at position 0.5 with {\arrow{>}} },postaction={decorate},domain=0:360] plot ({0.7*cos(\x)}, {0.7*sin(\x)});
\end{tikzpicture}}
=\emptyset.
\]
Here we used the invertibility of $V_{X_i}$ in the {second} step and that of $V_{X_j}$ in the last step.
\end{proof}
\begin{remark}
 Apparently, the invertibility of $V_{X_k}$ did not play a role in the proof of Lemma \ref{lem:is-invertible}. But actually already in the definition of $\alpha_{\Xi_{ijk}}$ one uses that $\rho_{f_{ik}}$ is an invertible morphism from $V_{X_i}$ to $V_{X_k}$. So the invertibility of $V_{X_i}$, that is used in the proof of  Lemma \ref{lem:is-invertible}, implies that of $V_{X_k}$.  
\end{remark}
\begin{lemma}\label{lem:is-2-cocycle}
 The elements $\alpha_{\Xi_{ijk}}$ define a 2-cocycle on $\mathcal{C}$ with values in $\Pic(\Omega\mathcal{V})$, i.e., they satisfy the equation    
\[
\alpha_{\Xi_{ikl}}\alpha_{\Xi_{ijk}}=\alpha_{\Xi_{ijl}}\alpha_{\Xi_{jkl}},
\]
for any 3-simplex in $\mathcal{C}$.    
\end{lemma}
\begin{proof}
 We add to our graphical associations the following ones: 
 \[
 V_{X_l}\rightsquigarrow \raisebox{-.4cm}
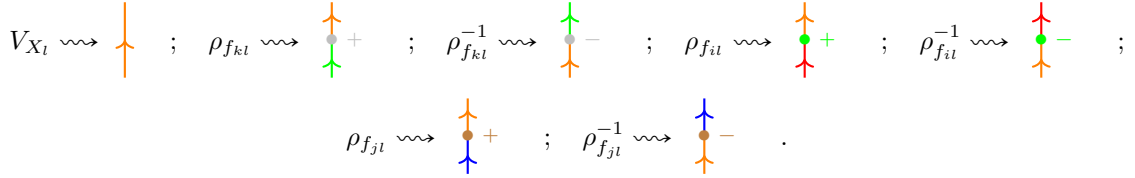
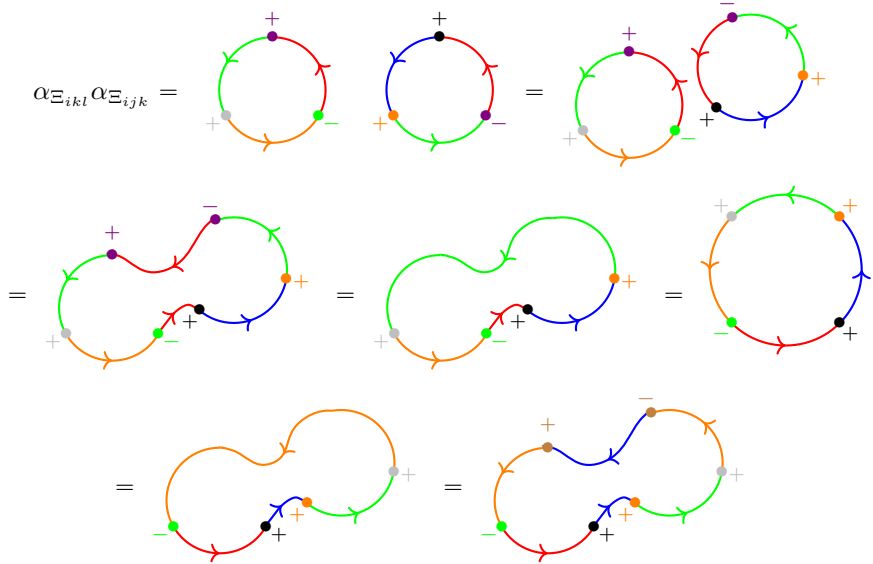
{
\begin{tikzpicture}
   \draw [orange,thick,decoration={markings, mark=at position 0.5 with {\arrow{>}} },postaction={decorate},domain=-30:90] (0,-0.5) -- (0,0.5);
  \end{tikzpicture} }
  \quad;\quad
  \rho_{f_{kl}}\rightsquigarrow \raisebox{-.4cm}{
\begin{tikzpicture}
   \draw [green,thick,decoration={markings, mark=at position 0.5 with {\arrow{>}} },postaction={decorate},domain=-30:90] (0,-0.5) -- (0,0);
   \draw [orange,thick,decoration={markings, mark=at position 0.7 with {\arrow{>}} },postaction={decorate},domain=-30:90] (0,0) -- (0,0.5);
   \node[lightgray] at (0,0){\textbullet};
   \node[lightgray] at (0.3,0){\footnotesize{$+$}};
  \end{tikzpicture} }\quad;\quad 
  \rho_{f_{kl}}^{-1}\rightsquigarrow \raisebox{-.4cm}{
\begin{tikzpicture}
   \draw [orange,thick,decoration={markings, mark=at position 0.5 with {\arrow{>}} },postaction={decorate},domain=-30:90] (0,-0.5) -- (0,0);
   \draw [green,thick,decoration={markings, mark=at position 0.7 with {\arrow{>}} },postaction={decorate},domain=-30:90] (0,0) -- (0,0.5);
   \node[lightgray] at (0,0){\textbullet};
   \node[lightgray] at (0.3,0){\footnotesize{$-$}};
  \end{tikzpicture} }
  \quad;\quad
  \rho_{f_{il}}\rightsquigarrow \raisebox{-.4cm}{
\begin{tikzpicture}
   \draw [red,thick,decoration={markings, mark=at position 0.5 with {\arrow{>}} },postaction={decorate},domain=-30:90] (0,-0.5) -- (0,0);
   \draw [orange,thick,decoration={markings, mark=at position 0.7 with {\arrow{>}} },postaction={decorate},domain=-30:90] (0,0) -- (0,0.5);
   \node[green] at (0,0){\textbullet};
   \node[green] at (0.3,0){\footnotesize{$+$}};
  \end{tikzpicture} }\quad;\quad 
  \rho_{f_{il}}^{-1}\rightsquigarrow \raisebox{-.4cm}{
\begin{tikzpicture}
   \draw [orange,thick,decoration={markings, mark=at position 0.5 with {\arrow{>}} },postaction={decorate},domain=-30:90] (0,-0.5) -- (0,0);
   \draw [red,thick,decoration={markings, mark=at position 0.7 with {\arrow{>}} },postaction={decorate},domain=-30:90] (0,0) -- (0,0.5);
   \node[green] at (0,0){\textbullet};
   \node[green] at (0.3,0){\footnotesize{$-$}};
  \end{tikzpicture} }\quad;
  \]
  \[
   \rho_{f_{jl}}\rightsquigarrow \raisebox{-.4cm}{
\begin{tikzpicture}
   \draw [blue,thick,decoration={markings, mark=at position 0.5 with {\arrow{>}} },postaction={decorate},domain=-30:90] (0,-0.5) -- (0,0);
   \draw [orange,thick,decoration={markings, mark=at position 0.7 with {\arrow{>}} },postaction={decorate},domain=-30:90] (0,0) -- (0,0.5);
   \node[brown] at (0,0){\textbullet};
   \node[brown] at (0.3,0){\footnotesize{$+$}};
  \end{tikzpicture} }\quad;\quad 
  \rho_{f_{jl}}^{-1}\rightsquigarrow \raisebox{-.4cm}{
\begin{tikzpicture}
   \draw [orange,thick,decoration={markings, mark=at position 0.5 with {\arrow{>}} },postaction={decorate},domain=-30:90] (0,-0.5) -- (0,0);
   \draw [blue,thick,decoration={markings, mark=at position 0.7 with {\arrow{>}} },postaction={decorate},domain=-30:90] (0,0) -- (0,0.5);
   \node[brown] at (0,0){\textbullet};
   \node[brown] at (0.3,0){\footnotesize{$-$}};
  \end{tikzpicture} }\quad.
  \]
Then we compute
\[
\alpha_{\Xi_{ikl}}\alpha_{\Xi_{ijk}}=
\raisebox{-.7cm}{
\begin{tikzpicture}
   \draw [red,thick,decoration={markings, mark=at position 0.5 with {\arrow{>}} },postaction={decorate},domain=-30:90] plot ({0.7*cos(\x)}, {0.7*sin(\x)});
   \draw [green,thick,,decoration={markings, mark=at position 0.5 with {\arrow{>}} },postaction={decorate},domain=90:210] plot ({0.7*cos(\x)}, {0.7*sin(\x)});
   \draw [orange,thick,,decoration={markings, mark=at position 0.5 with {\arrow{>}} },postaction={decorate},domain=210:330] plot ({0.7*cos(\x)}, {0.7*sin(\x)});
   \node[violet] at ({0.7*cos(90}, {0.7*sin(90)}){\textbullet};
   \node[green] at ({0.7*cos(-30)}, {0.7*sin(-30)}){\textbullet};
   \node[lightgray] at ({0.7*cos(210)}, {0.7*sin(210)}){\textbullet};
   \node[violet] at ({0.95*cos(90}, {0.95*sin(90)}){\footnotesize{$+$}};
   \node[green] at ({0.9*cos(-30)}, {0.9*sin(-30)}){\footnotesize{$-$}};
   \node[lightgray] at ({0.9*cos(210)}, {0.9*sin(210)}){\footnotesize{$+$}};
\end{tikzpicture}}\,
\raisebox{-.7cm}{
\begin{tikzpicture}
   \draw [red,thick,decoration={markings, mark=at position 0.5 with {\arrow{>}} },postaction={decorate},domain=-30:90] plot ({0.7*cos(\x)}, {0.7*sin(\x)});
   \draw [blue,thick,,decoration={markings, mark=at position 0.5 with {\arrow{>}} },postaction={decorate},domain=90:210] plot ({0.7*cos(\x)}, {0.7*sin(\x)});
   \draw [green,thick,,decoration={markings, mark=at position 0.5 with {\arrow{>}} },postaction={decorate},domain=210:330] plot ({0.7*cos(\x)}, {0.7*sin(\x)});
   \node[black] at ({0.7*cos(90}, {0.7*sin(90)}){\textbullet};
   \node[violet] at ({0.7*cos(-30)}, {0.7*sin(-30)}){\textbullet};
   \node[orange] at ({0.7*cos(210)}, {0.7*sin(210)}){\textbullet};
   \node[black] at ({0.95*cos(90}, {0.95*sin(90)}){\footnotesize{$+$}};
   \node[violet] at ({0.9*cos(-30)}, {0.9*sin(-30)}){\footnotesize{$-$}};
   \node[orange] at ({0.9*cos(210)}, {0.9*sin(210)}){\footnotesize{$+$}};
\end{tikzpicture}}
=
\raisebox{-.9cm}{\begin{tikzpicture}
   \draw [red,thick,decoration={markings, mark=at position 0.5 with {\arrow{>}} },postaction={decorate},domain=-30:90] plot ({0.7*cos(\x)}, {0.7*sin(\x)});
   \draw [green,thick,,decoration={markings, mark=at position 0.5 with {\arrow{>}} },postaction={decorate},domain=90:210] plot ({0.7*cos(\x)}, {0.7*sin(\x)});
   \draw [orange,thick,,decoration={markings, mark=at position 0.5 with {\arrow{>}} },postaction={decorate},domain=210:330] plot ({0.7*cos(\x)}, {0.7*sin(\x)});
   \node[violet] at ({0.7*cos(90}, {0.7*sin(90)}){\textbullet};
   \node[green] at ({0.7*cos(-30)}, {0.7*sin(-30)}){\textbullet};
   \node[lightgray] at ({0.7*cos(210)}, {0.7*sin(210)}){\textbullet};
   \node[violet] at ({0.95*cos(90}, {0.95*sin(90)}){\footnotesize{$+$}};
   \node[green] at ({0.9*cos(-30)}, {0.9*sin(-30)}){\footnotesize{$-$}};
   \node[lightgray] at ({0.9*cos(210)}, {0.9*sin(210)}){\footnotesize{$+$}};
   \draw [red,thick,decoration={markings, mark=at position 0.5 with {\arrow{>}} },postaction={decorate},domain=110:230] plot ({1.6+0.7*cos(\x)}, {.5+0.7*sin(\x)});
   \draw [blue,thick,,decoration={markings, mark=at position 0.5 with {\arrow{>}} },postaction={decorate},domain=230:350] plot ({1.6+0.7*cos(\x)}, {.5+0.7*sin(\x)});
   \draw [green,thick,,decoration={markings, mark=at position 0.5 with {\arrow{>}} },postaction={decorate},domain=-10:110] plot ({1.6+0.7*cos(\x)}, {.5+0.7*sin(\x)});
   \node[violet] at ({1.6+0.7*cos(110)}, {.5+0.7*sin(110)}){\textbullet};
   \node[black] at ({1.6+0.7*cos(230)}, {.5+0.7*sin(230)}){\textbullet};
   \node[orange] at ({1.6+0.7*cos(350)}, {.5+0.7*sin(350)}){\textbullet};
   \node[violet] at ({1.6+0.9*cos(110)}, {.5+0.9*sin(110)}){\footnotesize{$-$}};
   \node[black] at ({1.6+0.9*cos(230}, {.5+0.9*sin(230)}){\footnotesize{$+$}};
   \node[orange] at ({1.6+0.9*cos(350)}, {.5+0.9*sin(350)}){\footnotesize{$+$}};
\end{tikzpicture}
}
\]
\[
=
\raisebox{-.9cm}{\begin{tikzpicture}
   \draw [green,thick,decoration={markings, mark=at position 0.5 with {\arrow{>}} },postaction={decorate},domain=90:210] plot ({0.7*cos(\x)}, {0.7*sin(\x)});
   \draw [orange,thick,decoration={markings, mark=at position 0.5 with {\arrow{>}} },postaction={decorate},domain=210:330] plot ({0.7*cos(\x)}, {0.7*sin(\x)});
   \draw [blue,thick,decoration={markings, mark=at position 0.5 with {\arrow{>}} },postaction={decorate},domain=230:350] plot ({1.6+0.7*cos(\x)}, {.5+0.7*sin(\x)});
   \draw [green,thick,decoration={markings, mark=at position 0.5 with {\arrow{>}} },postaction={decorate},domain=-10:110] plot ({1.6+0.7*cos(\x)}, {.5+0.7*sin(\x)});
   \draw [red,thick,decoration={markings, mark=at position 0.5 with {\arrow{>}} },postaction={decorate}] plot [smooth, tension=1] coordinates { ({0.7*cos(330)}, {0.7*sin(330)}) ({0.7*cos(335)}, {0.7*sin(335)}) (.9,0) ({1.6+0.7*cos(225)}, {.5+0.7*sin(225)}) ({1.6+0.7*cos(230)}, {.5+0.7*sin(230)})};
   \draw [red,thick,decoration={markings, mark=at position 0.5 with {\arrow{>}} },postaction={decorate}] plot [smooth, tension=1] coordinates { ({1.6+0.7*cos(110)}, {.5+0.7*sin(110)}) ({1.6+0.7*cos(120)}, {.5+0.7*sin(120)}) (.7,.5) ({0.7*cos(85)}, {0.7*sin(85)}) ({0.7*cos(90)}, {0.7*sin(90)})};
   \node[violet] at ({0.7*cos(90)}, {0.7*sin(90)}){\textbullet};
   \node[green] at ({0.7*cos(330)}, {0.7*sin(330)}){\textbullet};
   \node[lightgray] at ({0.7*cos(210)}, {0.7*sin(210)}){\textbullet};
   \node[violet] at ({1*cos(90)}, {1*sin(90)}){\footnotesize{$+$}};
   \node[green] at ({0.9*cos(330)}, {0.9*sin(330)}){\footnotesize{$-$}};
   \node[lightgray] at ({0.9*cos(210)}, {0.9*sin(210)}){\footnotesize{$+$}};
   \node[violet] at ({1.6+0.7*cos(110)}, {.5+0.7*sin(110)}){\textbullet};
   \node[black] at ({1.6+0.7*cos(230)}, {.5+0.7*sin(230)}){\textbullet};
   \node[orange] at ({1.6+0.7*cos(350)}, {.5+0.7*sin(350)}){\textbullet};
   \node[violet] at ({1.6+0.9*cos(110)}, {.5+0.9*sin(110)}){\footnotesize{$-$}};
   \node[black] at ({1.6+0.9*cos(230}, {.5+0.9*sin(230)}){\footnotesize{$+$}};
   \node[orange] at ({1.6+0.9*cos(350)}, {.5+0.9*sin(350)}){\footnotesize{$+$}};
   \end{tikzpicture}
   }
=
\raisebox{-.9cm}{\begin{tikzpicture}
   \draw [green,thick,domain=90:210] plot ({0.7*cos(\x)}, {0.7*sin(\x)});
   \draw [orange,thick,decoration={markings, mark=at position 0.5 with {\arrow{>}} },postaction={decorate},domain=210:330] plot ({0.7*cos(\x)}, {0.7*sin(\x)});
   \draw [blue,thick,decoration={markings, mark=at position 0.5 with {\arrow{>}} },postaction={decorate},domain=230:350] plot ({1.6+0.7*cos(\x)}, {.5+0.7*sin(\x)});
   \draw [green,thick,domain=-10:100] plot ({1.6+0.7*cos(\x)}, {.5+0.7*sin(\x)});
   \draw [red,thick,decoration={markings, mark=at position 0.5 with {\arrow{>}} },postaction={decorate}] plot [smooth, tension=1] coordinates { ({0.7*cos(330)}, {0.7*sin(330)}) ({0.7*cos(335)}, {0.7*sin(335)}) (.9,0) ({1.6+0.7*cos(225)}, {.5+0.7*sin(225)}) ({1.6+0.7*cos(230)}, {.5+0.7*sin(230)})};
   \draw [green,thick,decoration={markings, mark=at position 0.5 with {\arrow{>}} },postaction={decorate}] plot [smooth, tension=1] coordinates { ({1.6+0.7*cos(100)}, {.5+0.7*sin(100)}) ({1.4+0.7*cos(120)}, {.43+0.7*sin(120)}) (.7,.5) ({0.1+0.7*cos(85)}, {0.67*sin(85)}) ({0.7*cos(88)}, {0.7*sin(88)})({0.7*cos(90)}, {0.7*sin(90)}) };
   \node[green] at ({0.7*cos(330)}, {0.7*sin(330)}){\textbullet};
   \node[lightgray] at ({0.7*cos(210)}, {0.7*sin(210)}){\textbullet};
   \node[green] at ({0.9*cos(330)}, {0.9*sin(330)}){\footnotesize{$-$}};
   \node[lightgray] at ({0.9*cos(210)}, {0.9*sin(210)}){\footnotesize{$+$}};
   \node[black] at ({1.6+0.7*cos(230)}, {.5+0.7*sin(230)}){\textbullet};
   \node[orange] at ({1.6+0.7*cos(350)}, {.5+0.7*sin(350)}){\textbullet};
   \node[black] at ({1.6+0.9*cos(230}, {.5+0.9*sin(230)}){\footnotesize{$+$}};
   \node[orange] at ({1.6+0.9*cos(350)}, {.5+0.9*sin(350)}){\footnotesize{$+$}};
   \end{tikzpicture}
   } 
=   
\raisebox{-.7cm}{
\begin{tikzpicture}
   \draw [red,thick,decoration={markings, mark=at position 0.5 with {\arrow{>}} },postaction={decorate},domain=225:315] plot ({cos(\x)}, {sin(\x)});
   \draw [blue,thick,,decoration={markings, mark=at position 0.5 with {\arrow{>}} },postaction={decorate},domain=-45:45] plot ({cos(\x)}, {sin(\x)});
   \draw [green,thick,decoration={markings, mark=at position 0.5 with {\arrow{>}} },postaction={decorate},domain=45:135] plot ({cos(\x)}, {sin(\x)});
   \draw [orange,thick,,decoration={markings, mark=at position 0.5 with {\arrow{>}} },postaction={decorate},domain=135:225] plot ({cos(\x)}, {sin(\x)});
   \node[orange] at ({cos(45}, {sin(45)}){\textbullet};
   \node[black] at ({cos(-45)}, {sin(-45)}){\textbullet};
   \node[green] at ({cos(225)}, {sin(225)}){\textbullet};
   \node[lightgray] at ({cos(135)}, {sin(135)}){\textbullet};
   \node[orange] at ({1.2*cos(45}, {1.2*sin(45)}){\footnotesize{$+$}};
   \node[black] at ({1.2*cos(-45)}, {1.2*sin(-45)}){\footnotesize{$+$}};
   \node[green] at ({1.2*cos(225)}, {1.2*sin(225)}){\footnotesize{$-$}};
   \node[lightgray] at ({1.2*cos(135)}, {1.2*sin(135)}){\footnotesize{$+$}};
\end{tikzpicture}}
\]
\[
=\raisebox{-.9cm}{\begin{tikzpicture}
   \draw [orange,thick,domain=90:210] plot ({0.7*cos(\x)}, {0.7*sin(\x)});
   \draw [red,thick,decoration={markings, mark=at position 0.5 with {\arrow{>}} },postaction={decorate},domain=210:330] plot ({0.7*cos(\x)}, {0.7*sin(\x)});
   \draw [green,thick,decoration={markings, mark=at position 0.5 with {\arrow{>}} },postaction={decorate},domain=230:350] plot ({1.6+0.7*cos(\x)}, {.5+0.7*sin(\x)});
   \draw [orange,thick,domain=-10:100] plot ({1.6+0.7*cos(\x)}, {.5+0.7*sin(\x)});
   \draw [blue,thick,decoration={markings, mark=at position 0.5 with {\arrow{>}} },postaction={decorate}] plot [smooth, tension=1] coordinates { ({0.7*cos(330)}, {0.7*sin(330)}) ({0.7*cos(335)}, {0.7*sin(335)}) (.9,0) ({1.6+0.7*cos(225)}, {.5+0.7*sin(225)}) ({1.6+0.7*cos(230)}, {.5+0.7*sin(230)})};
   \draw [orange,thick,decoration={markings, mark=at position 0.5 with {\arrow{>}} },postaction={decorate}] plot [smooth, tension=1] coordinates { ({1.6+0.7*cos(100)}, {.5+0.7*sin(100)}) ({1.4+0.7*cos(120)}, {.43+0.7*sin(120)}) (.7,.5) ({0.1+0.7*cos(85)}, {0.67*sin(85)}) ({0.7*cos(88)}, {0.7*sin(88)})({0.7*cos(90)}, {0.7*sin(90)}) };
   \node[black] at ({0.7*cos(330)}, {0.7*sin(330)}){\textbullet};
   \node[green] at ({0.7*cos(210)}, {0.7*sin(210)}){\textbullet};
   \node[black] at ({0.9*cos(330)}, {0.9*sin(330)}){\footnotesize{$+$}};
   \node[green] at ({0.9*cos(210)}, {0.9*sin(210)}){\footnotesize{$-$}};
   \node[orange] at ({1.6+0.7*cos(230)}, {.5+0.7*sin(230)}){\textbullet};
   \node[lightgray] at ({1.6+0.7*cos(350)}, {.5+0.7*sin(350)}){\textbullet};
   \node[orange] at ({1.6+0.9*cos(230}, {.5+0.9*sin(230)}){\footnotesize{$+$}};
   \node[lightgray] at ({1.6+0.9*cos(350)}, {.5+0.9*sin(350)}){\footnotesize{$+$}};
   \end{tikzpicture}
   } 
   =
\raisebox{-.9cm}{\begin{tikzpicture}
   \draw [orange,thick,decoration={markings, mark=at position 0.5 with {\arrow{>}} },postaction={decorate},domain=90:210] plot ({0.7*cos(\x)}, {0.7*sin(\x)});
   \draw [red,thick,decoration={markings, mark=at position 0.5 with {\arrow{>}} },postaction={decorate},domain=210:330] plot ({0.7*cos(\x)}, {0.7*sin(\x)});
   \draw [green,thick,decoration={markings, mark=at position 0.5 with {\arrow{>}} },postaction={decorate},domain=230:350] plot ({1.6+0.7*cos(\x)}, {.5+0.7*sin(\x)});
   \draw [orange,thick,decoration={markings, mark=at position 0.5 with {\arrow{>}} },postaction={decorate},domain=-10:110] plot ({1.6+0.7*cos(\x)}, {.5+0.7*sin(\x)});
   \draw [blue,thick,decoration={markings, mark=at position 0.5 with {\arrow{>}} },postaction={decorate}] plot [smooth, tension=1] coordinates { ({0.7*cos(330)}, {0.7*sin(330)}) ({0.7*cos(335)}, {0.7*sin(335)}) (.9,0) ({1.6+0.7*cos(225)}, {.5+0.7*sin(225)}) ({1.6+0.7*cos(230)}, {.5+0.7*sin(230)})};
   \draw [blue,thick,decoration={markings, mark=at position 0.5 with {\arrow{>}} },postaction={decorate}] plot [smooth, tension=1] coordinates { ({1.6+0.7*cos(110)}, {.5+0.7*sin(110)}) ({1.6+0.7*cos(120)}, {.5+0.7*sin(120)}) (.7,.5) ({0.7*cos(85)}, {0.7*sin(85)}) ({0.7*cos(90)}, {0.7*sin(90)})};
   \node[brown] at ({0.7*cos(90)}, {0.7*sin(90)}){\textbullet};
   \node[black] at ({0.7*cos(330)}, {0.7*sin(330)}){\textbullet};
   \node[green] at ({0.7*cos(210)}, {0.7*sin(210)}){\textbullet};
   \node[brown] at ({1*cos(90)}, {1*sin(90)}){\footnotesize{$+$}};
   \node[black] at ({0.9*cos(330)}, {0.9*sin(330)}){\footnotesize{$+$}};
   \node[green] at ({0.9*cos(210)}, {0.9*sin(210)}){\footnotesize{$-$}};
   \node[brown] at ({1.6+0.7*cos(110)}, {.5+0.7*sin(110)}){\textbullet};
   \node[orange] at ({1.6+0.7*cos(230)}, {.5+0.7*sin(230)}){\textbullet};
   \node[lightgray] at ({1.6+0.7*cos(350)}, {.5+0.7*sin(350)}){\textbullet};
   \node[brown] at ({1.6+0.9*cos(110)}, {.5+0.9*sin(110)}){\footnotesize{$-$}};
   \node[orange] at ({1.6+0.9*cos(230}, {.5+0.9*sin(230)}){\footnotesize{$+$}};
   \node[lightgray] at ({1.6+0.9*cos(350)}, {.5+0.9*sin(350)}){\footnotesize{$+$}};
   \end{tikzpicture}
   }
\]
\[
=\raisebox{-.9cm}{\begin{tikzpicture}
   \draw [blue,thick,decoration={markings, mark=at position 0.5 with {\arrow{>}} },postaction={decorate},domain=-30:90] plot ({0.7*cos(\x)}, {0.7*sin(\x)});
   \draw [orange,thick,,decoration={markings, mark=at position 0.5 with {\arrow{>}} },postaction={decorate},domain=90:210] plot ({0.7*cos(\x)}, {0.7*sin(\x)});
   \draw [red,thick,,decoration={markings, mark=at position 0.5 with {\arrow{>}} },postaction={decorate},domain=210:330] plot ({0.7*cos(\x)}, {0.7*sin(\x)});
   \node[brown] at ({0.7*cos(90}, {0.7*sin(90)}){\textbullet};
   \node[black] at ({0.7*cos(-30)}, {0.7*sin(-30)}){\textbullet};
   \node[green] at ({0.7*cos(210)}, {0.7*sin(210)}){\textbullet};
   \node[brown] at ({0.95*cos(90}, {0.95*sin(90)}){\footnotesize{$+$}};
   \node[black] at ({0.9*cos(-30)}, {0.9*sin(-30)}){\footnotesize{$+$}};
   \node[green] at ({0.9*cos(210)}, {0.9*sin(210)}){\footnotesize{$-$}};
   \draw [blue,thick,decoration={markings, mark=at position 0.5 with {\arrow{>}} },postaction={decorate},domain=110:230] plot ({1.6+0.7*cos(\x)}, {.5+0.7*sin(\x)});
   \draw [green,thick,,decoration={markings, mark=at position 0.5 with {\arrow{>}} },postaction={decorate},domain=230:350] plot ({1.6+0.7*cos(\x)}, {.5+0.7*sin(\x)});
   \draw [orange,thick,,decoration={markings, mark=at position 0.5 with {\arrow{>}} },postaction={decorate},domain=-10:110] plot ({1.6+0.7*cos(\x)}, {.5+0.7*sin(\x)});
   \node[brown] at ({1.6+0.7*cos(110)}, {.5+0.7*sin(110)}){\textbullet};
   \node[orange] at ({1.6+0.7*cos(230)}, {.5+0.7*sin(230)}){\textbullet};
   \node[lightgray] at ({1.6+0.7*cos(350)}, {.5+0.7*sin(350)}){\textbullet};
   \node[brown] at ({1.6+0.9*cos(110)}, {.5+0.9*sin(110)}){\footnotesize{$-$}};
   \node[orange] at ({1.6+0.9*cos(230}, {.5+0.9*sin(230)}){\footnotesize{$+$}};
   \node[lightgray] at ({1.6+0.9*cos(350)}, {.5+0.9*sin(350)}){\footnotesize{$+$}};
\end{tikzpicture}
}=
\raisebox{-.7cm}{
\begin{tikzpicture}
   \draw [red,thick,decoration={markings, mark=at position 0.5 with {\arrow{>}} },postaction={decorate},domain=-30:90] plot ({0.7*cos(\x)}, {0.7*sin(\x)});
   \draw [blue,thick,,decoration={markings, mark=at position 0.5 with {\arrow{>}} },postaction={decorate},domain=90:210] plot ({0.7*cos(\x)}, {0.7*sin(\x)});
   \draw [orange,thick,,decoration={markings, mark=at position 0.5 with {\arrow{>}} },postaction={decorate},domain=210:330] plot ({0.7*cos(\x)}, {0.7*sin(\x)});
   \node[black] at ({0.7*cos(90}, {0.7*sin(90)}){\textbullet};
   \node[green] at ({0.7*cos(-30)}, {0.7*sin(-30)}){\textbullet};
   \node[brown] at ({0.7*cos(210)}, {0.7*sin(210)}){\textbullet};
   \node[black] at ({0.95*cos(90}, {0.95*sin(90)}){\footnotesize{$+$}};
   \node[green] at ({0.9*cos(-30)}, {0.9*sin(-30)}){\footnotesize{$-$}};
   \node[brown] at ({0.9*cos(210)}, {0.9*sin(210)}){\footnotesize{$+$}};
\end{tikzpicture}}\,
\raisebox{-.7cm}{
\begin{tikzpicture}
   \draw [blue,thick,decoration={markings, mark=at position 0.5 with {\arrow{>}} },postaction={decorate},domain=-30:90] plot ({0.7*cos(\x)}, {0.7*sin(\x)});
   \draw [green,thick,,decoration={markings, mark=at position 0.5 with {\arrow{>}} },postaction={decorate},domain=90:210] plot ({0.7*cos(\x)}, {0.7*sin(\x)});
   \draw [orange,thick,,decoration={markings, mark=at position 0.5 with {\arrow{>}} },postaction={decorate},domain=210:330] plot ({0.7*cos(\x)}, {0.7*sin(\x)});
   \node[orange] at ({0.7*cos(90}, {0.7*sin(90)}){\textbullet};
   \node[brown] at ({0.7*cos(-30)}, {0.7*sin(-30)}){\textbullet};
   \node[lightgray] at ({0.7*cos(210)}, {0.7*sin(210)}){\textbullet};
   \node[orange] at ({0.95*cos(90}, {0.95*sin(90)}){\footnotesize{$+$}};
   \node[brown] at ({0.9*cos(-30)}, {0.9*sin(-30)}){\footnotesize{$-$}};
   \node[lightgray] at ({0.9*cos(210)}, {0.9*sin(210)}){\footnotesize{$+$}};
\end{tikzpicture}}
=\alpha_{\Xi_{ijl}}\alpha_{\Xi_{jkl}}.
\]

\end{proof}

\begin{lemma}\label{lem:satisfies-equation-proj}
The assignment $f_{ij}\mapsto \rho_{f_{ij}}$; $\Xi_{ijk}\mapsto \alpha_{\Xi_{ijk}}$ defines a projective representation $\rho\colon \mathcal{C}\to \mathcal{V}/\!/\mathbf{B}\Pic(\Omega\mathcal{V})$ with associated cocycle $\alpha_{\Xi_{ijk}}$, i.e., we have 
\[
\rho_{f_{jk}}\circ \rho_{f_{ij}} = \rho_{f_{ik}}\cdot \alpha_{\Xi_{ijk}}
\]
for any 2-simplex in $\mathcal{C}$.
\end{lemma}
\begin{proof}
We compote
\[
\rho_{f_{jk}}\circ \rho_{f_{ij}}=
\raisebox{-2.1cm}{
\begin{tikzpicture}
\draw [red,thick,decoration={markings, mark=at position 0.5 with {\arrow{>}} },postaction={decorate},domain=-30:90] (0,-2) -- (0,-0.5);
   \draw [blue,thick,decoration={markings, mark=at position 0.5 with {\arrow{>}} },postaction={decorate},domain=-30:90] (0,-0.5) -- (0,0.5);
   \draw [green,thick,decoration={markings, mark=at position 0.7 with {\arrow{>}} },postaction={decorate},domain=-30:90] (0,0.5) -- (0,2);
   \node[orange] at (0,0.5){\textbullet};
   \node[orange] at (0.3,0.5){\footnotesize{$+$}};
   \node[black] at (0,-0.5){\textbullet};
   \node[black] at (0.3,-0.5){\footnotesize{$+$}};
  \end{tikzpicture} }
  =
\raisebox{-2.1cm}{
\begin{tikzpicture}
\draw [red,thick,decoration={markings, mark=at position 0.5 with {\arrow{>}} },postaction={decorate},domain=-30:90] (0,-2) -- (0,-0.5);
   \draw [blue,thick,decoration={markings, mark=at position 0.5 with {\arrow{>}} },postaction={decorate},domain=-30:90] (0,-0.5) -- (0,0.5);
   \draw [green,thick,decoration={markings, mark=at position 0.7 with {\arrow{>}} },postaction={decorate},domain=-30:90] (0,0.5) -- (0,1);
   \draw [red,thick,decoration={markings, mark=at position 0.7 with {\arrow{>}} },postaction={decorate},domain=-30:90] (0,1) -- (0,1.5);
   \draw [green,thick,decoration={markings, mark=at position 0.7 with {\arrow{>}} },postaction={decorate},domain=-30:90] (0,1.5) -- (0,2);
   \node[orange] at (0,0.5){\textbullet};
   \node[orange] at (0.3,0.5){\footnotesize{$+$}};
   \node[black] at (0,-0.5){\textbullet};
   \node[black] at (0.3,-0.5){\footnotesize{$+$}};
   \node[violet] at (0,1){\textbullet};
   \node[violet] at (0.3,1){\footnotesize{$-$}};
   \node[violet] at (0,1.5){\textbullet};
   \node[violet] at (0.3,1.5){\footnotesize{$+$}};
  \end{tikzpicture} }
  =
  \raisebox{-2.1cm}{
\begin{tikzpicture}
 \draw [red,thick,decoration={markings, mark=at position 0.6 with {\arrow{>}} },postaction={decorate}] plot [smooth, tension=.8] coordinates { (0,-2) (0,-1.5) (0.1,-.5) (0.5,0) (1,-1) (1.55,-0.5) (1.7,0)};
 \draw [blue,thick,decoration={markings, mark=at position 0.6 with {\arrow{>}} },postaction={decorate}] plot [smooth, tension=.8] coordinates { (1.7,0) (1.7,0.6)};
 \draw [green,thick,decoration={markings, mark=at position 0.6 with {\arrow{>}} },postaction={decorate}] plot [smooth, tension=.8] coordinates { (1.7,0.6) (1.7,1.2)};
 \draw [red,thick,decoration={markings, mark=at position 0.6 with {\arrow{>}} },postaction={decorate}] plot [smooth, tension=.8] coordinates { (1.7,1.2) (1.6,1.5) (1.2,1.7) (0.75,0.8) (0.3,0.4)  (0,1)      };
 \draw [green,thick,decoration={markings, mark=at position 0.6 with {\arrow{>}} },postaction={decorate}] plot [smooth, tension=.8] coordinates {  (0,1) (0,2)      };
   \node[orange] at (1.7,0.6){\textbullet};
   \node[orange] at (1.9,0.6){\footnotesize{$+$}};
   \node[black] at (1.7,0){\textbullet};
   \node[black] at (1.9,0){\footnotesize{$+$}};
   \node[violet] at (0,1){\textbullet};
   \node[violet] at (0.3,1){\footnotesize{$+$}};
   \node[violet] at (1.7,1.2){\textbullet};
   \node[violet] at (1.9,1.2){\footnotesize{$-$}};
  \end{tikzpicture} }
=
\raisebox{-2.1cm}{
\begin{tikzpicture}
\draw [red,thick,decoration={markings, mark=at position 0.5 with {\arrow{>}} },postaction={decorate},domain=-30:90] (0,-2) -- (0,0);
   \draw [green,thick,decoration={markings, mark=at position 0.7 with {\arrow{>}} },postaction={decorate},domain=-30:90] (0,0) -- (0,2);
   \node[violet] at (0,0){\textbullet};
   \node[violet] at (0.3,0){\footnotesize{$+$}};
  \end{tikzpicture} 
  }
  \,
  \raisebox{-.7cm}{
  \begin{tikzpicture}
   \draw [red,thick,decoration={markings, mark=at position 0.5 with {\arrow{>}} },postaction={decorate},domain=-30:90] plot ({0.7*cos(\x)}, {0.7*sin(\x)});
   \draw [blue,thick,,decoration={markings, mark=at position 0.5 with {\arrow{>}} },postaction={decorate},domain=90:210] plot ({0.7*cos(\x)}, {0.7*sin(\x)});
   \draw [green,thick,,decoration={markings, mark=at position 0.5 with {\arrow{>}} },postaction={decorate},domain=210:330] plot ({0.7*cos(\x)}, {0.7*sin(\x)});
   \node[black] at ({0.7*cos(90}, {0.7*sin(90)}){\textbullet};
   \node[violet] at ({0.7*cos(-30)}, {0.7*sin(-30)}){\textbullet};
   \node[orange] at ({0.7*cos(210)}, {0.7*sin(210)}){\textbullet};
   \node[black] at ({0.95*cos(90}, {0.95*sin(90)}){\footnotesize{$+$}};
   \node[violet] at ({0.9*cos(-30)}, {0.9*sin(-30)}){\footnotesize{$-$}};
   \node[orange] at ({0.9*cos(210)}, {0.9*sin(210)}){\footnotesize{$+$}};
\end{tikzpicture}
  }
  =
  \rho_{f_{ik}}\cdot \alpha_{\Xi_{ijk}}.
\]

\end{proof}

\section{Invertible projective 2-representations and 2d TQFTs with defects}\label{sec:invertible-2-rep}
Here,  using the formalism of 2d TQFTs with defects we prove the 2-categorical version of the results in Section \ref{sec:invertible-1d}.  As an application, in Section \ref{sec:clifford-fock} we will recover Ludewig--Roos' result that the Clifford/Fock construction naturally extends to a projective 2-representation of the category of (finite dimensional\footnote{Ludewig and Roos investigate also the infinite dimensional case; namely, Lagrangian correspondences in separable Hilbert spaces. Here, not to be bothered by convergence and closedness issues, we will content us with the finite dimensional case. See Remark \ref{rem:von-neumann} for additional details on the infinite dimensional case.}) Lagrangian correspondences and Ludewig's result that the twisted Clifford/Fock construction is a 2-representation of (finite dimensional) Lagrangian spans. All definitions and constructions will be carried on in an arbitrary symmetric monoidal 2-category $2\mathcal{V}$. In order to provide a certain degree of concreteness, and in view of the application given in Section \ref{sec:clifford-fock}, at each step we will specialize the statements to the case where $2\mathcal{V}$ is either the Morita 2-category $2\Vect_\K$ of algebras, bimodules and intertwiners, or its super-analogue $2\mathrm{s}\Vect_\K$, i.e., the Morita 2-category of superalgebras, superbimodules and intertwiners. The terminology for 2-representations and projective 2-representations adopted in this Section should be self-explanatory. Additional details can be found in \cite[Chapter 3]{vuppulury}

\bigskip

We begin with the immediate generalization to a 2-categorical context of the initial definitions and constructions from Section \ref{sec:invertible-1d}. The proofs of the main results will then be a 2-dimensional version of the proofs in Section \ref{sec:invertible-1d}, that can actually be seen as ``slices'' of the 2-dimensional proofs.
\begin{definition}
    Let $\mathcal{C}$ be a category and $\mathcal{D}$ be a 2-category. A wannabe functor $\rho\colon \mathcal{C}\dashrightarrow \mathcal{D}$ is a map of simplicial sets $\rho\colon\mathrm{sk}_1(\mathcal{C})\to\mathrm{sk}_1(\mathcal{D})$, where $\mathrm{sk}_1$ is the 1-skeleton of a simplicial set.
\end{definition}
As in the 1-categorical case,  a wannabe functor $\rho\colon\mathcal{C}\dashrightarrow \mathcal{D}$ explicitly consists of the following associations:
\begin{itemize}
    \item an object $A_{X_i}$ in $\mathcal{D}$ for any object $X_i$ in $\mathcal{C}$    
    \item a morphism $M_{f_{ij}}$ in $\mathrm{Hom}_{\mathcal{D}}(A_{X_i},A_{X_j})$ for any morphism $f_{ij}\colon X_i\to X_j$ in $\mathcal{C}$,
\end{itemize}
such that $M_{\mathrm{id}_{X_i}}=\mathrm{id}_{A_{X_i}}$ for any object $X_i$ in $\mathcal{C}$. Since $\mathcal{D}$ is a 2-category, a wannabe functor only misses the datum of 2-simplices 
\[
    \begin{tikzcd}
        A_{X_i}
        \arrow[rr, "M_{f_{ik}}"]\arrow[dr,"M_{f_{ij}}"']&
        {}&
        A_{X_k}
        \\
        {}&
        A_{X_j}\arrow[u,shorten <=3pt,shorten > =1pt,Rightarrow,"\lambda_{\Xi_{ijk}}"',pos=.5]\ar[ur,"M_{f_{jk}}"']&
        {}
        \\
    \end{tikzcd}
\]       
for any 2-simplex
\[
            \begin{tikzcd}
            X_i
                \arrow[rr, "f_{ik}"]\arrow[dr,"f_{ij}"']&{}&
                X_k\\
                {}&X_j\arrow[u,shorten <=3pt,shorten > =1pt,Rightarrow,"\Xi_{ijk}"',pos=.5]\ar[ur,"f_{jk}"']&{}\\
            \end{tikzcd}
        \]       
in $\mathcal{C}$, such that we have a 3-simplex
\[
\begin{tikzcd}[row sep={7.2em,between origins}, column sep={10em,between origins}, ampersand replacement=\&]
    \&
    A_{X_{k}}
    \arrow[rdd, "M_{f_{kl}}", bend left=20]
    \&\\\&
    A_{X_{j}}
    \arrow[u, "M_{f_{jk}}"{description}]
    \arrow[rd, "M_{f_{jl}}"'{description,name=f13}]
    \&\\
    A_{X_{i}}
    \arrow[ruu, "M_{f_{ik}}"{name=f02}, bend left=20]
    \arrow[rr, "M_{f_{il}}"'{name=f03}, bend right=20]
    \arrow[ru, "M_{f_{ij}}"{description}]
    \&\&
    A_{X_{l}}
    \arrow[from=2-2, to=f02, Rightarrow, shorten=1.0em, "\lambda_{\Xi_{ijk}}"description, pos=0.475]
    \arrow[from=1-2, to=f13, Rightarrow, yshift=-0.25em, xshift=0.75em, shorten=3.5em, "\lambda_{\Xi_{jkl}}"description, pos=0.475]
    \arrow[from=2-2, to=f03, Rightarrow, shorten=2.5em, "\lambda_{\Xi_{ijl}}"description, pos=0.475]
    \arrow[from=1-2, to=f03, bend left=23.5, yshift=+0.0em, xshift=0.0em, crossing over, Rightarrow, shorten=2.0em, "\hspace{+0.625em}\lambda_{\Xi_{ikl}}"description, pos=0.425, crossing over clearance=1.5ex]
\end{tikzcd}
\]
in $\mathcal{D}$ for any 3-simplex $\Upsilon_{ijkl}$ in $\mathcal{C}$,
to be an actual functor.
As in the 1-categorical case, this is quite a strict condition, so a general wannabe functor has no hope to be an actual functor. But again, as in the 1-categorical case, under invertibility assumptions, any wannabe functor canonically extends to a projective representation. More precisely, let $2\mathcal{V}$ be a symmetric monoidal 2-category, and let $\mathcal{V}=\Omega2\mathcal{V}$. As in the 1-categorical case, a wannabe functor $\rho\colon \mathcal{C}\dashrightarrow \mathcal{V}$ will be called invertible if it factors through the 1-skeleton $\mathrm{sk}_1(\Pic(2\mathcal{V}))$ of the 3-group $\Pic(2\mathcal{V})$, i.e., if 
    \begin{itemize}
     \item the object $A_{X_i}$ is invertible in $(2\mathcal{V},\otimes)$ for any $X_i$ in $\mathcal{C}$;
     \item the morphism $M_{f_{ij}}\colon A_{x_i}\to A_{X_j}$ is invertible for every $f_{ij}\colon X_i\to X_j$ in $\mathcal{C}$.
    \end{itemize}
Then we have the following result, whose proof{, informed of the use of associahedra in homotopy associative algebras \cite{stasheff},} will take the whole Section.
\begin{proposition}\label{prop:wannabe-is-projective}
    Let $2\mathcal{V}$ be a symmetric monoidal 2-category, and let $\rho \colon \mathcal{C}\dashrightarrow {2}\mathcal{V}$ be an invertible wannabe functor. Then $\rho$ canonically extends to a projective representation $\rho\colon \mathcal{C}\to 2\mathcal{V}/\!/\mathbf{B}\Pic(\mathcal{V})$, where $\mathcal{V}=\Omega2\mathcal{V}=\mathrm{End}(\mathbf{1}_{2\mathcal{V}})$. More precisely, if we set
\[
l_{\Xi_{ijk}}=\tr(M_{f_{ik}}^{-1}\circ M_{f_{jk}}\circ M_{f_{ij}})
\]
for any 2-simplex in $\mathcal{C}$, then $l$ is a 2-cocycle on $\mathcal{C}$ with values in $\Pic(\mathcal{V})$ and $\rho$ is a  projective representation with 2-cocycle $l$.
\end{proposition}
\begin{remark}
When $2\mathcal{V}=2\Vect_\K$, we have 
\begin{align*}
\tr(M_{f_{ik}}^{-1}\circ M_{f_{jk}}\circ M_{f_{ij}})&=
\tr(M_{f_{ik}}^{-1}\otimes_{A_{X_k}} M_{f_{jk}}\otimes_{A_{X_j}} M_{f_{ij}})\\
&=\mathrm{Hom}_{(A_{X_i},A_{X_i})\text{-Mod}}(A_{X_i},M_{f_{ik}}^{-1}\otimes_{A_{X_k}} M_{f_{jk}}\otimes_{A_{X_j}} M_{f_{ij}})\\
&=
\mathrm{Hom}_{(A_{X_k},A_{X_i})\text{-Mod}}(M_{f_{ik}}, M_{f_{jk}}\otimes_{A_{X_j}} M_{f_{ij}})
\end{align*}
The same consideration applies when $2\mathcal{V}=2\mathrm{s}\Vect_K$.
\end{remark}
We begin by making the following graphical associations:

\[
A_{X_i}\rightsquigarrow \raisebox{-.3cm}{
\begin{tikzpicture}
\fill[red!40!white] (0,0) rectangle (.8,.8);
\end{tikzpicture} }\quad;\quad 
  A_{X_j}\rightsquigarrow \raisebox{-.3cm}{
\begin{tikzpicture}
\fill[blue!40!white] (0,0) rectangle (.8,.8);
\end{tikzpicture} }\quad;\quad 
  A_{X_k}\rightsquigarrow \raisebox{-.3cm}{
\begin{tikzpicture}
\fill[green!40!white] (0,0) rectangle (.8,.8);
\end{tikzpicture} }
  \]
  \[
M_{f_{ij}}\rightsquigarrow \raisebox{-.3cm}{
\begin{tikzpicture}
\fill[blue!40!white] (0,0) rectangle (0.4,.8);
\fill[red!40!white] (0.4,0) rectangle (0.8,.8);
\draw[black,thick,decoration={markings, mark=at position 0.5 with {\arrow{>}}},
        postaction={decorate}]  (0.4,0) -- (0.4,0.8);
\end{tikzpicture} }\quad;\quad 
  M_{f_{jk}}\rightsquigarrow \raisebox{-.3cm}{
\begin{tikzpicture}
\fill[green!40!white] (0,0) rectangle (0.4,.8);
\fill[blue!40!white] (0.4,0) rectangle (0.8,.8);
\draw[orange,thick,decoration={markings, mark=at position 0.5 with {\arrow{>}}},
        postaction={decorate}]  (0.4,0) -- (0.4,0.8);
\end{tikzpicture} }\quad;\quad 
  M_{f_{ik}}\rightsquigarrow \raisebox{-.3cm}{
\begin{tikzpicture}
\fill[green!40!white] (0,0) rectangle (0.4,.8);
\fill[red!40!white] (0.4,0) rectangle (0.8,.8);
\draw[black,thick,decoration={markings, mark=at position 0.5 with {\arrow{>}}},
        postaction={decorate}]  (0.4,0) -- (0.4,0.8);
\end{tikzpicture} }
  \]
    \[
M_{f_{ij}}^{-1}\rightsquigarrow \raisebox{-.3cm}{
\begin{tikzpicture}
\fill[red!40!white] (0,0) rectangle (0.4,.8);
\fill[blue!40!white] (0.4,0) rectangle (0.8,.8);
\draw[black,thick,decoration={markings, mark=at position 0.5 with {\arrow{<}}},
        postaction={decorate}]  (0.4,0) -- (0.4,0.8);
\end{tikzpicture} }\quad;\quad 
  M_{f_{jk}}^{-1}\rightsquigarrow \raisebox{-.3cm}{
\begin{tikzpicture}
\fill[blue!40!white] (0,0) rectangle (0.4,.8);
\fill[green!40!white] (0.4,0) rectangle (0.8,.8);
\draw[orange,thick,decoration={markings, mark=at position 0.5 with {\arrow{<}}},
        postaction={decorate}]  (0.4,0) -- (0.4,0.8);
\end{tikzpicture} }\quad;\quad 
  M_{f_{ik}}^{-1}\rightsquigarrow \raisebox{-.3cm}{
\begin{tikzpicture}
\fill[red!40!white] (0,0) rectangle (0.4,.8);
\fill[green!40!white] (0.4,0) rectangle (0.8,.8);
\draw[black,thick,decoration={markings, mark=at position 0.5 with {\arrow{<}}},
        postaction={decorate}]  (0.4,0) -- (0.4,0.8);
\end{tikzpicture} }
  \]
\begin{remark}
When $2\mathcal{V}=2\Vect_\K$ or $2\mathcal{V}=2\mathrm{s}\Vect_\K$ we have the following simple mnemonic trick for the direction of the arrow on the defect lines above: diagrams are read right to left and bottom to top; the defect line corresponding to $M_{f_{ij}}$ is up-going and has the colour $A_{X_i}$ on the right and the colour $A_{X_j}$ on the left: $M_{f_{ij}}$ is an $(A_{X_j},A_{X_i})$-module. A {planar $180$ degrees} rotation of the graphical element corresponds to taking the dual and so, in the invertible case, the inverse.
\end{remark}
The element $l_{\Xi_{ijk}}$ of $\mathcal{V}=\mathrm{End}(\mathbf{1}_{2\mathcal{V}})$ is then 
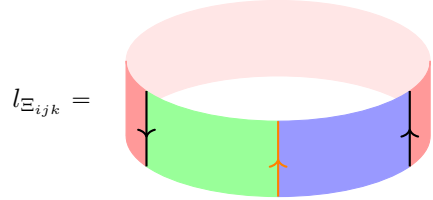
\[
l_{\Xi_{ijk}}=\quad\raisebox{-1.2cm}{\begin{tikzpicture}
\filldraw[fill=red!10!white, draw=red!10!white] (-2,0) arc
    [
        start angle=180,
        end angle=0,
        x radius=2,
        y radius =.8
    ]    -- ({2*cos(0)},{.8*sin(0)+1}) 
    arc
    [
        start angle=0,
        end angle=180,
        x radius=2,
        y radius =.8
    ]
     --cycle;
    \filldraw[fill=red!40!white, draw=red!40!white] (-2,0) arc
    [
        start angle=180,
        end angle=210,
        x radius=2,
        y radius =.8
    ]    -- ({2*cos(210)},{.8*sin(210)+1}) 
    arc
    [
        start angle=210,
        end angle=180,
        x radius=2,
        y radius =.8
    ]
     --cycle;
     \filldraw[fill=green!40!white, draw=green!40!white] ({2*cos(210)},{.8*sin(210)}) arc
    [
        start angle=210,
        end angle=270,
        x radius=2,
        y radius =.8
    ]    -- ({2*cos(270)},{.8*sin(270)+1}) 
    arc
    [
        start angle=270,
        end angle=210,
        x radius=2,
        y radius =.8
    ]
     --cycle;
     \filldraw[fill=blue!40!white, draw=blue!40!white] ({2*cos(270)},{.8*sin(270)}) arc
    [
        start angle=270,
        end angle=330,
        x radius=2,
        y radius =.8
    ]    -- ({2*cos(330)},{.8*sin(330)+1}) 
    arc
    [
        start angle=330,
        end angle=270,
        x radius=2,
        y radius =.8
    ]
     --cycle;
     \filldraw[fill=red!40!white, draw=red!40!white] ({2*cos(330)},{.8*sin(330)}) arc
    [
        start angle=330,
        end angle=360,
        x radius=2,
        y radius =.8
    ]    -- ({2*cos(360)},{.8*sin(360)+1}) 
    arc
    [
        start angle=360,
        end angle=330,
        x radius=2,
        y radius =.8
    ]
    --cycle;
    \draw[black,thick,decoration={markings, mark=at position 0.5 with {\arrow{<}}},
        postaction={decorate}]  ({2*cos(210)},{.8*sin(210)})  -- ({2*cos(210)},{.8*sin(210)+1}) ;
    \draw[orange,thick,decoration={markings, mark=at position 0.5 with {\arrow{>}}},
        postaction={decorate}]  ({2*cos(270)},{.8*sin(270)}) -- ({2*cos(270)},{.8*sin(270)+1}) ;
    \draw[black,thick,decoration={markings, mark=at position 0.5 with {\arrow{>}}},
        postaction={decorate}]  ({2*cos(330)},{.8*sin(330)}) -- ({2*cos(330)},{.8*sin(330)+1}) ;    
\end{tikzpicture}
}
\]
\begin{lemma}\label{lem:is-invertible2}
The element   $l_{\Xi_{ijk}}$ is an element in  $\Pic(\mathcal{V})$, i.e., it is an invertible object in the symmetric monoidal category $\mathcal{V}$.
\end{lemma}
\begin{proof}
Identical to the proof of Lemma \ref{lem:is-invertible}.    
\end{proof}
\begin{lemma}
 For any two adjacent 2-simplices    
 \[
            \begin{tikzcd}
            X_i
                \arrow[rr, "f_{ik}"]\arrow[dr,"f_{ij}"']&{}&
                X_k\\
                {}&X_j\arrow[u,shorten <=3pt,shorten > =1pt,Rightarrow,"\Xi_{ijk}"',pos=.5]\ar[ur,"f_{jk}"']&{}\\
            \end{tikzcd}
            \qquad \text{and}\qquad
            \begin{tikzcd}
            X_i
                \arrow[rr, "f_{il}"]\arrow[dr,"f_{ik}"']&{}&
                X_{l}\\
                {}&X_{k}\arrow[u,shorten <=3pt,shorten > =1pt,Rightarrow,"\Xi_{ikl}"',pos=.5]\ar[ur,"f_{kl}"']&{}\\
            \end{tikzcd}
        \]   
in $\mathcal{C}$, we have a distinguished morphism
\[
\tau_{\Xi_{ijk},\Xi_{ikl}}\colon \tr(M_{f_{il}}^{-1}\circ M_{f_{kl}}\circ M_{f_{ik}})\otimes \tr(M_{f_{ik}^{-1}}\circ M_{f_{jk}}\circ M_{f_{ij}})\to \tr(M_{f_{il}}^{-1}\circ  M_{f_{kl}}\circ M_{f_{jk}}\circ M_{f_{ij}})
\]
in $\mathcal{V}$. For any two adjacent 2-simplices    
 \[
            \begin{tikzcd}
            X_j
                \arrow[rr, "f_{jl}"]\arrow[dr,"f_{jk}"']&{}&
                X_l\\
                {}&X_k\arrow[u,shorten <=3pt,shorten > =1pt,Rightarrow,"\Xi_{jkl}"',pos=.5]\ar[ur,"f_{kl}"']&{}\\
            \end{tikzcd}
            \qquad \text{and}\qquad
            \begin{tikzcd}
            X_i
                \arrow[rr, "f_{il}"]\arrow[dr,"f_{ij}"']&{}&
                X_{l}\\
                {}&X_j\arrow[u,shorten <=3pt,shorten > =1pt,Rightarrow,"\Xi_{ijl}"',pos=.5]\ar[ur,"f_{jl}"']&{}\\
            \end{tikzcd}
        \]   
in $\mathcal{C}$, we have a distinguished morphism
\[
\tau_{\Xi_{jkl},\Xi_{ijl}}\colon \tr(M_{f_{jl}}^{-1}\circ M_{f_{kl}}\circ M_{f_{jk}})\otimes \tr(M_{f_{il}^{-1}}\circ M_{f_{jl}}\circ M_{f_{ij}})\to \tr(M_{f_{il}}^{-1}\circ  M_{f_{kl}}\circ M_{f_{jk}}\circ M_{f_{ij}})
\]
in $\mathcal{V}$.
\end{lemma}
\begin{proof}
Let us add to our graphical assignments the following ones: 
 \[
 A_{X_l}\rightsquigarrow \raisebox{-.3cm}{
\begin{tikzpicture}
\fill[orange!40!white] (0,0) rectangle (.8,.8);
\end{tikzpicture} }
  \quad;\quad
  M_{f_{kl}}\rightsquigarrow \raisebox{-.3cm}{
\begin{tikzpicture}
\fill[orange!40!white] (0,0) rectangle (0.4,.8);
\fill[green!40!white] (0.4,0) rectangle (0.8,.8);
\draw[lightgray,thick,decoration={markings, mark=at position 0.5 with {\arrow{>}}},
        postaction={decorate}]  (0.4,0) -- (0.4,0.8);
\end{tikzpicture} }\quad;\quad 
  M_{f_{kl}}^{-1}\rightsquigarrow \raisebox{-.3cm}{
\begin{tikzpicture}
\fill[green!40!white] (0,0) rectangle (0.4,.8);
\fill[orange!40!white] (0.4,0) rectangle (0.8,.8);
\draw[lightgray,thick,decoration={markings, mark=at position 0.5 with {\arrow{<}}},
        postaction={decorate}]  (0.4,0) -- (0.4,0.8);
\end{tikzpicture} }
  \quad;\quad
  M_{f_{il}}\rightsquigarrow \raisebox{-.3cm}{
\begin{tikzpicture}
\fill[orange!40!white] (0,0) rectangle (0.4,.8);
\fill[red!40!white] (0.4,0) rectangle (0.8,.8);
\draw[green,thick,decoration={markings, mark=at position 0.5 with {\arrow{>}}},
        postaction={decorate}]  (0.4,0) -- (0.4,0.8);
\end{tikzpicture} }\quad;\quad 
  M_{f_{il}}^{-1}\rightsquigarrow \raisebox{-.3cm}{
\begin{tikzpicture}
\fill[red!40!white] (0,0) rectangle (0.4,.8);
\fill[orange!40!white] (0.4,0) rectangle (0.8,.8);
\draw[green,thick,decoration={markings, mark=at position 0.5 with {\arrow{<}}},
        postaction={decorate}]  (0.4,0) -- (0.4,0.8);
\end{tikzpicture} }\quad;
  \]
  \[
   M_{f_{jl}}\rightsquigarrow \raisebox{-.3cm}{
\begin{tikzpicture}
\fill[orange!40!white] (0,0) rectangle (0.4,.8);
\fill[blue!40!white] (0.4,0) rectangle (0.8,.8);
\draw[brown,thick,decoration={markings, mark=at position 0.5 with {\arrow{>}}},
        postaction={decorate}]  (0.4,0) -- (0.4,0.8);
\end{tikzpicture} }\quad;\quad 
  M_{f_{jl}}^{-1}\rightsquigarrow \raisebox{-.3cm}{
\begin{tikzpicture}
\fill[blue!40!white] (0,0) rectangle (0.4,.8);
\fill[orange!40!white] (0.4,0) rectangle (0.8,.8);
\draw[brown,thick,decoration={markings, mark=at position 0.5 with {\arrow{<}}},
        postaction={decorate}]  (0.4,0) -- (0.4,0.8);
\end{tikzpicture} }\quad.
  \]
Then we have the distinguished morphism $\tau_{\Xi_{ijk},\Xi_{ikl}}$ given by the following 2-dimensional bordism with defects:
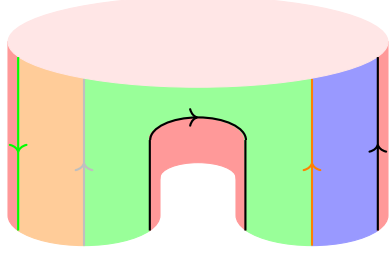
\[
\begin{tikzpicture}
\filldraw[fill=red!40!white, draw=red!40!white] (-3,0) arc
    [
        start angle=180,
        end angle=210,
        x radius=1,
        y radius =.4
    ]    -- ({cos(210)-2},{.4*sin(210)+2.4}) 
    -- (-3,2.3)
     --cycle;
     \filldraw[fill=orange!40!white, draw=orange!40!white] ({cos(210)-2},{.4*sin(210)}) arc
    [
        start angle=210,
        end angle=270,
        x radius=1,
        y radius =.4
    ]    -- ({cos(270)-2},{.4*sin(270)+2.4}) --  ({cos(210)-2},{.4*sin(210)+2.4})
     --cycle;
     \filldraw[fill=green!40!white, draw=green!40!white] ({cos(270)-2},{.4*sin(270)}) arc
    [
        start angle=270,
        end angle=330,
        x radius=1,
        y radius =.4
    ]    -- ({cos(330)-2},{.4*sin(330)+2.4}) -- ({cos(270)-2},{.4*sin(270)+2.4})
     --cycle;
     \filldraw[fill=red!40!white, draw=red!40!white] ({cos(330)-2},{.4*sin(330)}) arc
    [
        start angle=330,
        end angle=360,
        x radius=1,
        y radius =.4
    ]    -- ({cos(360)-2},{.4*sin(360)+2.4}) -- ({cos(330)-2},{.4*sin(330)+2.4})
    --cycle;
    \filldraw[fill=red!40!white, draw=red!40!white] (0,0) arc
    [
        start angle=180,
        end angle=210,
        x radius=1,
        y radius =.4
    ]    -- ({cos(210)+1},{.4*sin(210)+2.4}) -- ({cos(180)+1},{.4*sin(180)+2.4})
     --cycle;
     \filldraw[fill=green!40!white, draw=green!40!white] ({cos(210)+1},{.4*sin(210)}) arc
    [
        start angle=210,
        end angle=270,
        x radius=1,
        y radius =.4
    ]    -- ({cos(270)+1},{.4*sin(270)+2.4}) -- ({cos(210)+.95},{.4*sin(210)+2.4})
     --cycle;
     \filldraw[fill=blue!40!white, draw=blue!40!white] ({cos(270)+1},{.4*sin(270)}) arc
    [
        start angle=270,
        end angle=330,
        x radius=1,
        y radius =.4
    ]    -- ({cos(330)+1},{.4*sin(330)+2.4}) -- ({cos(270)+1},{.4*sin(270)+2.4})
     --cycle;
     \filldraw[fill=red!40!white, draw=red!40!white] ({cos(330)+1},{.4*sin(330)}) arc
    [
        start angle=330,
        end angle=360,
        x radius=1,
        y radius =.4
    ]    -- ({cos(360)+1},{.4*sin(360)+2.3}) -- ({cos(330)+1},{.4*sin(330)+2.4})
    --cycle;
   \filldraw[fill=red!40!white, draw=red!40!white] ({cos(330)-2},0) -- (-1,0.5) arc
    [
        start angle=180,
        end angle=0,
        x radius=0.5,
        y radius =.2
    ]    -- ({cos(210)+.99},0) -- ({cos(210)+.99},2.4) --  ({cos(330)-2},2.4)
    --cycle;
  \filldraw[fill=green!40!white, draw=green!40!white] ({cos(330)-2.025},1)  arc
    [
        start angle=180,
        end angle=0,
        x radius=0.637,
        y radius =.3
    ]    -- ({cos(210)+1},0) -- ({cos(210)+1},2.4) --  ({cos(330)-2},2.4)
    --cycle;  
    \filldraw[fill=red!10!white, draw=red!10!white] (2,2.3) arc
    [
        start angle=0,
        end angle=360,
        x radius=2.5,
        y radius =.6
    ];    
    \draw[green,thick,decoration={markings, mark=at position 0.5 with {\arrow{<}}},
        postaction={decorate}]  ({cos(210)-2},{.4*sin(210)})  -- ({cos(210)-2},{.4*sin(210)+2.3}) ;
    \draw[lightgray,thick,decoration={markings, mark=at position 0.5 with {\arrow{>}}},
        postaction={decorate}]  ({cos(270)-2},{.4*sin(270)}) -- ({cos(270)-2},{.4*sin(270)+2.21}) ;
    \draw[black,thick]  ({cos(330)-2},{.4*sin(330)}) -- ({cos(330)-2},{.4*sin(330)+1.2}) ;   
    \draw[black,thick,decoration={markings, mark=at position 0.5 with {\arrow{>}}},
        postaction={decorate}] ({cos(330)-1.995},1)  arc
    [
        start angle=180,
        end angle=0,
        x radius=0.630,
        y radius =.3
    ] ; 
    \draw[black,thick]  ({cos(210)+.99},{.4*sin(210)})  -- ({cos(210)+.99},{.4*sin(210)+1.2}) ;
    \draw[orange,thick,decoration={markings, mark=at position 0.5 with {\arrow{>}}},
        postaction={decorate}]  ({cos(270)+1},{.4*sin(270)}) -- ({cos(270)+1},{.4*sin(270)+2.21}) ;
    \draw[black,thick,decoration={markings, mark=at position 0.5 with {\arrow{>}}},
        postaction={decorate}]  ({cos(330)+1},{.4*sin(330)}) -- ({cos(330)+1},{.4*sin(330)+2.3}) ;        
\end{tikzpicture}
\]
and the distinguished morphism $\tau_{\Xi_{jkl},\Xi_{ijl}}$ given by the following 2-dimensional bordism with defects:
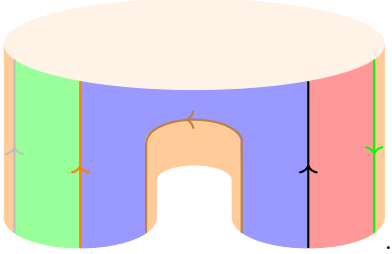
\[
\begin{tikzpicture}
\filldraw[fill=orange!40!white, draw=orange!40!white] (-3,0) arc
    [
        start angle=180,
        end angle=210,
        x radius=1,
        y radius =.4
    ]    -- ({cos(210)-2},{.4*sin(210)+2.4}) 
    -- (-3,2.3)
     --cycle;
     \filldraw[fill=green!40!white, draw=green!40!white] ({cos(210)-2},{.4*sin(210)}) arc
    [
        start angle=210,
        end angle=270,
        x radius=1,
        y radius =.4
    ]    -- ({cos(270)-2},{.4*sin(270)+2.4}) --  ({cos(210)-2},{.4*sin(210)+2.4})
     --cycle;
     \filldraw[fill=blue!40!white, draw=blue!40!white] ({cos(270)-2},{.4*sin(270)}) arc
    [
        start angle=270,
        end angle=330,
        x radius=1,
        y radius =.4
    ]    -- ({cos(330)-2},{.4*sin(330)+2.4}) -- ({cos(270)-2},{.4*sin(270)+2.4})
     --cycle;
     \filldraw[fill=orange!40!white, draw=orange!40!white] ({cos(330)-2},{.4*sin(330)}) arc
    [
        start angle=330,
        end angle=360,
        x radius=1,
        y radius =.4
    ]    -- ({cos(360)-2},{.4*sin(360)+2.4}) -- ({cos(330)-2},{.4*sin(330)+2.4})
    --cycle;
    \filldraw[fill=orange!40!white, draw=orange!40!white] (0,0) arc
    [
        start angle=180,
        end angle=210,
        x radius=1,
        y radius =.4
    ]    -- ({cos(210)+1},{.4*sin(210)+2.4}) -- ({cos(180)+1},{.4*sin(180)+2.4})
     --cycle;
     \filldraw[fill=blue!40!white, draw=blue!40!white] ({cos(210)+1},{.4*sin(210)}) arc
    [
        start angle=210,
        end angle=270,
        x radius=1,
        y radius =.4
    ]    -- ({cos(270)+1},{.4*sin(270)+2.4}) -- ({cos(210)+.95},{.4*sin(210)+2.4})
     --cycle;
     \filldraw[fill=red!40!white, draw=red!40!white] ({cos(270)+1},{.4*sin(270)}) arc
    [
        start angle=270,
        end angle=330,
        x radius=1,
        y radius =.4
    ]    -- ({cos(330)+1},{.4*sin(330)+2.4}) -- ({cos(270)+1},{.4*sin(270)+2.4})
     --cycle;
     \filldraw[fill=orange!40!white, draw=orange!40!white] ({cos(330)+1},{.4*sin(330)}) arc
    [
        start angle=330,
        end angle=360,
        x radius=1,
        y radius =.4
    ]    -- ({cos(360)+1},{.4*sin(360)+2.3}) -- ({cos(330)+1},{.4*sin(330)+2.4})
    --cycle;
   \filldraw[fill=orange!40!white, draw=orange!40!white] ({cos(330)-2},0) -- (-1,0.5) arc
    [
        start angle=180,
        end angle=0,
        x radius=0.5,
        y radius =.2
    ]    -- ({cos(210)+.99},0) -- ({cos(210)+.99},2.4) --  ({cos(330)-2},2.4)
    --cycle;
  \filldraw[fill=blue!40!white, draw=blue!40!white] ({cos(330)-2.025},1)  arc
    [
        start angle=180,
        end angle=0,
        x radius=0.637,
        y radius =.3
    ]    -- ({cos(210)+1},0) -- ({cos(210)+1},2.4) --  ({cos(330)-2},2.4)
    --cycle;  
    \filldraw[fill=orange!10!white, draw=orange!10!white] (2,2.3) arc
    [
        start angle=0,
        end angle=360,
        x radius=2.5,
        y radius =.6
    ];    
    \draw[lightgray,thick,decoration={markings, mark=at position 0.5 with {\arrow{>}}},
        postaction={decorate}]  ({cos(210)-2},{.4*sin(210)})  -- ({cos(210)-2},{.4*sin(210)+2.3}) ;
    \draw[orange,thick,decoration={markings, mark=at position 0.5 with {\arrow{>}}},
        postaction={decorate}]  ({cos(270)-2},{.4*sin(270)}) -- ({cos(270)-2},{.4*sin(270)+2.21}) ;
    \draw[brown,thick]  ({cos(330)-2},{.4*sin(330)}) -- ({cos(330)-2},{.4*sin(330)+1.2}) ;   
    \draw[brown,thick,decoration={markings, mark=at position 0.5 with {\arrow{<}}},
        postaction={decorate}] ({cos(330)-1.995},1)  arc
    [
        start angle=180,
        end angle=0,
        x radius=0.630,
        y radius =.3
    ] ; 
    \draw[brown,thick]  ({cos(210)+.99},{.4*sin(210)})  -- ({cos(210)+.99},{.4*sin(210)+1.2}) ;
    \draw[black,thick,decoration={markings, mark=at position 0.5 with {\arrow{>}}},
        postaction={decorate}]  ({cos(270)+1},{.4*sin(270)}) -- ({cos(270)+1},{.4*sin(270)+2.21}) ;
    \draw[green,thick,decoration={markings, mark=at position 0.5 with {\arrow{<}}},
        postaction={decorate}]  ({cos(330)+1},{.4*sin(330)}) -- ({cos(330)+1},{.4*sin(330)+2.3}) ;        
\end{tikzpicture}.
\]
\end{proof}
\begin{remark}
When $2\mathcal{V}=2\Vect_\K$ or $2\mathcal{V}=2\mathrm{s}\Vect_\K$, the morphisms $\tau_{\Xi_{ijk},\Xi_{ikl}}$ and $\tau_{\Xi_{jkl},\Xi_{ijl}}$ specialize to 
 the compositions
 \[
 \mathrm{Hom}(M_{f_{ik}},M_{f_{jk}}\circ M_{f_{ij}})\otimes 
\mathrm{Hom}(M_{f_{il}},M_{f_{kl}}\circ M_{f_{ik}}) \to \mathrm{Hom}(M_{f_{il}},M_{f_{kl}}\circ M_{f_{jk}}\circ M_{f_{ij}})
\]
and
\[
\mathrm{Hom}(M_{f_{jl}},M_{f_{kl}}\circ M_{f_{jk}})\otimes 
\mathrm{Hom}(M_{f_{il}},M_{f_{jl}}\circ M_{f_{ij}}) \to \mathrm{Hom}(M_{f_{il}},M_{f_{kl}}\circ M_{f_{jk}}\circ M_{f_{ij}}),
\]
respectively.
\end{remark}
\begin{remark}
In drawing the morphism   $\tau_{\Xi_{jkl},\Xi_{ijl}}$ we made implicit use of the cyclic invariance of the trace. A drawing literally matching the source and target of $\tau_{\Xi_{jkl},\Xi_{ijl}}$, without making use of the cyclic invariance of the trace is:
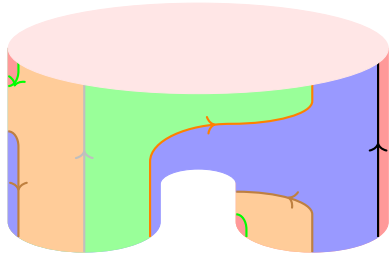
\[
\begin{tikzpicture}
\filldraw[fill=red!40!white, draw=red!40!white] 
 (0,0) arc
    [
        start angle=180,
        end angle=360,
        x radius=1,
        y radius =.4
    ]   -- (2,2.3) -- (0,2.3) --cycle;
\filldraw[fill=orange!40!white, draw=orange!40!white] 
 (0,0) arc
    [
        start angle=180,
        end angle=270,
        x radius=1,
        y radius =.4
    ]   -- (1,1) -- (0,1) --cycle;
 \filldraw[fill=blue!40!white, draw=blue!40!white] 
 (-3,0) arc
    [
        start angle=180,
        end angle=360,
        x radius=1,
        y radius =.4
    ]  --  (-1,0) -- (-1,0.5) arc
    [
        start angle=180,
        end angle=0,
        x radius=0.5,
        y radius =.2
    ]  
    -- (0,0.4) 
    arc
    [
        start angle=90,
        end angle=0,
        x radius=1,
        y radius =.3
    ] -- (1,-0.4)
    arc
    [
        start angle=270,
        end angle=330,
        x radius=1,
        y radius =.4
    ] -- ({cos(330)+1},2.3) -- (-3,2.3) --(-3,0)
    --cycle;    
 \filldraw[fill=green!40!white, draw=green!40!white] (-3,0) arc
    [
        start angle=180,
        end angle=330,
        x radius=1,
        y radius =.4
    ]    -- ({cos(330)-2},0.8)
    arc[
        start angle=180,
        end angle=90,
        x radius=((3-cos(330))/2),
        y radius =.5
    ]  
        arc[
      start angle=270,
        end angle=360,
        x radius=((3-cos(330))/2),
        y radius =.3
    ]  
    -- (1,2.3) -- (-3,2.3)
     --cycle;
 \filldraw[fill=orange!40!white, draw=orange!40!white] (-3,0) arc
    [
        start angle=180,
        end angle=270,
        x radius=1,
        y radius =.4
    ]    -- (-2,2.3)
    -- (-3,2.3)
     --cycle;
 \filldraw[fill=blue!40!white, draw=blue!40!white] (-3,0) arc
    [
        start angle=180,
        end angle=210,
        x radius=1,
        y radius =.4
    ]    -- ({cos(210)-2},1)
    arc
    [
        start angle=0,
        end angle=90,
        x radius=(cos(210)+1),
        y radius =.2
    ]  
     --cycle;
      \filldraw[fill=red!40!white, draw=red!40!white] (-3,2.3) -- (-3,1.8) arc
    [
        start angle=270,
        end angle=360,
        x radius=(cos(210)+1),
        y radius =.2
    ]    -- ({cos(210)-2},2.3)
     --cycle;
 \filldraw[fill=red!40!white, draw=red!40!white] (0,0) arc
    [
        start angle=180,
        end angle=210,
        x radius=1,
        y radius =.4
    ]    -- ({cos(210)+1},-0.1)
    arc
    [
        start angle=0,
        end angle=90,
        x radius=(cos(210)+1),
        y radius =.2
    ]  
     --cycle;  
\draw[black,thick,decoration={markings, mark=at position 0.5 with {\arrow{>}}},
        postaction={decorate}]  ({cos(330)+1},{0.4*sin(330)}) -- ({cos(330)+1},2.3)  
    ;       
\draw[green,thick]  ({cos(210)+1},{0.4*sin(210)})--({cos(210)+1},-0.1)
    arc
    [
        start angle=0,
        end angle=90,
        x radius=(cos(210)+1),
        y radius =.2
    ]  ;  
\draw[brown,thick,decoration={markings, mark=at position 0.5 with {\arrow{<}}},
        postaction={decorate}]  (0,0.4) 
    arc
    [
        start angle=90,
        end angle=0,
        x radius=1,
        y radius =.3
    ] -- (1,-0.4);  
\draw[orange,thick,decoration={markings, mark=at position 0.5 with {\arrow{>}}},
        postaction={decorate}]  ({cos(330)-2},{0.4*sin(330)}) -- ({cos(330)-2},.8)
    arc[
        start angle=180,
        end angle=90,
        x radius=((3-cos(330))/2),
        y radius =.5
    ]  
        arc[
      start angle=270,
        end angle=360,
        x radius=((3-cos(330))/2),
        y radius =.3
    ]  
    -- (1,2.3);
\draw[lightgray,thick,decoration={markings, mark=at position 0.5 with {\arrow{>}}},
        postaction={decorate}]  (-2,-0.4) -- (-2,2.3) ;     
\draw[brown,thick,decoration={markings, mark=at position 0.5 with {\arrow{<}}},
        postaction={decorate}]  ({cos(210)-2},{0.4*sin(210)}) -- ({cos(210)-2},1)
        arc
    [
        start angle=0,
        end angle=90,
        x radius=(cos(210)+1),
        y radius =.2
    ]  ;     
\draw[green,thick,decoration={markings, mark=at position 0.3 with {\arrow{<}}},
        postaction={decorate}]   (-3,1.8) arc
    [
        start angle=270,
        end angle=360,
        x radius=(cos(210)+1),
        y radius =.2
    ] -- ({cos(210)-2},2.3);  

 \filldraw[fill=red!10!white, draw=red!10!white] (2,2.3) arc
    [
        start angle=0,
        end angle=360,
        x radius=2.5,
        y radius =.6
    ];         
\end{tikzpicture}
\]
Clearly, the two drawings are related by a (unique up to isotopy) diffemorphism of the surfaces with defects realizing them. Notice also how in this second drawing it is manifest that $\tau_{\Xi_{ijk},\Xi_{ikl}}$ and $\tau_{\Xi_{jkl},\Xi_{ijl}}$ have the same target.
\end{remark}
\begin{lemma}\label{tau-isomorphism}
The morphisms    $\tau_{\Xi_{ijk},\Xi_{ikl}}$ and $\tau_{\Xi_{jkl},\Xi_{ijl}}$ are isomorphisms. 
\end{lemma}
\begin{proof}
We give a proof for $\tau_{\Xi_{ijk},\Xi_{ikl}}$; the proof for $\tau_{\Xi_{jkl},\Xi_{ijl}}$ is identical. Consider the morphism $\hat{\tau}_{\Xi_{ijk},\Xi_{ikl}}$ defined by the following surface with defect lines:
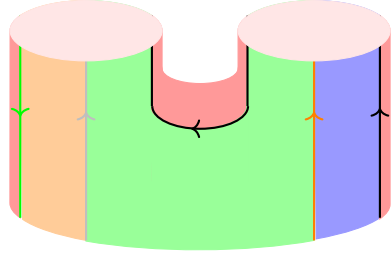
\[

\]
We have that $\hat{\tau}_{\Xi_{ijk},\Xi_{ikl}}$ is the inverse of ${\tau}_{\Xi_{ijk},\Xi_{ikl}}$. To see this we compute
\begin{align*}
{\tau}_{\Xi_{ijk},\Xi_{ikl}}\circ \hat{\tau}_{\Xi_{ijk},\Xi_{ikl}}&=\raisebox{-3cm}{
%
}
\quad=\mathrm{id}_{\mathrm{tr}(\tr(M_{f_{il}}^{-1}\circ  M_{f_{kl}}\circ M_{f_{jk}}\circ M_{f_{ij}}))}.
\end{align*}
Here we used the invertibility of \raisebox{-.3cm}{
%
 } in the first step and the invertibility of \raisebox{-.3cm}{
%
 } in the second step. Next we compute
\begin{align*}
\hat{\tau}_{\Xi_{ijk},\Xi_{ikl}}\circ {\tau}_{\Xi_{ijk},\Xi_{ikl}}&=\raisebox{-2.5cm}{
%
}
\quad=  \mathrm{id}_{\tr(M_{f_{il}}^{-1}\circ M_{f_{kl}}\circ M_{f_{ik}})}\otimes \mathrm{id}_{\tr(M_{f_{ik}^{-1}}\circ M_{f_{jk}}\circ M_{f_{ij}})}.
\end{align*}
Here we used the invertibility of \raisebox{-.3cm}{
%
 } in the first step and the invertibility of 
\raisebox{-.3cm}{
%
 } in the second step.
\end{proof}

\begin{remark}
Lemma \ref{tau-isomorphism} implies in particular that
\[
\tr(M_{f_{il}}^{-1}\circ  M_{f_{kl}}\circ M_{f_{jk}}\circ M_{f_{ij}})
\]
is an invertible object of $\mathcal{V}$. This could also have been verified directly with a proof analogous to that of Lemma \ref{lem:is-invertible2}.
\end{remark}
\begin{corollary}\label{cor:m}
Given a 3-simplex in $\mathcal{C}$ with 3-filler $\Upsilon_{ijkl}$, let us write $m_{\Upsilon_{ijkl}}$ for the invertible object 
\[
m_{\Upsilon_{ijkl}}=\tr(M_{f_{il}}^{-1}\circ  M_{f_{kl}}\circ M_{f_{jk}}\circ M_{f_{ij}})
\]
of $\mathcal{V}$. Also, let us write  $\tau^{\mathrm{odd}}_{\Upsilon_{ijkl}}$ and $\tau^{\mathrm{even}}_{\Upsilon_{ijkl}}$ for $\tau_{\Xi_{ijk},\Xi_{ikl}}$ and $\tau_{\Xi_{jkl},\Xi_{ijl}}$, respectively. Then we have isomorphisms
\[
l_{\Xi_{ikl}}\otimes l_{\Xi_{ijk}}\xrightarrow{\tau^{\mathrm{odd}}_{\Upsilon_{ijkl}}} m_{\Upsilon_{ijkl}}
\xleftarrow{\tau^{\mathrm{even}}_{\Upsilon_{ijkl}}}
l_{\Xi_{jkl}}\otimes l_{\Xi_{ijl}}
\]
\end{corollary}
\begin{definition}
 In the same notation as in Corollary \ref{cor:m} we write $\psi_{\Upsilon_{ijkl}}$ for the isomorphism
 \[
 \psi_{\Upsilon_{ijkl}}=( \tau^{\mathrm{odd}}_{\Upsilon_{ijkl}})^{-1}\circ \tau^{\mathrm{even}}_{\Upsilon_{ijkl}}\colon l_{\Xi_{jkl}}\otimes l_{\Xi_{ijl}}\xrightarrow{\sim}l_{\Xi_{ijk}}\otimes l_{\Xi_{ikl}}. 
 \]
 and $\phi_{\Upsilon_{ijkl}}$ for the inverse isomorphism
 \[
 \phi_{\Upsilon_{ijkl}}\colon l_{\Xi_{ijk}}\otimes l_{\Xi_{ikl}}\xrightarrow{\sim}l_{\Xi_{jkl}}\otimes l_{\Xi_{ijl}}.
 \]
 \end{definition}
\begin{remark}
By definition, $\psi_{\Upsilon_{ijkl}}$ participates in a commutative diagram of isomorphisms of invertible objects of $\mathcal{V}$:
\[
 \begin{tikzcd}
 l_{\Xi_{jkl}}\otimes l_{\Xi_{ijl}}\ar[rd,"\tau^{\mathrm{even}}_{\Upsilon_{ijkl}}"']\ar[rr,"\psi_{\Upsilon_{ijkl}}"]&&l_{\Xi_{ijk}}\otimes l_{\Xi_{ikl}}\ar[dl,"\tau^{\mathrm{odd}}_{\Upsilon_{ijkl}}"]\\
 {}&m_{\Upsilon_{ijkl}}
 \end{tikzcd}.
\]
For the pleasure of drawing it, we notice that the isomorphism $\phi_{\Upsilon_{ijkl}}$ is represented by the following surface with defects:
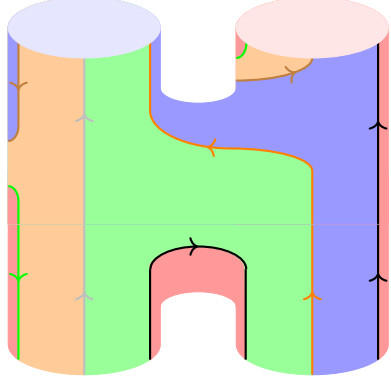
\[
\begin{tikzpicture}
\begin{scope}
\filldraw[fill=red!40!white, draw=red!40!white] (-3,0) arc
    [
        start angle=180,
        end angle=210,
        x radius=1,
        y radius =.4
    ]    -- ({cos(210)-2},{.4*sin(210)+2.4}) 
    -- (-3,2.3)
     --cycle;
     \filldraw[fill=orange!40!white, draw=orange!40!white] ({cos(210)-2},{.4*sin(210)}) arc
    [
        start angle=210,
        end angle=270,
        x radius=1,
        y radius =.4
    ]    -- ({cos(270)-2},{.4*sin(270)+2.4}) --  ({cos(210)-2},{.4*sin(210)+2.4})
     --cycle;
     \filldraw[fill=green!40!white, draw=green!40!white] ({cos(270)-2},{.4*sin(270)}) arc
    [
        start angle=270,
        end angle=330,
        x radius=1,
        y radius =.4
    ]    -- ({cos(330)-2},{.4*sin(330)+2.4}) -- ({cos(270)-2},{.4*sin(270)+2.4})
     --cycle;
     \filldraw[fill=red!40!white, draw=red!40!white] ({cos(330)-2},{.4*sin(330)}) arc
    [
        start angle=330,
        end angle=360,
        x radius=1,
        y radius =.4
    ]    -- ({cos(360)-2},{.4*sin(360)+2.4}) -- ({cos(330)-2},{.4*sin(330)+2.4})
    --cycle;
    \filldraw[fill=red!40!white, draw=red!40!white] (0,0) arc
    [
        start angle=180,
        end angle=210,
        x radius=1,
        y radius =.4
    ]    -- ({cos(210)+1},{.4*sin(210)+2.4}) -- ({cos(180)+1},{.4*sin(180)+2.4})
     --cycle;
     \filldraw[fill=green!40!white, draw=green!40!white] ({cos(210)+1},{.4*sin(210)}) arc
    [
        start angle=210,
        end angle=270,
        x radius=1,
        y radius =.4
    ]    -- ({cos(270)+1},{.4*sin(270)+2.4}) -- ({cos(210)+.95},{.4*sin(210)+2.4})
     --cycle;
     \filldraw[fill=blue!40!white, draw=blue!40!white] ({cos(270)+1},{.4*sin(270)}) arc
    [
        start angle=270,
        end angle=330,
        x radius=1,
        y radius =.4
    ]    -- ({cos(330)+1},{.4*sin(330)+2.4}) -- ({cos(270)+1},{.4*sin(270)+2.4})
     --cycle;
     \filldraw[fill=red!40!white, draw=red!40!white] ({cos(330)+1},{.4*sin(330)}) arc
    [
        start angle=330,
        end angle=360,
        x radius=1,
        y radius =.4
    ]    -- ({cos(360)+1},{.4*sin(360)+2.3}) -- ({cos(330)+1},{.4*sin(330)+2.4})
    --cycle;
   \filldraw[fill=red!40!white, draw=red!40!white] ({cos(330)-2},0) -- (-1,0.5) arc
    [
        start angle=180,
        end angle=0,
        x radius=0.5,
        y radius =.2
    ]    -- ({cos(210)+.99},0) -- ({cos(210)+.99},2.4) --  ({cos(330)-2},2.4)
    --cycle;
  \filldraw[fill=green!40!white, draw=green!40!white] ({cos(330)-2.025},1)  arc
    [
        start angle=180,
        end angle=0,
        x radius=0.637,
        y radius =.3
    ]    -- ({cos(210)+1},0) -- ({cos(210)+1},2.4) --  ({cos(330)-2},2.4)
    --cycle;  
    \filldraw[fill=red!10!white, draw=red!10!white] (2,2.3) arc
    [
        start angle=0,
        end angle=360,
        x radius=2.5,
        y radius =.6
    ];    
    \draw[green,thick,decoration={markings, mark=at position 0.5 with {\arrow{<}}},
        postaction={decorate}]  ({cos(210)-2},{.4*sin(210)})  -- ({cos(210)-2},{.4*sin(210)+2.3}) ;
    \draw[lightgray,thick,decoration={markings, mark=at position 0.5 with {\arrow{>}}},
        postaction={decorate}]  ({cos(270)-2},{.4*sin(270)}) -- ({cos(270)-2},{.4*sin(270)+2.21}) ;
    \draw[black,thick]  ({cos(330)-2},{.4*sin(330)}) -- ({cos(330)-2},{.4*sin(330)+1.2}) ;   
    \draw[black,thick,decoration={markings, mark=at position 0.5 with {\arrow{>}}},
        postaction={decorate}] ({cos(330)-1.995},1)  arc
    [
        start angle=180,
        end angle=0,
        x radius=0.630,
        y radius =.3
    ] ; 
    \draw[black,thick]  ({cos(210)+.99},{.4*sin(210)})  -- ({cos(210)+.99},{.4*sin(210)+1.2}) ;
    \draw[orange,thick,decoration={markings, mark=at position 0.5 with {\arrow{>}}},
        postaction={decorate}]  ({cos(270)+1},{.4*sin(270)}) -- ({cos(270)+1},{.4*sin(270)+2.21}) ;
    \draw[black,thick,decoration={markings, mark=at position 0.5 with {\arrow{>}}},
        postaction={decorate}]  ({cos(330)+1},{.4*sin(330)}) -- ({cos(330)+1},{.4*sin(330)+2.3}) ;   
\end{scope}
\begin{scope}[xscale=1,yscale=-1,yshift=-3.9cm]
\filldraw[fill=red!40!white, draw=red!40!white] 
 (0,-.3) arc
    [
        start angle=180,
        end angle=360,
        x radius=1,
        y radius =.4
    ]   -- (2,2.3) -- (0,2.3) --cycle;
\filldraw[fill=orange!40!white, draw=orange!40!white] 
 (0,-.3) arc
    [
        start angle=180,
        end angle=270,
        x radius=1,
        y radius =.4
    ]   -- (1,1) -- (0,1) --cycle;
 \filldraw[fill=blue!40!white, draw=blue!40!white] 
 (-3,-.3) arc
    [
        start angle=180,
        end angle=360,
        x radius=1,
        y radius =.4
    ]  --  (-1,0) -- (-1,0.5) arc
    [
        start angle=180,
        end angle=0,
        x radius=0.5,
        y radius =.2
    ]  
    -- (0,0.4) 
    arc
    [
        start angle=90,
        end angle=0,
        x radius=1,
        y radius =.3
    ] -- (1,-0.4)
    arc
    [
        start angle=270,
        end angle=330,
        x radius=1,
        y radius =.4
    ] -- ({cos(330)+1},2.3) -- (-3,2.3) --(-3,0)
    --cycle;    
 \filldraw[fill=green!40!white, draw=green!40!white] (-3,-.3) arc
    [
        start angle=180,
        end angle=330,
        x radius=1,
        y radius =.4
    ]    -- ({cos(330)-2},0.8)
    arc[
        start angle=180,
        end angle=90,
        x radius=((3-cos(330))/2),
        y radius =.5
    ]  
        arc[
      start angle=270,
        end angle=360,
        x radius=((3-cos(330))/2),
        y radius =.3
    ]  
    -- (1,2.3) -- (-3,2.3)
     --cycle;
 \filldraw[fill=orange!40!white, draw=orange!40!white] (-3,-.3) arc
    [
        start angle=180,
        end angle=270,
        x radius=1,
        y radius =.4
    ]    -- (-2,2.3)
    -- (-3,2.3)
     --cycle;
 \filldraw[fill=blue!40!white, draw=blue!40!white] (-3,-.3) arc
    [
        start angle=180,
        end angle=210,
        x radius=1,
        y radius =.4
    ]    -- ({cos(210)-2},1)
    arc
    [
        start angle=0,
        end angle=90,
        x radius=(cos(210)+1),
        y radius =.2
    ]  
     --cycle;
      \filldraw[fill=red!40!white, draw=red!40!white] (-3,2.3) -- (-3,1.8) arc
    [
        start angle=270,
        end angle=360,
        x radius=(cos(210)+1),
        y radius =.2
    ]    -- ({cos(210)-2},2.3)
     --cycle;
 \filldraw[fill=red!40!white, draw=red!40!white] (0,-.3) arc
    [
        start angle=180,
        end angle=210,
        x radius=1,
        y radius =.4
    ]    -- ({cos(210)+1},-0.1)
    arc
    [
        start angle=0,
        end angle=90,
        x radius=(cos(210)+1),
        y radius =.2
    ]  
     --cycle;  
\draw[black,thick,decoration={markings, mark=at position 0.5 with {\arrow{<}}},
        postaction={decorate}]  ({cos(330)+1},{0.4*sin(330)}) -- ({cos(330)+1},2.3)  
    ;       
\draw[green,thick]  ({cos(210)+1},{0.4*sin(210)})--({cos(210)+1},-0.1)
    arc
    [
        start angle=0,
        end angle=90,
        x radius=(cos(210)+1),
        y radius =.2
    ]  ;  
\draw[brown,thick,decoration={markings, mark=at position 0.5 with {\arrow{>}}},
        postaction={decorate}]  (0,0.4) 
    arc
    [
        start angle=90,
        end angle=0,
        x radius=1,
        y radius =.3
    ] -- (1,-0.4);  
\draw[orange,thick,decoration={markings, mark=at position 0.5 with {\arrow{<}}},
        postaction={decorate}]  ({cos(330)-2},{0.4*sin(330)}) -- ({cos(330)-2},.8)
    arc[
        start angle=180,
        end angle=90,
        x radius=((3-cos(330))/2),
        y radius =.5
    ]  
        arc[
      start angle=270,
        end angle=360,
        x radius=((3-cos(330))/2),
        y radius =.3
    ]  
    -- (1,2.3);
\draw[lightgray,thick,decoration={markings, mark=at position 0.5 with {\arrow{<}}},
        postaction={decorate}]  (-2,-0.4) -- (-2,2.3) ;     
\draw[brown,thick,decoration={markings, mark=at position 0.5 with {\arrow{>}}},
        postaction={decorate}]  ({cos(210)-2},{0.4*sin(210)}) -- ({cos(210)-2},1)
        arc
    [
        start angle=0,
        end angle=90,
        x radius=(cos(210)+1),
        y radius =.2
    ]  ;     
\draw[green,thick]   (-3,1.8) arc
    [
        start angle=270,
        end angle=360,
        x radius=(cos(210)+1),
        y radius =.2
    ] -- ({cos(210)-2},2.3);  
\end{scope}    
\filldraw[fill=blue!10!white, blue!10!white] (-1,4.2) arc
    [
        start angle=0,
        end angle=360,
        x radius=1,
        y radius =.4
    ] ;       
    \filldraw[fill=red!10!white, draw=red!10!white] (2,4.2) arc
    [
        start angle=0,
        end angle=360,
        x radius=1,
        y radius =.4
    ]  ;     
\end{tikzpicture}
\]

\end{remark}
We also have the following generalization of Lemma \ref{tau-isomorphism}.
\begin{lemma}\label{lem:n-invertible}
Given a 4-simplex in $\mathcal{C}$ with filler $\Sigma_{ijklm}$, let  $\Upsilon_{ijkl}$ be the fillers for its faces, $\Xi_{ijk}$ be the fillers of its 2-simplices, and $f_{ij}$ be its 1-simplices. Then
\[
n_{\Sigma_{ijklm}}=\mathrm{tr}(M_{f_{im}}^{-1}\circ M_{f_{lm}}\circ M_{f_{kl}}\circ M_{f_{jk}}\circ M_{f_{ij}}).
\]
is an invertible object of  $\mathcal{V}$ and we have distinguished isomorphisms
\[
\kappa_{\Upsilon_{ijkl}}^{\Xi_{ilm}}\colon m_{\Upsilon_{ijkl}}\otimes l_{\Xi_{ilm}}\xrightarrow{\sim} n_{\Sigma_{ijklm}}
\]
\[
\kappa_{\Upsilon_{iklm}}^{\Xi_{ijk}}\colon l_{\Xi_{ijk}}\otimes m_{\Upsilon_{iklm}}\xrightarrow{\sim} n_{\Sigma_{ijklm}}
\]
\end{lemma}
\begin{proof}
The proof of the invertibility of  $n_{\Sigma_{ijklm}}=\mathrm{tr}(M_{f_{im}}^{-1}\circ M_{f_{lm}}\circ M_{f_{kl}}\circ M_{f_{jk}}\circ M_{f_{ij}})$ is analogous to the proof of Lemma \ref{lem:is-invertible2}. The morphisms $\kappa_{\Upsilon_{ijkl}}^{\Xi_{ilm}}$ and $\kappa_{\Upsilon_{iklm}}^{\Xi_{ijk}}$ are defined by surfaces with defects analogous to those defining $\tau_{\Xi_{ijk},\Xi_{ikl}}$. The proof that $\kappa_{\Upsilon_{ijkl}}^{\Xi_{ilm}}$ and $\kappa_{\Upsilon_{iklm}}^{\Xi_{ijk}}$ are isomorphisms is analogous to the proof of Lemma \ref{tau-isomorphism}.
\end{proof}
\begin{lemma}\label{lem:two-kappas}
In the same notation as in Lemma \ref{lem:n-invertible} above, the diagram
\[
\begin{tikzcd}
    l_{\Xi_{ijk}}\otimes l_{\Xi_{ikl}}\otimes l_{\Xi_{ilm}}\ar[rrr,"\mathrm{id}\otimes \tau_{\Xi_{ikl},\Xi_{ilm}}" ]\ar[d,"\tau_{\Xi_{ijk},\Xi_{ikl}}\otimes\mathrm{id}"']&&&
    l_{\Xi_{ijk}}\otimes m_{\Upsilon_{iklm}}\ar[d,"\kappa_{\Upsilon_{iklm}}^{\Xi_{ijk}}"]\\
    m_{\Upsilon_{ijkl}}\otimes l_{\Xi_{ilm}}\ar[rrr,"\kappa_{\Upsilon_{ijkl}}^{\Xi_{ilm}}"'] &&& n_{\Sigma_{ijklm}}
\end{tikzcd} 
\]
commutes.
\end{lemma}
\begin{proof}
    The commutativity of the diagram is given by the following manifest diffeomorphism of surfaces with defect lines.
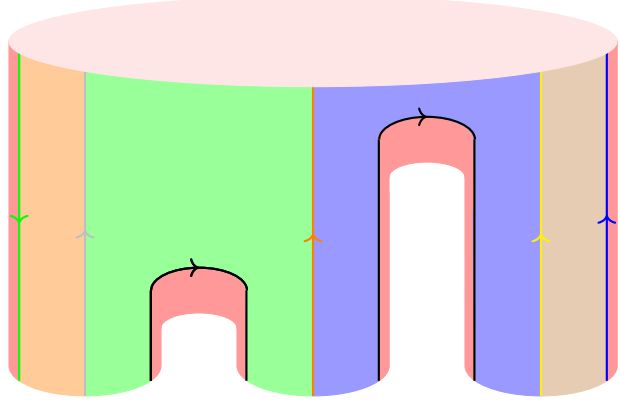
\[
\begin{tikzpicture}
\filldraw[fill=red!40!white, draw=red!40!white] (-3,0) arc
    [
        start angle=180,
        end angle=210,
        x radius=1,
        y radius =.4
    ]    -- ({cos(210)-2},{.4*sin(210)+4.4}) 
    -- (-3,4.3)
     --cycle;
     \filldraw[fill=orange!40!white, draw=orange!40!white] ({cos(210)-2},{.4*sin(210)}) arc
    [
        start angle=210,
        end angle=270,
        x radius=1,
        y radius =.4
    ]    -- ({cos(270)-2},{.4*sin(270)+4.4}) --  ({cos(210)-2},{.4*sin(210)+4.4})
     --cycle;
     \filldraw[fill=green!40!white, draw=green!40!white] ({cos(270)-2},{.4*sin(270)}) arc
    [
        start angle=270,
        end angle=330,
        x radius=1,
        y radius =.4
    ]    -- ({cos(330)-2},{.4*sin(330)+4.4}) -- ({cos(270)-2},{.4*sin(270)+4.4})
     --cycle;
     \filldraw[fill=red!40!white, draw=red!40!white] ({cos(330)-2},{.4*sin(330)}) arc
    [
        start angle=330,
        end angle=360,
        x radius=1,
        y radius =.4
    ]    -- ({cos(360)-2},{.4*sin(360)+2.4}) -- ({cos(330)-2},{.4*sin(330)+2.4})
    --cycle;
    \filldraw[fill=red!40!white, draw=red!40!white] (0,0) arc
    [
        start angle=180,
        end angle=210,
        x radius=1,
        y radius =.4
    ]    -- ({cos(210)+1},{.4*sin(210)+2.4}) -- ({cos(180)+1},{.4*sin(180)+2.4})
     --cycle;
     \filldraw[fill=green!40!white, draw=green!40!white] ({cos(210)+1},{.4*sin(210)}) arc
    [
        start angle=210,
        end angle=270,
        x radius=1,
        y radius =.4
    ]    -- ({cos(270)+1},{.4*sin(270)+4.4}) -- ({cos(210)+.95},{.4*sin(210)+4.4})
     --cycle;
     \filldraw[fill=blue!40!white, draw=blue!40!white] ({cos(270)+1},{.4*sin(270)}) arc
    [
        start angle=270,
        end angle=330,
        x radius=1,
        y radius =.4
    ]    -- ({cos(330)+1},{.4*sin(330)+4.4}) -- ({cos(270)+1},{.4*sin(270)+4.4})
     --cycle;
     \filldraw[fill=red!40!white, draw=red!40!white] ({cos(330)+1},{.4*sin(330)}) arc
    [
        start angle=330,
        end angle=360,
        x radius=1,
        y radius =.4
    ]    -- ({cos(360)+1},{.4*sin(360)+2.3}) -- ({cos(330)+1},{.4*sin(330)+2.4})
    --cycle;
   \filldraw[fill=red!40!white, draw=red!40!white] ({cos(330)-2},0) -- (-1,0.5) arc
    [
        start angle=180,
        end angle=0,
        x radius=0.5,
        y radius =.2
    ]    -- ({cos(210)+.99},0) -- ({cos(210)+.99},2.4) --  ({cos(330)-2},2.4)
    --cycle;
  \filldraw[fill=green!40!white, draw=green!40!white] ({cos(330)-2.025},1)  arc
    [
        start angle=180,
        end angle=0,
        x radius=0.637,
        y radius =.3
    ]    -- ({cos(210)+1},0) -- ({cos(210)+1},4.4) --  ({cos(330)-2},4.4)
    --cycle;  
   
    \filldraw[fill=red!10!white, draw=red!40!white] ({cos(330)+4},{.4*sin(330)}) arc
    [
        start angle=330,
        end angle=360,
        x radius=1,
        y radius =.4
    ]    -- ({cos(360)+4},{.4*sin(360)+4.3}) -- ({cos(330)+4},{.4*sin(330)+4.4})
    --cycle;

         \filldraw[fill=red!40!white, draw=red!40!white] ({cos(330)+1},{.4*sin(330)}) arc
    [
        start angle=330,
        end angle=360,
        x radius=1,
        y radius =.4
    ]    -- ({cos(360)+1},{.4*sin(360)+3.4}) -- ({cos(330)+1},{.4*sin(330)+3.4})
    --cycle;
    \filldraw[fill=red!40!white, draw=red!40!white] (3,0) arc
    [
        start angle=180,
        end angle=210,
        x radius=1,
        y radius =.4
    ]    -- ({cos(210)+4},{.4*sin(210)+3.4}) -- ({cos(180)+4},{.4*sin(180)+3.4})
     --cycle;
     \filldraw[fill=blue!40!white, draw=blue!40!white] ({cos(210)+4},{.4*sin(210)}) arc
    [
        start angle=210,
        end angle=270,
        x radius=1,
        y radius =.4
    ]    -- ({cos(270)+4},{.4*sin(270)+4.4}) -- ({cos(210)+3.95},{.4*sin(210)+4.4})
     --cycle;
     \filldraw[fill=brown!40!white, draw=brown!40!white] ({cos(270)+4},{.4*sin(270)}) arc
    [
        start angle=270,
        end angle=330,
        x radius=1,
        y radius =.4
    ]    -- ({cos(330)+4},{.4*sin(330)+4.4}) -- ({cos(270)+4},{.4*sin(270)+4.4})
     --cycle;
     \filldraw[fill=red!40!white, draw=red!40!white] ({cos(330)+4},{.4*sin(330)}) arc
    [
        start angle=330,
        end angle=360,
        x radius=1,
        y radius =.4
    ]    -- ({cos(360)+4},{.4*sin(360)+4.3}) -- ({cos(330)+4},{.4*sin(330)+4.4})
    --cycle;
   \filldraw[fill=red!40!white, draw=red!40!white] ({cos(330)+1},2) -- (2,2.5) arc
    [
        start angle=180,
        end angle=0,
        x radius=0.5,
        y radius =.2
    ]    -- ({cos(210)+3.99},2) -- ({cos(210)+3.99},4.4) --  ({cos(330)+1},4.4)
    --cycle;
  \filldraw[fill=blue!40!white, draw=blue!40!white] ({cos(330)+1-.025},3)  arc
    [
        start angle=180,
        end angle=0,
        x radius=0.637,
        y radius =.3
    ]    -- ({cos(210)+4},2) -- ({cos(210)+4},4.4) --  ({cos(330)+1},4.4)
    --cycle;      
   \draw[green,thick,decoration={markings, mark=at position 0.5 with {\arrow{<}}},
        postaction={decorate}]  ({cos(210)-2},{.4*sin(210)})  -- ({cos(210)-2},{.4*sin(210)+4.4}) ;
    \draw[lightgray,thick,decoration={markings, mark=at position 0.5 with {\arrow{>}}},
        postaction={decorate}]  ({cos(270)-2},{.4*sin(270)}) -- ({cos(270)-2},{.4*sin(270)+4.41}) ;
    \draw[black,thick]  ({cos(330)-2},{.4*sin(330)}) -- ({cos(330)-2},{.4*sin(330)+1.2}) ;   
    \draw[black,thick,decoration={markings, mark=at position 0.5 with {\arrow{>}}},
        postaction={decorate}] ({cos(330)-1.995},1)  arc
    [
        start angle=180,
        end angle=0,
        x radius=0.630,
        y radius =.3
    ] ; 
    \draw[black,thick]  ({cos(210)+.99+3},{.4*sin(210)})  -- ({cos(210)+.99+3},{.4*sin(210)+3.2}) ;
    \draw[yellow,thick,decoration={markings, mark=at position 0.5 with {\arrow{>}}},
        postaction={decorate}]  ({cos(270)+1+3},{.4*sin(270)}) -- ({cos(270)+1+3},{.4*sin(270)+4.31}) ;
    \draw[black,thick]  ({cos(330)+1},{.4*sin(330)}) -- ({cos(330)+1},{.4*sin(330)+3.2}) ;
       
    \draw[black,thick,decoration={markings, mark=at position 0.5 with {\arrow{>}}},
        postaction={decorate}] ({cos(330)-1.995+3},3)  arc
    [
        start angle=180,
        end angle=0,
        x radius=0.630,
        y radius =.3
    ] ; 
    \draw[orange,thick,decoration={markings, mark=at position 0.5 with {\arrow{>}}},
        postaction={decorate}]  ({cos(270)+1},{.4*sin(270)}) -- ({cos(270)+1},{.4*sin(270)+4.31}) ;    
    \draw[black,thick]  ({cos(210)+.99},{.4*sin(210)})  -- ({cos(210)+.99},{.4*sin(210)+1.2}) ;

    \draw[black,thick]  ({cos(330)-2},{.4*sin(330)}) -- ({cos(330)-2},{.4*sin(330)+1.2}) ;   
    \draw[black,thick,decoration={markings, mark=at position 0.5 with {\arrow{>}}},
        postaction={decorate}] ({cos(330)-1.995},1)  arc
    [
        start angle=180,
        end angle=0,
        x radius=0.630,
        y radius =.3
    ] ; 
\draw[blue,thick,decoration={markings, mark=at position 0.5 with {\arrow{>}}},
        postaction={decorate}]  ({cos(330)+4},{.4*sin(330)}) -- ({cos(330)+4},{.4*sin(330)+4.4}) ;
    
    \filldraw[fill=red!10!white, draw=red!10!white] (5,4.3) arc
    [
        start angle=0,
        end angle=360,
        x radius=4,
        y radius =.6
    ];         
\end{tikzpicture}
\]
\[
{\rotatebox[origin=c]{90}{$\cong$}}
\]
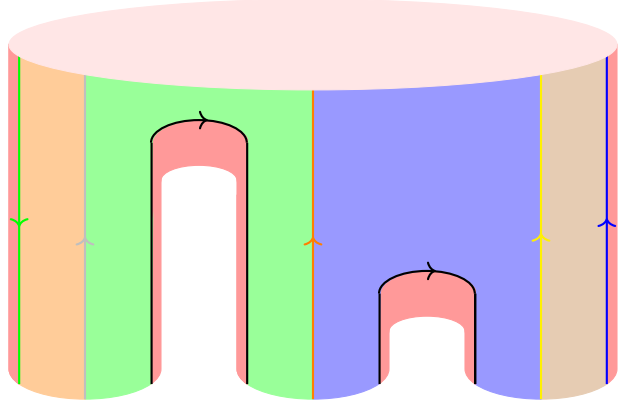
\[
\begin{tikzpicture}
\begin{scope}[xscale=-1,yscale=1]
\filldraw[fill=red!40!white, draw=red!40!white] (-3,0) arc
    [
        start angle=180,
        end angle=210,
        x radius=1,
        y radius =.4
    ]    -- ({cos(210)-2},{.4*sin(210)+4.4}) 
    -- (-3,4.3)
     --cycle;
     \filldraw[fill=brown!40!white, draw=brown!40!white] ({cos(210)-2},{.4*sin(210)}) arc
    [
        start angle=210,
        end angle=270,
        x radius=1,
        y radius =.4
    ]    -- ({cos(270)-2},{.4*sin(270)+4.4}) --  ({cos(210)-2},{.4*sin(210)+4.4})
     --cycle;
     \filldraw[fill=blue!40!white, draw=blue!40!white] ({cos(270)-2},{.4*sin(270)}) arc
    [
        start angle=270,
        end angle=330,
        x radius=1,
        y radius =.4
    ]    -- ({cos(330)-2},{.4*sin(330)+4.4}) -- ({cos(270)-2},{.4*sin(270)+4.4})
     --cycle;
     \filldraw[fill=red!40!white, draw=red!40!white] ({cos(330)-2},{.4*sin(330)}) arc
    [
        start angle=330,
        end angle=360,
        x radius=1,
        y radius =.4
    ]    -- ({cos(360)-2},{.4*sin(360)+2.4}) -- ({cos(330)-2},{.4*sin(330)+2.4})
    --cycle;
    \filldraw[fill=red!40!white, draw=red!40!white] (0,0) arc
    [
        start angle=180,
        end angle=210,
        x radius=1,
        y radius =.4
    ]    -- ({cos(210)+1},{.4*sin(210)+2.4}) -- ({cos(180)+1},{.4*sin(180)+2.4})
     --cycle;
     \filldraw[fill=blue!40!white, draw=blue!40!white] ({cos(210)+1},{.4*sin(210)}) arc
    [
        start angle=210,
        end angle=270,
        x radius=1,
        y radius =.4
    ]    -- ({cos(270)+1},{.4*sin(270)+4.4}) -- ({cos(210)+.95},{.4*sin(210)+4.4})
     --cycle;
     \filldraw[fill=green!40!white, draw=green!40!white] ({cos(270)+1},{.4*sin(270)}) arc
    [
        start angle=270,
        end angle=330,
        x radius=1,
        y radius =.4
    ]    -- ({cos(330)+1},{.4*sin(330)+4.4}) -- ({cos(270)+1},{.4*sin(270)+4.4})
     --cycle;
     \filldraw[fill=red!40!white, draw=red!40!white] ({cos(330)+1},{.4*sin(330)}) arc
    [
        start angle=330,
        end angle=360,
        x radius=1,
        y radius =.4
    ]    -- ({cos(360)+1},{.4*sin(360)+2.3}) -- ({cos(330)+1},{.4*sin(330)+2.4})
    --cycle;
   \filldraw[fill=red!40!white, draw=red!40!white] ({cos(330)-2},0) -- (-1,0.5) arc
    [
        start angle=180,
        end angle=0,
        x radius=0.5,
        y radius =.2
    ]    -- ({cos(210)+.99},0) -- ({cos(210)+.99},2.4) --  ({cos(330)-2},2.4)
    --cycle;
  \filldraw[fill=blue!40!white, draw=blue!40!white] ({cos(330)-2.025},1)  arc
    [
        start angle=180,
        end angle=0,
        x radius=0.637,
        y radius =.3
    ]    -- ({cos(210)+1},0) -- ({cos(210)+1},4.4) --  ({cos(330)-2},4.4)
    --cycle;  
   
    \filldraw[fill=red!10!white, draw=red!40!white] ({cos(330)+4},{.4*sin(330)}) arc
    [
        start angle=330,
        end angle=360,
        x radius=1,
        y radius =.4
    ]    -- ({cos(360)+4},{.4*sin(360)+4.3}) -- ({cos(330)+4},{.4*sin(330)+4.4})
    --cycle;

         \filldraw[fill=red!40!white, draw=red!40!white] ({cos(330)+1},{.4*sin(330)}) arc
    [
        start angle=330,
        end angle=360,
        x radius=1,
        y radius =.4
    ]    -- ({cos(360)+1},{.4*sin(360)+3.4}) -- ({cos(330)+1},{.4*sin(330)+3.4})
    --cycle;
    \filldraw[fill=red!40!white, draw=red!40!white] (3,0) arc
    [
        start angle=180,
        end angle=210,
        x radius=1,
        y radius =.4
    ]    -- ({cos(210)+4},{.4*sin(210)+3.4}) -- ({cos(180)+4},{.4*sin(180)+3.4})
     --cycle;
     \filldraw[fill=green!40!white, draw=green!40!white] ({cos(210)+4},{.4*sin(210)}) arc
    [
        start angle=210,
        end angle=270,
        x radius=1,
        y radius =.4
    ]    -- ({cos(270)+4},{.4*sin(270)+4.4}) -- ({cos(210)+3.95},{.4*sin(210)+4.4})
     --cycle;
     \filldraw[fill=orange!40!white, draw=orange!40!white] ({cos(270)+4},{.4*sin(270)}) arc
    [
        start angle=270,
        end angle=330,
        x radius=1,
        y radius =.4
    ]    -- ({cos(330)+4},{.4*sin(330)+4.4}) -- ({cos(270)+4},{.4*sin(270)+4.4})
     --cycle;
     \filldraw[fill=red!40!white, draw=red!40!white] ({cos(330)+4},{.4*sin(330)}) arc
    [
        start angle=330,
        end angle=360,
        x radius=1,
        y radius =.4
    ]    -- ({cos(360)+4},{.4*sin(360)+4.3}) -- ({cos(330)+4},{.4*sin(330)+4.4})
    --cycle;
   \filldraw[fill=red!40!white, draw=red!40!white] ({cos(330)+1},2) -- (2,2.5) arc
    [
        start angle=180,
        end angle=0,
        x radius=0.5,
        y radius =.2
    ]    -- ({cos(210)+3.99},2) -- ({cos(210)+3.99},4.4) --  ({cos(330)+1},4.4)
    --cycle;
  \filldraw[fill=green!40!white, draw=green!40!white] ({cos(330)+1-.025},3)  arc
    [
        start angle=180,
        end angle=0,
        x radius=0.637,
        y radius =.3
    ]    -- ({cos(210)+4},2) -- ({cos(210)+4},4.4) --  ({cos(330)+1},4.4)
    --cycle;      
   \draw[blue,thick,decoration={markings, mark=at position 0.5 with {\arrow{>}}},
        postaction={decorate}]  ({cos(210)-2},{.4*sin(210)})  -- ({cos(210)-2},{.4*sin(210)+4.4}) ;
    \draw[yellow,thick,decoration={markings, mark=at position 0.5 with {\arrow{>}}},
        postaction={decorate}]  ({cos(270)-2},{.4*sin(270)}) -- ({cos(270)-2},{.4*sin(270)+4.41}) ;
    \draw[black,thick]  ({cos(330)-2},{.4*sin(330)}) -- ({cos(330)-2},{.4*sin(330)+1.2}) ;   
    \draw[black,thick,decoration={markings, mark=at position 0.5 with {\arrow{<}}},
        postaction={decorate}] ({cos(330)-1.995},1)  arc
    [
        start angle=180,
        end angle=0,
        x radius=0.630,
        y radius =.3
    ] ; 
    \draw[black,thick]  ({cos(210)+.99+3},{.4*sin(210)})  -- ({cos(210)+.99+3},{.4*sin(210)+3.2}) ;
    \draw[lightgray,thick,decoration={markings, mark=at position 0.5 with {\arrow{>}}},
        postaction={decorate}]  ({cos(270)+1+3},{.4*sin(270)}) -- ({cos(270)+1+3},{.4*sin(270)+4.31}) ;
    \draw[black,thick]  ({cos(330)+1},{.4*sin(330)}) -- ({cos(330)+1},{.4*sin(330)+3.2}) ;
       
    \draw[black,thick,decoration={markings, mark=at position 0.5 with {\arrow{<}}},
        postaction={decorate}] ({cos(330)-1.995+3},3)  arc
    [
        start angle=180,
        end angle=0,
        x radius=0.630,
        y radius =.3
    ] ; 
    \draw[orange,thick,decoration={markings, mark=at position 0.5 with {\arrow{>}}},
        postaction={decorate}]  ({cos(270)+1},{.4*sin(270)}) -- ({cos(270)+1},{.4*sin(270)+4.31}) ;    
    \draw[black,thick]  ({cos(210)+.99},{.4*sin(210)})  -- ({cos(210)+.99},{.4*sin(210)+1.2}) ;

    \draw[black,thick]  ({cos(330)-2},{.4*sin(330)}) -- ({cos(330)-2},{.4*sin(330)+1.2}) ;   
    
\draw[green,thick,decoration={markings, mark=at position 0.5 with {\arrow{<}}},
        postaction={decorate}]  ({cos(330)+4},{.4*sin(330)}) -- ({cos(330)+4},{.4*sin(330)+4.4}) ;
    
    \filldraw[fill=red!10!white, draw=red!10!white] (5,4.3) arc
    [
        start angle=0,
        end angle=360,
        x radius=4,
        y radius =.6
    ];         
\end{scope}
\end{tikzpicture}
\]
\end{proof}

We can now show that the invertible objects $l_{\Xi_{ijk}}$ satisfy the first coherence condition for projective 2-representations. This is the content of the following.
\begin{lemma}
The diagram
\begin{equation}\label{eq:coherenceC2-drawings}
    \begin{tikzcd}[column sep={8em,between origins}]
    l_{\Xi_{jkl}}\otimes l_{\Xi_{ijl}}\otimes l_{\Xi{ilm}}\ar[dd,"\mathrm{id}\otimes \phi_{\Upsilon_{ijlm}}"']&& l_{\Xi_{ijk}}\otimes l_{\Xi_{ikl}}\otimes l_{\Xi_{ilm}}\ar[d,"\mathrm{id}\otimes \phi_{\Upsilon_{iklm}}"]\ar[ll,"\phi_{\Upsilon_{ijkl}}\otimes\mathrm{id}"']&\\
    {}&&l_{\Xi_{ijk}}\otimes l_{\Xi_{klm}}\otimes l_{\Xi_{ikm}}\ar[d,"\wr"]\\
    l_{\Xi_{jkl}}\otimes l_{\Xi_{jlm}}\otimes l_{\Xi_{ijm}}\ar[dr,"\phi_{\Upsilon_{jklm}}\otimes\mathrm{id}"']&&
    l_{\Xi_{klm}}\otimes l_{\Xi_{ijk}} \otimes l_{\Xi_{ikm}}
    \ar[dl,"\mathrm{id}\otimes \phi_{\Upsilon_{ijkm}}"]\\
    &l_{\Xi_{klm}}\otimes l_{\Xi_{jkm}}\otimes l_{\Xi_{ijm}}
    \end{tikzcd}
    \end{equation}
commutes, for every 4-simplex in $\mathcal{C}$.
\end{lemma}
\begin{proof}
As above, let us denote by $\Sigma_{ijklm}$ the filler for the 4-simplex, by $\Upsilon_{ijkl}$ the fillers for its faces, by $\Xi_{ijk}$ the fillers of its 2-simplices, and by $f_{ij}$ its 1-simplices.
Recall from Proposition \ref{prop:wannabe-is-projective}, Corollary \ref{cor:m} and Lemma \ref{lem:n-invertible} the definitions of the invertible objects  $l_{\Xi_{ijk}}$, $m_{\Upsilon_{ijkl}}$ and $n_{\Sigma_{ijklm}}$ of $\mathcal{V}$.  Then we have a diagram of the form
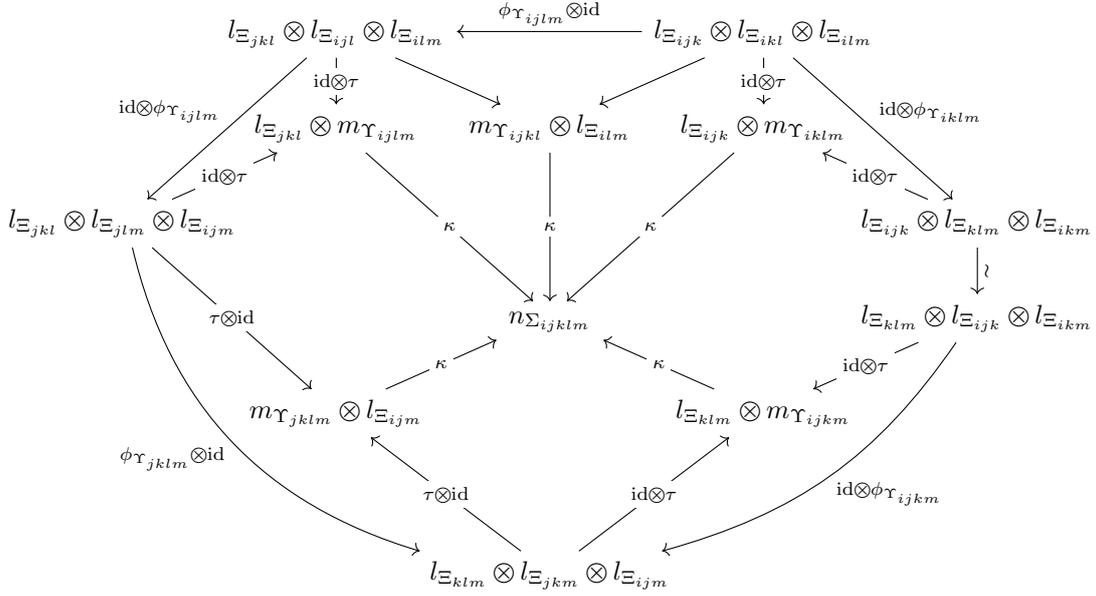
\begin{equation}\label{eq:pentacle}
    \begin{tikzcd}[column sep={8em,between origins}]
    &l_{\Xi_{jkl}}\otimes l_{\Xi_{ijl}}\otimes l_{\Xi_{ilm}}\ar[ddl,"\mathrm{id}\otimes \phi_{\Upsilon_{ijlm}}"']\ar[d,"\mathrm{id}\otimes \tau"{description}]\ar[dr]&& l_{\Xi_{ijk}}\otimes l_{\Xi_{ikl}}\otimes l_{\Xi_{ilm}}\ar[d,"\mathrm{id}\otimes \tau"{description}]\ar[dl]\ar[ddr,"\mathrm{id}\otimes \phi_{\Upsilon_{iklm}}"]\ar[ll,"\phi_{\Upsilon_{ijlm}}\otimes\mathrm{id}"']&\\
    &l_{\Xi_{jkl}}\otimes m_{\Upsilon_{ijlm}}\ar[ddr,"\kappa"{description}]&m_{\Upsilon_{ijkl}}\otimes l_{\Xi_{ilm}}\ar[dd,"\kappa"{description}]&{l_{\Xi_{ijk}}\otimes m_{\Upsilon_{iklm}}}\ar[ddl,"\kappa"{description}]\\
    l_{\Xi_{jkl}}\otimes l_{\Xi_{jlm}}\otimes l_{\Xi_{ijm}}
    \ar[ddddrr,bend right=30,"\phi_{\Upsilon_{jklm}}\otimes\mathrm{id}"']
    \ar[rdd," \tau\otimes\mathrm{id}"{description}]\ar[ru,"\mathrm{id}\otimes \tau"{description}]&&&& l_{\Xi_{ijk}}\otimes l_{\Xi_{klm}}\otimes l_{\Xi_{ikm}}\ar[lu,"\mathrm{id}\otimes \tau"{description}]\ar[d,"\wr"]
    \\
    &&n_{\Sigma_{ijklm}}&& l_{\Xi_{klm}}\otimes l_{\Xi_{ijk}}\otimes l_{\Xi_{ikm}}\ar[dl,"\mathrm{id}\otimes \tau"{description}]\ar[dddll,bend left=20,"\mathrm{id}\otimes \phi_{\Upsilon_{ijkm}}"]\\
    &m_{\Upsilon_{jklm}}\otimes l_{\Xi_{ijm}}\ar[ur,"\kappa"{description}]&&l_{\Xi_{klm}}\otimes m_{\Upsilon_{ijkm}}\ar[ul,"\kappa"{description}]\\
    {}
    \\
    &&l_{\Xi_{klm}}\otimes l_{\Xi_{jkm}}\otimes l_{\Xi_{ijm}}\ar[uul," \tau\otimes\mathrm{id}"{description}]\ar[uur,"\mathrm{id}\otimes \tau"{description}]
    \end{tikzcd}
    \end{equation}  
where all arrows are isomorphisms. All the triangular subdiagrams commute by definition, and all the quadrangular subdiagrams and the pentagonal subdiagram on the right commute by Lemma \ref{lem:two-kappas}. Notice that the appearance of the commutation isomorphism $l_{\Xi_{ijk}}\otimes l_{\Xi_{klm}}\cong l_{\Xi_{klm}}\otimes l_{\Xi_{ijk}}$ on the right side of the diagram is solely due to the fact that to write tensor products in line an order of the factors has to be chosen. This order is clearly irrelevant, due to the symmetry of the monoidal structure on $\mathcal{V}$.
\end{proof}

\begin{lemma}\label{lemma:drawing-ev}
For any 2-simplex $\Xi_{ijk}$ in $\mathcal{C}$ we have a distinguished morphism
\[
\epsilon_{\Xi_{ijk}}\colon M_{f_{ik}}\otimes  l_{\Xi_{ijk}}\to M_{f_{jk}}\circ M_{f_{ij}}.
\]
\end{lemma}
\begin{proof}
The morphism $\epsilon_{\Xi_{ijk}}$ is defined by the following surface with defects (the ``dorade box''):
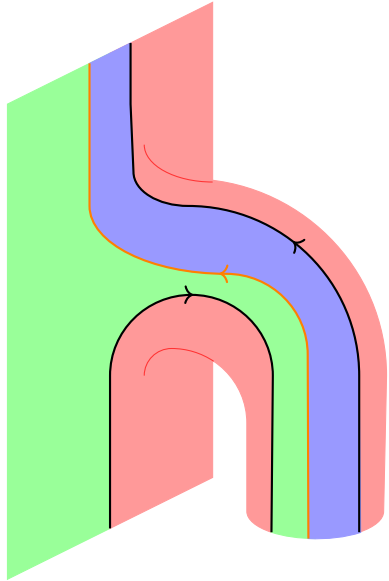
\[
\begin{tikzpicture}
\begin{scope}[xscale=.9,yscale=.9]
\filldraw[fill=green!40!white, draw=green!40!white] (0,0) -- (1.5,0.75) -- (1.5,7.75) -- (0,7) -- cycle;
\filldraw[fill=red!40!white, draw=red!40!white] (1.5,0.75) -- (1.5,7.75) -- (3,8.5) -- (3,1.5) -- cycle;
\filldraw[fill=red!40!white, draw=red!40!white] (3.5,1) arc
    [
        start angle=180,
        end angle=230,
        x radius=1,
        y radius =.4
    ]    -- ({cos(230)+4.5},3) 
     arc [
        start angle=0,
        end angle=90,
        x radius=1.5,
        y radius =1.5
    ]   
    arc [
        start angle=90,
        end angle=180,
        x radius=.8,
        y radius =.8
    ]   
    -- (2,3)
    arc [
        start angle=180,
        end angle=90,
        x radius=.4,
        y radius =.4
    ]   
    arc [
        start angle=90,
        end angle=0,
        x radius=1.1,
        y radius =1.1
    ]   
    -- cycle; 
    \draw[red!80!white] 
     (2,3)
    arc [
        start angle=180,
        end angle=90,
        x radius=.4,
        y radius =.4
    ]   
    arc [
        start angle=90,
        end angle=56,
        x radius=1.1,
        y radius =1.1
    ]  ; 
 \filldraw[fill=red!40!white, draw=red!40!white] (5.5,1) arc
    [
        start angle=0,
        end angle=-50,
        x radius=1,
        y radius =.4
    ]    -- ({cos(-50)+4.5},3) 
     arc [
        start angle=0,
        end angle=90,
        x radius=2.5,
        y radius =2.5
    ]   
    arc [
        start angle=270,
        end angle=90,
        x radius=.2,
        y radius =.2
    ]   
    arc [
        start angle=90,
        end angle=0,
        x radius=2.9,
        y radius =2.9
    ]   
    -- cycle;    
\draw[red!80!white] 
    (3,5.8525)
    arc [
        start angle=270,
        end angle=180,
        x radius=1,
        y radius =.55
    ]   
    ;   
    \filldraw[fill=green!40!white, draw=green!40!white] ({cos(230)+4.5},{0.4*sin(230)+1})
    arc
    [
        start angle=230,
        end angle=270,
        x radius=1,
        y radius =.4
    ]    --  
    (4.5,3)
     arc [
        start angle=0,
        end angle=90,
        x radius=1.8,
        y radius =1.8
    ]   
    arc [
        start angle=270,
        end angle=180,
        x radius=2,
        y radius =.8
    ]   
    -- (0.7,3) -- (1.5,3)
    arc [
        start angle=180,
        end angle=0,
        x radius=1.19,
        y radius =1.19
    ]   
    -- cycle;
    \filldraw[fill=blue!40!white, draw=blue!40!white] 
    (4.4,.6)
    arc
    [
        start angle=263,
        end angle=310,
        x radius=1,
        y radius =.4
    ]    --
    ({4.5+cos(310)},3)
     arc [
        start angle=0,
        end angle=90,
        x radius=2.5,
        y radius =2.5
    ]   
    arc [
        start angle=270,
        end angle=180,
        x radius=.8,
        y radius =.5
    ]   
    -- (1.8,7) -- (1.8,7.9) -- ({1.8-.6},{7.9-.3})-- (1.2,5.5)
    arc [
        start angle=180,
        end angle=270,
        x radius=2,
        y radius =1
    ]   
    arc [
        start angle=90,
        end angle=0,
        x radius=1.19,
        y radius =1.19
    ]   
    -- cycle;
   \draw[orange,thick,decoration={markings, mark=at position 0.5 with {\arrow{<}}},
        postaction={decorate}]   ({1.8-.6},{7.9-.3})-- (1.2,5.5)
    arc [
        start angle=180,
        end angle=270,
        x radius=2,
        y radius =1
    ]   
    arc [
        start angle=90,
        end angle=0,
        x radius=1.19,
        y radius =1.19
        ]
        -- (4.4,.6); 
    \draw[black,thick,decoration={markings, mark=at position 0.5 with {\arrow{>}}},
        postaction={decorate}]   ({4.5+cos(310)},{1+0.4*sin(310)}) -- ({4.5+cos(310)},3)
     arc [
        start angle=0,
        end angle=90,
        x radius=2.5,
        y radius =2.5
    ]   
    arc [
        start angle=270,
        end angle=180,
        x radius=.8,
        y radius =.5
    ]   
    -- (1.8,7) -- (1.8,7.9);
    \draw[black,thick,decoration={markings, mark=at position 0.5 with {\arrow{>}}},
        postaction={decorate}]   (1.5,.75)--(1.5,3)
     arc [
        start angle=180,
        end angle=0,
        x radius=1.19,
        y radius =1.19
    ]   -- ({1.5+2*1.18},{0.4*sin(230)+1})
    ;
\end{scope}    
\end{tikzpicture}     
\]
\end{proof}

\begin{remark}
When $2\mathcal{V}$ is $2\Vect_\K$ or $2\mathrm{s}\Vect_\K$, the morphism $\epsilon_{\Xi_{ijk}}$ from Lemma \ref{lemma:drawing-ev} is the evaluation morphism
\[
\mathrm{ev}_{\Xi_{ijk}}\colon M_{f_{ik}}\otimes \mathrm{Hom}(M_{f_{ik}},M_{f_{jk}}\otimes_{A_{X_j}}M_{f_{ij{}}})\to M_{f_{jk}}\otimes_{A_{X_j}}M_{f_{ij{}}}
\]
\end{remark}
\begin{lemma}\label{lemma:epsilon-iso}
The morphism $\epsilon_{\Xi_{ijk}}$ is an isomorphism.    
\end{lemma}
\begin{proof}
Let $\lambda_{\Xi_{ijk}}$ be the morphism represented by the following surface with defects:
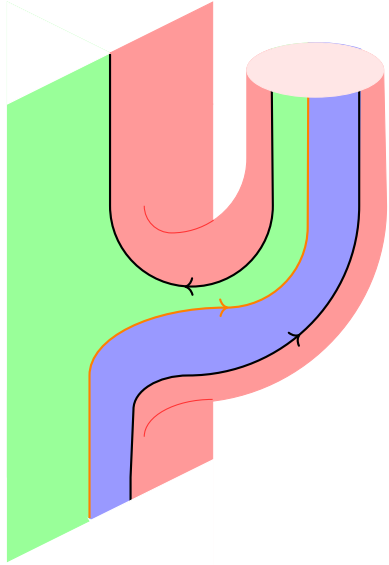
\[
     
\]
The morphism $\lambda_{\Xi_{ijk}}$ is the inverse of ${\epsilon}_{\Xi_{ijk}}$. Indeed, if we compute ${\epsilon}_{\Xi_{ijk}}\circ \lambda_{\Xi_{ijk}}$ we find
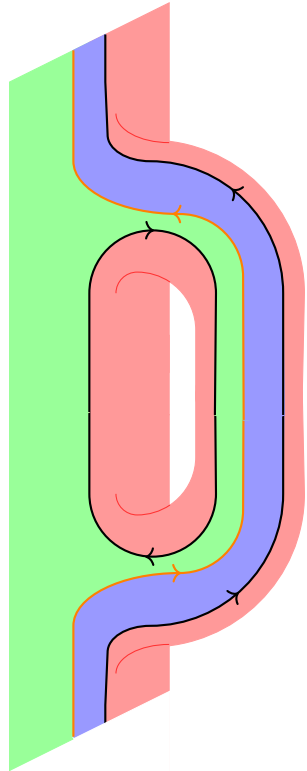
\[
%
\]
By the invertibility  of
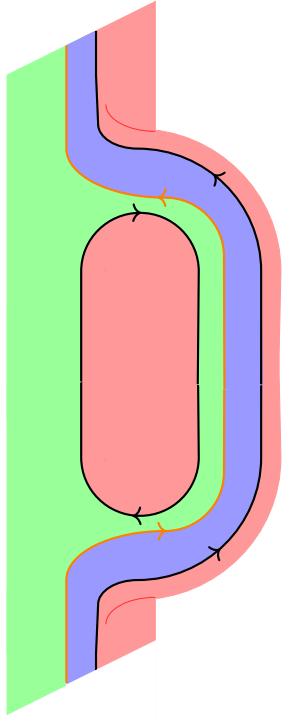
\raisebox{-.3cm}{
%
\]
Now we use the invertibility of \raisebox{-.3cm}{
%
\]
and this is manifestly equivalent to
\raisebox{-.3cm}{
%
}. Therefore,
\[
{\epsilon}_{\Xi_{ijk}}\circ \lambda_{\Xi_{ijk}}=\mathrm{id}_{M_{f_{jk}}\circ M_{f_{ij}}}.
\]
Now we compute $\lambda_{\Xi_{ijk}}\circ {\epsilon}_{\Xi_{ijk}}$. This is represented by
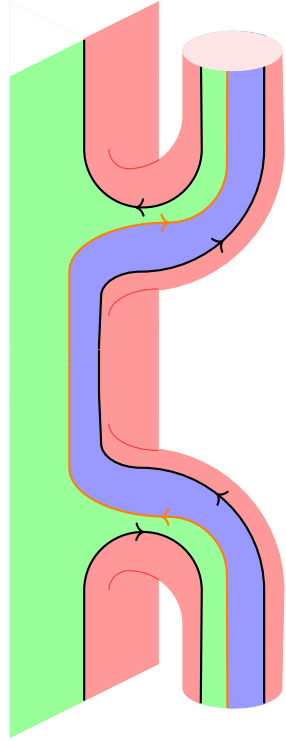
\[
%
\]
This is equivalent to
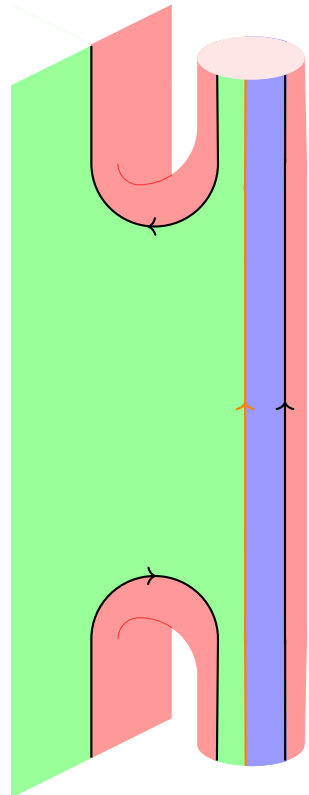
\[
%
\]
By the invertibility of \raisebox{-.3cm}{
%
\]
Finally, by the invertibility of \raisebox{-.3cm}{
%
\]
Therefore, 
\[
\lambda_{\Xi_{ijk}}\circ {\epsilon}_{\Xi_{ijk}}=\mathrm{id}_{M_{f_{ik}}\otimes l_{\Xi_{ijk}}}.
\]
\end{proof}
Lemmas \ref{lemma:drawing-ev} and \ref{lemma:epsilon-iso} have the following immediate generalization.
\begin{lemma}\label{lemma:drawing-ev4}
For any 3-simplex $\Upsilon_{ijkl}$ in $\mathcal{C}$ we have a distinguished isomorphism
\[
\epsilon_{\Upsilon_{ijkl}}\colon M_{f_{il}}\otimes  m_{\Upsilon_{ijkl}}\to M_{f_{kl}}\circ M_{f_{jk}}\circ M_{f_{ij}}.
\]
\end{lemma}
The next lemma is the key to proving that the coherence condition \eqref{eq:coherenceC1-drawings-inverted} holds.
\begin{lemma}\label{lemma:for-second-coherence-condition}
For any 3-simplex $\Upsilon_{ijkl}$ in $\mathcal{C}$, with faces $\Xi_{ijk}$, the following diagrams commute:
\[
\begin{tikzcd}
    M_{f_{il}}\otimes l_{\Xi_{ijk}}\otimes l_{\Xi_{ikl}}\ar[r,"\sim"]\ar[d,"\mathrm{id}\otimes \tau^{\mathrm{odd}}_{\Upsilon_{ijkl}}"']&M_{f_{il}}\otimes l_{\Xi_{ikl}}\otimes l_{\Xi_{ijk}}\ar[r,"\epsilon_{\Xi_{ikl}}\otimes \mathrm{id}"]&
    (M_{f_{kl}}\circ M_{f_{ik}})\otimes l_{\Xi_{ijk}}\ar[r,"\sim"]&M_{f_{kl}}\circ (M_{f_{ik}}\otimes l_{\Xi_{ijk}})\ar[d,"\mathrm{id}\circ \epsilon_{\Xi_{ijk}}"]\\
    M_{f_{il}}\otimes m_{\Upsilon_{ijkl}}\ar[rrr,"\epsilon_{\Upsilon_{ijkl}}"'] &&& M_{f_{kl}}\circ M_{f_{jk}}\circ M_{f_{ij}}
\end{tikzcd} 
\]
\[
\begin{tikzcd}
    M_{f_{il}}\otimes l_{\Xi_{jkl}}\otimes l_{\Xi_{ijl}}\ar[r,"\sim"]\ar[d,"\mathrm{id}\otimes\tau^{\mathrm{even}}_{\Upsilon_{ijkl}}"']& M_{f_{il}}\otimes l_{\Xi_{ijl}}\otimes l_{\Xi_{jkl}}\ar[r,"\epsilon_{\Xi_{ijl}}\otimes \mathrm{id}"]&
    (M_{f_{jl}}\circ M_{f_{ij}})\otimes l_{\Xi_{jkl}}\ar[r,"\sim"]&(M_{f_{jl}}\otimes l_{\Xi_{jkl}})\circ M_{f_{ij}}\ar[d," \epsilon_{\Xi_{jkl}}\circ \mathrm{id}"]\\
    M_{f_{il}}\otimes m_{\Upsilon_{ijkl}}\ar[rrr,"\epsilon_{\Upsilon_{ijkl}}"'] &&& M_{f_{kl}}\circ M_{f_{jk}}\circ M_{f_{ij}}
\end{tikzcd} 
\]
\end{lemma}
\begin{proof}
    The surfaces with defects representing the compositions $(\mathrm{id}\circ \epsilon_{\Xi_{ijk}})\circ (\epsilon_{\Xi_{ikl}}\otimes \mathrm{id})$ and $\epsilon_{\Upsilon_{ijkl}}\circ \tau^{\mathrm{odd}}_{\Upsilon_{ijkl}}$ are manifestly diffeomorphic. This proves the commutativity of the first diagram. The proof for the second diagram is identical.
\end{proof}

\begin{corollary}
    For any 3-simplex $\Upsilon_{ijkl}$ in $\mathcal{C}$ following diagram commutes:
\begin{equation}\label{eq:coherenceC1-drawings}
\begin{tikzcd}
&M_{f_{kl}}\circ M_{f_{jk}}\circ M_{f_{ij}}
\\
M_{f_{kl}}\circ (M_{f_{ik}}\otimes l_{\Xi_{ijk}})  \arrow[ur,"\mathrm{id}\circ \epsilon_{\Xi_{ijk}}"]              && (M_{f_{jl}}\otimes l_{\Xi_{jkl}})\circ M_{f_{ij}}\arrow[ul,"\epsilon_{\Xi_{jkl}}\circ \mathrm{id}"']\\ 
(M_{f_{kl}}\circ M_{f_{ik}})\otimes l_{\Xi_{ijk}}\ar[u,"\wr"'] && (M_{f_{jl}}\circ M_{f_{ij}})\otimes l_{\Xi_{jkl}}\arrow[u,"\wr"']
\\
M_{f_{il}} \otimes l_{\Xi_{ikl}}\otimes l_{\Xi_{ijk}}\arrow[u,"\epsilon_{\Xi_{ikl}}\otimes \mathrm{id}"] && M_{f_{il}}\otimes l_{\Xi_{ijl}} \arrow[u,"\epsilon_{\Xi_{ijl}}\otimes \mathrm{id}"']
\\
M_{f_{il}} \otimes l_{\Xi_{ijk}}\otimes l_{\Xi_{ikl}} \arrow[u,"\wr"] 
&& M_{f_{il}}\otimes l_{\Xi_{ijl}}\otimes l_{\Xi_{jkl}} \arrow[u,"\wr"'] \arrow[ll,"\mathrm{id}\otimes \psi_{\Upsilon{ijkl}} "']
\end{tikzcd}
\end{equation}
\end{corollary}
\begin{proof}
Diagram \eqref{eq:coherenceC1-drawings} factors as
\[
\begin{tikzcd}
&M_{f_{kl}}\circ M_{f_{jk}}\circ M_{f_{ij}}
\\
M_{f_{kl}}\circ (M_{f_{ik}}\otimes l_{\Xi_{ijk}})  \arrow[ur,"\mathrm{id}\circ \epsilon_{\Xi_{ijk}}"]              && (M_{f_{jl}}\otimes l_{\Xi_{jkl}})\circ M_{f_{ij}}\arrow[ul,"\epsilon_{\Xi_{jkl}}\circ \mathrm{id}"']\\ 
(M_{f_{kl}}\circ (M_{f_{ik}})\otimes l_{\Xi_{ijk}}\ar[u,"\wr"'] && (M_{f_{jl}}\circ M_{f_{ij}})\otimes l_{\Xi_{jkl}}\arrow[u,"\wr"']
\\
M_{f_{il}} \otimes l_{\Xi_{ikl}}\otimes l_{\Xi_{ijk}}\arrow[u,"\epsilon_{\Xi_{ikl}}\otimes \mathrm{id}"]&M_{f_{il}}\otimes m_{\Upsilon_{ijkl}}\arrow[uuu,"\epsilon_{\Upsilon_{ijkl}}"]& M_{f_{il}}\otimes l_{\Xi_{ijl}} \otimes l_{\Xi_{jkl}} \arrow[u,"\epsilon_{\Xi_{ijl}}\otimes \mathrm{id}"']
\\
M_{f_{il}} \otimes l_{\Xi_{ijk}}\otimes l_{\Xi_{ikl}} \arrow[u,"\wr"] \ar[ur,"\mathrm{id}\otimes \tau^{\mathrm{odd}}_{\Upsilon_{ijkl}}"]
&& M_{f_{il}}\otimes l_{\Xi_{jkl}}\otimes l_{\Xi_{ijl}} \arrow[u,"\wr"'] \arrow[ll,"\mathrm{id}\otimes \psi_{\Upsilon{ijkl}} "]\ar[ul,"\mathrm{id}\otimes \tau^{\mathrm{even}}_{\Upsilon_{ijkl}}"']
\end{tikzcd}
\]
Where the subdiagrams commute by Lemma \ref{lemma:for-second-coherence-condition} and by definition of $\phi_{\Upsilon_{ijkl}}$. This, together with the invertibility of $\tau^{\mathrm{even}}_{\Upsilon_{ijkl}}$, concludes the proof.
\end{proof}
Inverting arrows we get the second coherence condition for projective 2-representations. That is, we have the following.
\begin{corollary}
    For any 3-simplex $\Upsilon_{ijkl}$ in $\mathcal{C}$ following diagram commutes:
\begin{equation}\label{eq:coherenceC1-drawings-inverted}
\begin{tikzcd}
&M_{f_{kl}}\circ M_{f_{jk}}\circ M_{f_{ij}}
\\
M_{f_{kl}}\circ (M_{f_{ik}}\otimes l_{\Xi_{ijk}})  \arrow[ur,<-,"\mathrm{id}\otimes \lambda_{\Xi_{ijk}}"]              && (M_{f_{jl}}\otimes l_{\Xi_{jkl}})\circ M_{f_{ij}}\arrow[ul,<-,"\lambda_{\Xi_{jkl}}\otimes \mathrm{id}"']\\ 
(M_{f_{kl}}\circ M_{f_{ik}})\otimes l_{\Xi_{ijk}}\ar[u,<-,"\wr"'] && (M_{f_{jl}}\circ M_{f_{ij}})\otimes l_{\Xi_{jkl}}\arrow[u,<-,"\wr"']
\\
M_{f_{il}} \otimes l_{\Xi_{ikl}}\otimes l_{\Xi_{ijk}}
\arrow[u,<-,"\lambda_{\Xi_{ikl}}\otimes \mathrm{id}"] 
  && M_{f_{il}}\otimes l_{\Xi_{ijl}}\otimes l_{\Xi_{jkl}} \arrow[u,<-,"\lambda_{\Xi_{ijl}}\otimes \mathrm{id}"']
\\
M_{f_{il}} \otimes l_{\Xi_{ijk}}\otimes l_{\Xi_{ikl}} \arrow[u,<-,"\wr"']
&& M_{f_{il}}\otimes  l_{\Xi_{jkl}} \otimes l_{\Xi_{ijl}} \arrow[u,<-,"\wr"'] \arrow[ll,<-,"\mathrm{id}\otimes \phi_{\Upsilon{ijkl}} "']
\end{tikzcd}
\end{equation}
\end{corollary}
This concludes the proof of Proposition \ref{prop:wannabe-is-projective}.

\begin{remark} The statements and proofs of Propositions \ref{prop:wannabe-is-proj-0}and \ref{prop:wannabe-is-projective} clearly suggest the following statement should be true for every $n$: let $n\mathcal{V}$ be a symmetric monoidal $n$-category, and let $\rho \colon \mathcal{C}\dashrightarrow n\mathcal{V}$ be an invertible wannabe functor. Then $\rho$ canonically extends to a projective representation $\rho\colon \mathcal{C}\to n\mathcal{V}/\!/\mathbf{B}\Pic((n-1)\mathcal{V})$, where $(n-1)\mathcal{V}=\Omega n\mathcal{V}=\mathrm{End}(\mathbf{1}_{n\mathcal{V}})$. Moreover,  a proof should be naturally espressed in terms of $n$-dimensional TQFTs with defects, and the $n$-dimensional proof should be a ``slicing'' of the $(n+1)$-dimensional proof. Clearly, for $n\geq 3$, it will be impossible to derive the proof by drawing. Yet, the structure of the proofs suggests that a Morse theoretic argument could possibly systematically handle the argument for any $n$. This possibility will be hopefully investigated elsewhere.

\end{remark}

\section{The Clifford/Fock construction}\label{sec:clifford-fock}
We now present the prototypical example of an invertible projective representation: the Clifford/Fock construction. Realizing that this is an example of a projective representation of Lagrangian correspondences is somehow implicit in the work of Stolz and Teichner, especially \cite{what-is-an-elliptic-object}, and has been made fully explicit in the work of Ludewig and Roos \cite{ludewig-roos}. Strictly speaking, Ludewig and Roos do not talk the language of projective 2-representations, so there is not an explicit realization of the Pfaffian lines as a $\mathrm{Pic}(\mathrm{sVect}_\mathbb{C})$ 2-cocycle, and only coherence condition \eqref{eq:coherenceC1-drawings-inverted} is verified in \cite{ludewig-roos}. But we think it is honest to say that the lack of verification of coherence condition \eqref{eq:coherenceC2-drawings} is a minor detail in \cite{ludewig-roos}, and that the result that the Clifford/Fock construction is a projective 2-representation is to be entirely ascribed to them. Understanding their work, and in particular whether the proof of \eqref{eq:coherenceC1-drawings-inverted} in \cite{ludewig-roos} was an consequence of exceptional features of the Clifford/Fock construction or a general result holding under the sole assumption of invertibility has been one of the main motivations behind our work. The other main motivation has been a better understanding of the definition of Stolz and Teichner's Clifford field theories. Not surprisingly the two questions are intimately related, see \cite{vuppulury}. Since all the material in this Section is well established in the literature we will be quite sketchy, omitting most of the proofs, and address the reader to \cite{ludewig-roos} for all details. A notable exception is Proposition \ref{prop:beta-to-the-minus-one}: since we were not able to explicitly locate it in the literature, and despite it is essentially a (long) exercise in linear algebra, we provide a detailed proof of it. Rather than establishing a new result, he main contribution from the present Section consists in showing how the Clifford/Fock construction is a particular case of the construction of a projective 2-representation out of an invertible wannabe functor presented in Section \ref{sec:invertible-2-rep}. To avoid any topological issue, we will content ourselves in working with finite dimensional (super-)Hilbert spaces. So, unless differently specified, by `Hilbert space' we always mean `finite dimensional Hilbert space' in this Section.

\subsection{The category $\mathbf{LagrCorr}_\K$}  
We recall the definition of the category of linear Lagrangian correspondences. The symbol $\mathbb{K}$ will always denote either the field $\mathbb{R}$ of real numbers or the field $\mathbb{C}$ of complex numbers. For simplicity, we present the construction for (finite dimensional) Hilbert spaces, leaving the straightforward generalization to (finite dimensional) super-Hilbert spaces to the reader. All details can be found in \cite{ludewig-roos, ludewig}; for the notation, however, we will be mostly following \cite{what-is-an-elliptic-object}.

 \begin{definition}An \emph{antiunvolutive Hilbert space} is a pair $(H,\alpha)$ where $H$ is a (real or complex) Hilbert space and $\alpha\colon H\to H$ is an antiunitary involution (an antiunvolution for short). 
 \end{definition}
 A classical example is $H=\mathbb{C}^n$ with $\alpha(v)=\overline{v}$. Another example is $\mathbb{R}^n$ with the identity. When the antiunvolution is clear from the context we will simply write $H$ for $(H,\alpha)$. Notice that if $\alpha$ is an antiunvolution of $H$ then also $-\alpha$ is an antiunvolution. When $\alpha$ is clear from the context we will write $-H$ for $(H,-\alpha)$ .
 Direct sum of Hilbert spaces induces a direct sum of antiunvolutive Hilbert spaces.
 \begin{definition}A \emph{quadratic} $\mathbb{K}$-vector space is a pair $(V,b)$, where $V$ is a $\mathbb{K}$ vector space and $b\colon V\otimes _{\mathbb{K}} V\to \mathbb{K}$ is a symmetric $\mathbb{K}$-linear map. A quadratic vector space $(V,b)$ is called non-degenerate if the symmetric bilinear form $b$ is nondegenerate. 
\end{definition}
 With an antiunvolutive Hilbert space $(H,\alpha)$ is naturally associated a nondegenerate quadratic $\mathbb{K}$-vector space, with  $\mathbb{K}=\mathbb{R}$ or $\mathbb{K}=\mathbb{C}$ depending whether $H$ is real or complex, by the rule 
 \[
 b_\alpha(v\otimes w)=\langle\alpha(v)|w\rangle.
 \]
 Indeed, antiunitarity and involutivness of $\alpha$ give
 \[
 b_\alpha(v\otimes w)=\langle \alpha(v)|w\rangle = \overline{\langle \alpha^2(v)|\alpha(w)\rangle} = \overline{\langle v|\alpha(w)\rangle}=\langle \alpha(w)|v\rangle=b_\alpha(w\otimes v).
 \]
 \begin{definition}{A \emph{subLagrangian} subspace $U$ of $(H,\alpha)$ is a closed isotropic subspace of $(H,b_\alpha)$, i.e., a closed linear subspace $U$ such that $b_\alpha\bigr\vert_{U\otimes_{\mathbb{K}}U}=0$, such that the sum $U+\alpha(U)$ is direct and has finite codimension in $H$.\footnote{Both the closedness condition and the finite codimension condition are trivial in the finite-dimensional setting we are considering; we include them in the definition since they are essential in the generalization to possibly infinite dimensional Hilbert spaces.} A subLagrangian subspace $L$ is called  
     \emph{Lagrangian} if  $H$ has the direct sum decomposition $H=L\oplus \alpha(L)$.} 
\end{definition}
\begin{remark}
     {Notice that the latter condition in the definition of a subLagrangian subspace implies that antiunvolutive Hilbert spaces of the form $(H,\mathrm{id}_H)$ have no subLagrangian subspaces except the 0-dimensional subspace if $H$ is finite dimensional, and no subLagrangian subspace at all if $H$ is infinite-dimensional. In either case, they have no  Lagrangian subspace.} 
\end{remark}   
\begin{remark}
 If $L_i$ are Lagrangian subspaces of $(H_i,\alpha_i)$ for $i\in I$, then $\oplus_i L_i$ is a Lagrangian subspace of $\oplus_i(H_i,\alpha_i)$.     
 \end{remark}
Now recall the definition of \emph{linear correspondences}.
\begin{definition}
 The category of linear correspondences over $\K$ is the category $\mathbf{LinCorr}_\K$ whose objects are (finite dimensional) vector spaces over $\K$ and with
 \[
 \mathrm{Hom}_{\mathbf{LinCorr}_\K}(V,W)=\{U\colon U \text{ is a subspace of } V\oplus W\}.
 \]
 If $V_i\xrightarrow{U_{ij}}V_j$ and $V_j\xrightarrow{U_{jk}}V_k$ are morphisms, the composition $V_i\xrightarrow{U_{jk}\circ U_{ij}}V_k$ is defined as
 \[
 U_{jk}\circ U_{ij}=\{(v_i, v_k)\in V_i\oplus V_k \,\,\vert \,\, \exists v_j\in V_j \text{ such that } (v_i,v_j) \in U_{ij} \text{ and } (v_j, v_k) \in U_{jk}\}
\]
The identity correspondence of $V$ is the diagonal subspace 
 \[
 \Delta_V=\{(v,v)\, : v\in V\} \subseteq V\oplus V.
 \]
\end{definition}
\begin{remark}\label{rem:span}
An useful point of view on the composition of linear correspondences we will come back later is the following. Let us denote by $\pi_i,\pi_j$ the projections $\pi_i\colon V_i\oplus V_j\to V_i$, and $\pi_j\colon V_i\oplus V_j\to V_j$. Then the composition of the inclusion $U_{ij}\hookrightarrow V_i\oplus V_j$ with the projections $\pi_i$ realizes $U_{ij}$ as a linear span:
\[
\begin{tikzcd}
&U_{ij}\arrow[dr,"\pi_j"]\arrow[dl,"\pi_i"']\\
V_i                 && V_j
\end{tikzcd}.
\]
We can then form the composition of linear spans $U_{jk}\circ_{\mathrm{spans}} U_{ij}$ given by the fiber product over $V_j$:
\[
\begin{tikzcd}
&&U_{ij}\times_{V_j}U_{jk}\arrow[dr]\arrow[dl]\\
&U_{ij}\arrow[dr,"\pi_j"]\arrow[dl,"\pi_i"']&&U_{jk}\arrow[dr,"\pi_k"]\arrow[dl,"\pi_j"']\\
V_i                 && V_j && V_k
\end{tikzcd}.
\]
One has
\[
U_{jk}\circ_{\mathrm{spans}} U_{ij}=\{(v_i,v_j,v_k)\in V_i\oplus V_j\oplus V_k\,\, \vert\,\, (v_i,v_j)\in U_{ij}\text{ and }(v_j,v_k)\in U_{jk}\}
\]
and so
\[
U_{jk}\circ U_{ij}=\pi_{i,k}(U_{jk}\circ_{\mathrm{spans}} U_{ij}),
\]
where $\pi_{i,k}$ is a shorthand notation for $(\pi_i,\pi_k)$.
\end{remark}

 \begin{lemma}\label{lemma:diagonal}
 Let $(H,\alpha)$ be an antiunvolutive Hilbert space. Then the identity correspondence on $H$ is a Lagrangian subspace of $(H,\alpha)\oplus (H,-\alpha)$.
 \end{lemma}
 
 \begin{lemma}\label{lemma:lagr-composition}
     Let $(H_i,\alpha_i)$ be antiunvolutive Hilbert spaces for $i=0,1,2$, and let $L_{ij}$ and $L_{jk}$ be Lagrangian subspaces of $(H_i,-\alpha_i)\oplus(H_j,\alpha_j)$ and of $(H_j,-\alpha_j)\oplus(H_k,-\alpha_k)$, respectively.
     The composition of linear correspondences $L_{jk}\circ L_{ij}$ is a Lagrangian subspace of $(H_i,-\alpha_i)\oplus(H_k,\alpha_k)$.
     \end{lemma}
{
 \begin{remark}
 The above lemma is generally false for infinite dimensional Hilbert spaces, see \cite[Example 2.7]{ludewig-roos}.
 \end{remark}
 }
 By the two lemmas above, we can define the category  of \emph{Lagrangian correspondences} as follows. 
 \begin{definition}
 Let $\mathbb{K}$ be $\mathbb{R}$ or $\mathbb{C}$. The category $\mathbf{LagrCorr}_{\mathbb{K}}$ of Lagrangian correspondences is the category with antiunvolutive Hilbert spaces over $\mathbb{K}$ as objects and with 
 \[
 \mathrm{Hom}_{\mathbf{LagrCorr}_\K}((H_i,\alpha_i),(H_j,\alpha_j))=\{L_{ij}\colon L_{ij}\text{ is a Lagrangian subspace of } (H_i,-\alpha_i)\oplus (H_j,\alpha_j) \}.
 \]
 The identity morphism and compositions are defined as for linear correspondences.
\end{definition} 
\begin{remark}
Forgetting the antiinvolutions and the inner product one obtains a forgetful functor
\[
\mathbf{LagrCorr}_\K \to \mathbf{LinCorr}_\K.
\]
If one retains the Hilbert space structures one gets the forgetful functor
\[
\mathbf{LagrCorr}_\K \to \mathbf{HilbCorr}_\K,
\]
where $\mathbf{HilbCorr}_\K$ is the category of Hilbert correspondences. In the setting of possibly infinite-dimensional Hilbert spaces, the linear subspaces giving the correspondences are required to be closed, and so also Lagrangian subspaces of antiunvolutive Hilbert spaces are required to be closed. {This requirement alone is however not sufficient in order to have a well defined category of Lagrangian correspondences and the choice of a polarization on Hilbert spaces is needed; see the beginning of Remark \ref{rem:von-neumann} for details.}
\end{remark}

\begin{lemma}
    The category $\mathbf{LagrCorr}_{\mathbb{K}}$ is symmetric monoidal with the direct sum as monoidal product. It has duals, given by $(H,\alpha)^\vee=(H,-\alpha)$
\end{lemma} 

\subsection{The Fock functor}    
With a quadratic vector space $(V,b)$ is associated a Clifford algebra $\Cl(V,b)$ defined as the quotient of the tensor algebra of $V$ by the ideal generated by the elements 
\[
v\otimes v+ b(v\otimes v)\mathbf{1}.
\]
In other words, in the Clifford algebra $\Cl(V,b)$ we have noncommutative polynomials in the elements of $V$ and the multiplication obeys the rule $v\cdot v=-b(v\otimes v)\mathbf{1}$. Since we associated a quadratic vector space with any antiunvolutive Hilbert space, we can talk of the Clifford algebra 
of an antiunvolutive Hilbert space, and we will write $\Cl(H,\alpha)$ for $\Cl(H,b_\alpha)$. {Since all monomials in the expression $v\otimes v+ b(v\otimes v)\mathbf{1}$ have even degree with respect to the natural grading of the tensor algebra of $H$ having elements of $H$ in degree $1$, the ideal defining the Clifford algebra is homogeneous $\mod 2$, and so $\Cl(H,b)$ is naturally a superalgebra. Looking at $\Cl(H,b)$ as a superalgebra and not simply as an algebra will be essential in what follows.}
\begin{remark}\label{rem:opposite-and-tensor}
    One has natural isomorphism of {super}algebras
    \[
    \Cl(H,-\alpha)\cong \Cl(H,\alpha)^{\mathrm{op}}; \qquad \Cl((H_j,\alpha_j)\oplus(H_k,\alpha_k))\cong \Cl(H_j,\alpha_j)\otimes_{\mathbb{K}} \Cl(H_k,\alpha_k)),
    \]
    {where $(-)^{\mathrm{op}}$ and $\otimes_{\mathbb{K}}$ denote the opposite superalgebra and the tensor product of superalgebras, respectively.
   It is important to stress that this statement would be false if we would look at Clifford algebras simply as algebras and not as superalgebras, and would not consider the $\mathbb{Z}
   /2\mathbb{Z}$-graded versions of the tensor product and op-operation.
    }
    \end{remark}
\begin{remark}
The { fact that the} ideal defining a Clifford algebra is homogeneous $\mod 2$, { giving Clifford algebras a natural structure of superalgebras, remains clearly true also when considering quadratic super-vector spaces.}    
\end{remark}    
\begin{lemma}
Let $L$ be a Lagrangian subspace of $(H,\alpha)$, and write $v=v_L+v_{\alpha(L)}$ for the decomposition of a vector $v$ of $H$ induced by the splitting $H=L\oplus\alpha(L)$. Let $F(L)=\bigwedge^\bullet \alpha(L)$ be the exterior algebra  on the subspace $\alpha(L)$ of $H$. { It has a natural $\mathbb{Z}/2\mathbb{Z}$-grading.} For an element $\alpha(w)\in \alpha(L)\subseteq F(L)$, set
\begin{align}\label{eq:clifford-action}
v_L\cdot \alpha(w)&= b_\alpha(v_L\otimes \alpha(w)) \\
\notag v_{\alpha(L)}\cdot \alpha(w)&= v_{\alpha(L)}\wedge \alpha(w).
\end{align}
and extend these to the whole of $F(L)$ by imposing the {(super-)}Leibniz rule. Then, \eqref{eq:clifford-action} induces an action of $\Cl(H,\alpha)$ on $F(L)$ making it a left $\Cl(H,\alpha)$-{super}module.
\end{lemma}
\begin{remark} In Physics literature the Clifford {(super-)}module $F(L)$ from the previous lemma is called a \emph{Fock module}, and one says that $v_L$ acts as an annihilation operator and that $v_{\alpha(L)}$ acts as a creation operator. Also, the unit element of the exterior algebra $F(L)$ is commonly called the \emph{vacuum vector} in the Physics literature, and denoted by the symbol $\Omega_L$.
\end{remark}
\begin{remark}
Notice that, by definition of $b_\alpha$ one can equivalently write $v_L\cdot \alpha(w)=\langle w|v_L  \rangle$
\end{remark}

If $L$ is a Lagrangian correspondence between $(H_i,\alpha_i)$ and $(H_j,\alpha_j)$, then $L$ is a Lagrangian subspace of $(H_i,-\alpha_i)\oplus (H_j,\alpha_j)$ and so  $F(L)=\bigwedge^\bullet (-\alpha_i\oplus \alpha_j)(L)$ is naturally a left $\Cl((H_i,-\alpha_i)\oplus (H_j,\alpha_j))$-{super}module.
    By Remark \ref{rem:opposite-and-tensor}, this is equivalently a left $\Cl(H_i,\alpha_i)^{\mathrm{op}}\otimes_\K \Cl(H_j,\alpha_j)$-{super}module, i.e., equivalently, a $(\Cl(H_j,\alpha_j),\Cl(H_i,\alpha_i))$-{super}bimodule. Explicitly, the left action of an element $u_j\in H_j\subseteq \mathrm{Cl}(H_j)$ and the right action of an element $u_i\in H_i\subseteq \mathrm{Cl}(H_i)$ on a homogeneous element $w$ in $F(L)$ are given by
    \[
    u_j\cdot w= (u_j,0)\cdot w\qquad \text{and}\qquad w\cdot u_i=-(-1)^w (0,u_i)\cdot w,
    \]
 respectively. 
    We are therefore tempted to define a 2-representation 
\[
\mathrm{Fock}\colon \mathbf{LagrCorr}_\K\to {2\mathrm{sVect}_\K}
\]
of the category of Lagrangian correspondences by associating with an object $(H,\alpha)$ of $\mathbf{LagrCorr}_\K$ the Clifford {super}algebra $\Cl(H,\alpha)$ and with a Lagrangian correspondence $(H_i,\alpha_i)\xrightarrow{L} (H_j,\alpha_j)$ the Fock bi{super}module $F(L)$. 

This is however \emph{not} the case, since one does not have a canonical isomorphism $F(L_{jk}\circ L_{ij})\cong F(L_{jk})\otimes_{\Cl(H_j,\alpha_j)} F(L_{ij})$. More precisely, the Clifford bi{super}modules $F(L_{jk}\circ L_{ij})$ and $F(L_{jk})\otimes_{\Cl(H_j,\alpha_j)} F(L_{ij})$ are indeed isomorphic, but not canonically so. Things however are not so bad, thanks to the following result; see \cite{kristel-ludewig-waldorf,ludewig}
\begin{proposition}
Let $\K=\mathbb{R}$ or $\K=\mathbb{C}$. Finite dimensional Clifford {super}algebras associated with nondegenerate quadratic forms are invertible objects in the Morita category {$2\mathrm{sVect}_\K$} of {super} 2-vector spaces over $\K$. Fock bi{super}modules are invertible morphisms in {$2\mathrm{sVect}_\K$}.
\end{proposition}
{
\begin{remark}
Here is another point where working in the setting of superalgebras is essential: in the setting of algebras, Clifford algebras of odd degree, i.e., those associated with odd-dimensional Hilbert spaces $H$ are \emph{not} invertible.
\end{remark}
}
In other words, in the terminology of this Chapter, $\mathrm{Fock}\colon \mathbf{LagrCorr}_\K\dashrightarrow {2\mathrm{sVect}_\K}$ is an invertible wannabe functor. Therefore it will naturally extend to a projective 2-representation of Lagrangian correspondences: we will have a canonical isomorphism
\[
\lambda_{L_{ij},L_{jk}}\colon F(L_{jk}\circ L_{ij})\otimes \mathrm{Pf}(L_{ij},L_{jk}) \xrightarrow{\sim} F(L_{jk})\otimes_{\Cl(H_j,\alpha_j)} F(L_{ij}),
\]
where
\[
\mathrm{Pf}(L_{ij},L_{jk}):=\mathrm{Hom}(F(L_{jk}\circ L_{ij}),F(L_{jk})\otimes_{\Cl(H_j,\alpha_j)} F(L_{ij}))
\]
is the \emph{Pfaffian line} of the pair $(L_{ij},L_{jk})$. Moreover the Pfaffian lines will come with distinguished isomorphisms 
\[
\phi_{L_{ij},L_{jk},L_{kl}}\colon \mathrm{Pf}(L_{ij},L_{kl}\circ L_{jk})\otimes \mathrm{Pf}(L_{jk},L_{kl})   \xrightarrow{\sim}\mathrm{Pf}(L_{jk}\circ L_{ij},L_{kl}) \otimes \mathrm{Pf}(L_{ij},L_{jk}) ,
\]
satisfying the coherence conditions given by the commutativity of diagrams
\begin{equation}\label{eq:coherenceC2-fock}
    \begin{tikzcd}[column sep={12em,between origins}]
    \mathrm{Pf}(L_{jk},L_{kl})\otimes \mathrm{Pf}(L_{ij},L_{jl}) \otimes \mathrm{Pf}(L_{il},L_{lm})\ar[dd,"\mathrm{id}\otimes \phi"']&& \mathrm{Pf}(L_{ij},L_{jk})\otimes \mathrm{Pf}(L_{ik},L_{kl})\otimes \mathrm{Pf}(L_{il},L_{lm})\ar[d,"\mathrm{id}\otimes \phi"]\ar[ll,"\phi\otimes\mathrm{id}"']&\\
    {}&&\mathrm{Pf}(L_{ij},L_{jk})\otimes \mathrm{Pf}(L_{kl},L_{lm})\otimes \mathrm{Pf}(L_{ik},L_{km})\ar[d,"\wr"]\\
    \mathrm{Pf}(L_{jk},L_{kl})\otimes \mathrm{Pf}(L_{jl},L_{lm})\otimes \mathrm{Pf}(L_{ij},L_{jm})\ar[dr,"\phi\otimes\mathrm{id}"']&&
    \mathrm{Pf}(L_{kl},L_{lm})\otimes \mathrm{Pf}(L_{ij},L_{jk}) \otimes \mathrm{Pf}(L_{ik},L_{km})
    \ar[dl,"\mathrm{id}\otimes \phi"]\\
    &\mathrm{Pf}(L_{kl},L_{lm}) \otimes \mathrm{Pf}(L_{jk},L_{km}) \otimes \mathrm{Pf}(L_{ij},L_{jm})
    \end{tikzcd},
    \end{equation}
where we wrote $L_{ik}$ for $L_{jk}\circ L_{ij}$, etc.,     
and together with the isomorphisms $\lambda_{L_{ij},L_{jk}}$ they will satisfy the coherence condition given by the commutativity of the diagram
\begin{equation}\label{coherence-lr}
\begin{tikzcd}[column sep={8em,between origins}]
&F(L_{kl})\otimes_{\mathrm{Cl}(H_k,\alpha_k)}  F(L_{jk})\otimes_{\mathrm{Cl}(H_j,\alpha_j)} F(L_{ij})
\\
F(L_{kl})\otimes_{\mathrm{Cl}(H_k,\alpha_k)} (F(L_{ik})\otimes \mathrm{Pf}(L_{ij},L_{jk}))  \arrow[ur,<-,"\mathrm{id}\otimes \lambda"]              && (F(L_{jl})\otimes \mathrm{Pf}(L_{jk},L_{kl}))\otimes_{\mathrm{Cl}(H_j,\alpha_j)} F(L_{ij})\arrow[ul,<-,"\lambda\otimes \mathrm{id}"']\\ 
(F(L_{kl})\otimes_{\mathrm{Cl}(H_k,\alpha_k)} F(L_{ik}))\otimes \mathrm{Pf}(L_{ij},L_{jk})\ar[u,<-,"\wr"'] && (F(L_{jl})\otimes_{\mathrm{Cl}(H_j,\alpha_j)} F(L_{ij}))\otimes \mathrm{Pf}(L_{jk},L_{kl})\arrow[u,<-,"\wr"']
\\
F(L_{il}) \otimes \mathrm{Pf}(L_{ik},L_{kl})\otimes \mathrm{Pf}(L_{ij},L_{jk}) \arrow[u,<-,"\lambda\otimes \mathrm{id}"] && F(L_{il})\otimes  \mathrm{Pf}(L_{ij},L_{jl})\otimes\mathrm{Pf}(L_{jk},L_{kl}) \arrow[u,<-,"\lambda\otimes \mathrm{id}"']
\\
F(L_{il}) \otimes \mathrm{Pf}(L_{ij},L_{jk})\otimes \mathrm{Pf}(L_{ik},L_{kl})
\arrow[u,<-,"\wr"] 
&& F(L_{il})\otimes \mathrm{Pf}(L_{jk},L_{kl})\otimes \mathrm{Pf}(L_{ij},L_{jl}) \arrow[u,<-,"\wr"'] \arrow[ll,<-,"\mathrm{id}\otimes \phi "']
\end{tikzcd}
\end{equation}
thus recovering the main result from \cite{ludewig-roos}. {In the terminology of projective 2-representations, the Pfaffian line $\mathrm{Pf}\colon  \mathbf{LagrCorr}_\K\to \mathbf{B}^2\Pic(\mathrm{sVect}_\K)$ is the 2-cocycle of the invertible projective 2-representation $\mathrm{Fock}\colon \mathbf{LagrCorr}_\K\to 2\mathrm{sVect}_\K/\!/\mathbf{B}\Pic(\mathrm{sVect}_\K)$.}

\begin{remark}\label{rem:pf-as-det}
One has an alternative description of the Pfaffian line $\mathrm{Pf}(L_{ij},L_{jk})$. 
 Namely,
one has a canonical isomorphism
\begin{equation}\label{eq:lines}
\mathrm{Pf}(L_{ij},L_{jk})\cong\det(\ker(\pi_{i,k})),
\end{equation}
where $\pi_{i,k}\colon L_{jk}\circ_{\mathrm{spans}} L_{ij}\to H_i\oplus H_k$ is the natural map from the fiber product to the product.
In terms of this description of the Pfaffian line, the canonical isomorphism 
\[
\lambda_{L_{ij},L_{jk}}\colon F(L_{jk}\circ L_{ij})\otimes \mathrm{Pf}(L_{ij},L_{jk}) \xrightarrow{\sim} F(L_{jk})\otimes_{\Cl(H_j,\alpha_j)} F(L_{ij}),
\]
is characterized as the unique homomorphism of $(\mathrm{CL}(H_j,\alpha_j)),\mathrm{CL}(H_i,\alpha_i))$-bi{super}modules such that
\[
\lambda_{L_{ij},L_{jk}}\colon \Omega_{L_{jk}\circ L_{ij}}\otimes_\K u_1\wedge\cdots \wedge u_n \mapsto \Omega_{L_{jk}}\otimes_{\Cl(H_j,\alpha_j)}  u_1\cdots u_n \Omega_{L_{ij}},
\]
where $n=\dim_\K(\ker(\pi_{i,k}\colon L_{jk}\circ_{\mathrm{spans}} L_{ij}\to H_i\oplus H_k))$, see
\cite[Theorem 2.15]{ludewig-roos}. 
Also, in terms of the description of the Pfaffian line as a determinant line one can explicitly write the isomorphism 
\[
\phi_{ijkl}\colon \mathrm{Pf}(L_{jk},L_{kl})\otimes  \mathrm{Pf}(L_{ij},L_{kl}\circ L_{jk}) \xrightarrow{\sim}\mathrm{Pf}(L_{ij},L_{jk})\otimes \mathrm{Pf}(L_{jk}\circ L_{ij},L_{kl}),
\]
see \cite[above Theorem 2.18]{ludewig-roos}. By means of these explicit descriptions of $\lambda$ and $\phi$, Ludewig and Roos are able to prove the commutativity of \eqref{coherence-lr}; this is their Theorem 2.18. Although we didn't need the description of the Pfaffian line as a determinant line in order to prove these results, this description will play a major role in the following section {so we provide the relevant details here. 
 let $M_{ijkl}\subseteq L_{ij}\oplus L_{jk}\oplus L_{kl}$ the subspace consisting of those triples $(v_{ij},v_{jk},v_{kl})$ such that
    \[
\begin{cases}
    \pi_i(v_{ij})=0\\
    \pi_j(v_{ij})=\pi_j(v_{jk})\\
    \pi_k(v_{jk})=\pi_k(v_{kl})\\
    \pi_l(v_{kl})=0.
\end{cases}
    \]
Also, let $K_{ijk}\subseteq L_{ij}\oplus L_{jk}$ the subspace consisting of those triples $(v_{ij},v_{ik})$ such that
    \[
\begin{cases}
    \pi_i(v_{ij})=0\\
    \pi_j(v_{ij})=\pi_j(v_{jk})\\
    \pi_k(v_{jk})=0
\end{cases}
    \]
where $L_{ik}=L_{jk}\circ_{\mathrm{spans}} L_{ij}$ and $L_{jl}=L_{kl}\circ_{\mathrm{spans}} L_{jk}$. Then $K_{ijk}= \ker(\pi_{i,k})$, $M_{ijkl}=\ker(\pi_{i,l})$, and we have short exact sequences
\[
0\to K_{ijk}\to M_{ijkl}\xrightarrow{\pi_{ikl}} K_{ikl}\to 0
\]
and
\[
0\to K_{jkl}\to M_{ijkl}\xrightarrow{\pi_{ijl}} K_{ijl}\to 0.
\]
The identity  $K_{ijk}= \ker(\pi_{i,k})$ is clear: $L_{jk}\circ_{\mathrm{spans}} L_{ij}$ is the space of pairs $(v_{ij},v_{jk})$ with $\pi_j(v_{ij})=\pi_j(v_{jk})$, so that $\ker(\pi_{i,k}\colon L_{jk}\circ_{\mathrm{spans}} L_{ij}\to H_i\oplus H_k)$ is precisely $K_{ijk}$. Similarly, $M_{ijkl}=\ker(\pi_{0,3})$. The map
\[
\pi_{ikl}\colon M_{ijkl}\to K_{ikl}
\]
sends a triple $(v_{ij},v_{jk},v_{kl})$ to the pair $(\pi_{ik}(v_{ij},v_{jk}),v_{kl})$. \par 
This is indeed an element in $K_{ikl}$: the element $\pi_{ik}(v_{ij},v_{jk})$ is an element of $L_{ik}$ by definition of $L_{ik}$ and we have $\pi_i(\pi_{ik}(v_{ij},v_{jk}))=\pi_i(\pi_i(v_{ij}),\pi_k(v_{jk}))=\pi_i(v_{ij})=0$, $\pi_k(\pi_{ik}(v_{ij},v_{jk}))=\pi_k(\pi_i(v_{ij}),\pi_k(v_{jk}))=\pi_k(v_{jk})=\pi_k(v_{kl})$, and $\pi_k(v_{kl})=0$. The map $\pi_{ikl}\colon M_{ijkl}\to K_{ikl}$ is surjective: if $(v_{ik},v_{kl})$ is an element in $K_{ikl}$ then $v_{ik}\in L_{ik}$ and so by definition of $L_{ik}$ we have $v_{ik}=\pi_{ik}(v_{ij},v_{jk})$ for some $v_{ij}\in L_{ij}$ and $v_{jk}\in L_{jk}$ with $\pi_j(v_{ij})=\pi_j(v_{jk})$. Then we have $\pi_i(v_{ij})=\pi_i(\pi_{ik}(v_{ij},v_{jk}))=\pi_i(v_{ik})=0$ and $\pi_k(v_{jk})=\pi_k(\pi_{ik}(v_{ij},v_{jk}))=\pi_k(v_{ik})=\pi_l(v_{kl})$. This shows that $(v_{ij},v_{jk},v_{kl})\in M_{ijkl}$ and $(v_{ik},v_{kl})=\pi_{ikl}(v_{ij},v_{jk},v_{kl})$.\par
The kernel of $\pi_{ikl}$ consists of those triples $(v_{ij},v_{jk},v_{kl})$ in $M_{ijkl}$ with $\pi_{ik}(v_{ij},v_{jk})=0$ and $v_{kl}=0$. These are equivalently the pairs $(v_{ij},v_{jk})$ such that $\pi_i(v_{ij})=0$, $\pi_j(v_{ij})=\pi_j(v_{jk})$ and $\pi_k(v_{jk})=0$. But this is precisely the definition of $K_{ijk}$. The proof of the second short exact sequence is analogous. We therefore have canonical isomorphisms
\[
\det(K_{ijk})\otimes\det(K_{ikl}) \xrightarrow{\sim} \det(M_{ijkl}) \xleftarrow{\sim} \det(K_{jkl})\otimes \det(K_{ijl})
\]
giving an isomorphisms 
\[
\phi_{ijkl}^{\det}\colon \mathrm{Pf}(L_{jk},L_{kl})\otimes  \mathrm{Pf}(L_{ij},L_{kl}\circ L_{jk}) \xrightarrow{\sim}\mathrm{Pf}(L_{ij},L_{jk})\otimes \mathrm{Pf}(L_{jk}\circ L_{ij},L_{kl}),
\]
 via the identification $\mathrm{Pf}(L_{ij},L_{jk})\cong \det(K_{ijk})$. One would expect that $\phi_{ijkl}=\phi_{ijkl}^{\det}$, but actually it is not so. Yet, since both $\phi_{ijkl}$ and $\phi_{ijkl}^{\det}$ are isomorphisms between 1-dimensional vector spaces, they will at most differ by a phase factor:
 \[
 \phi_{ijkl}=\omega_{ijkl}^{LR}\phi_{ijkl}^{\det}
 \]
 for some $\omega_{ijkl}^{LR}\in \mathbb{K}^\ast$. This phase factor, which we refer to as the Ludewig--Roos phase, has been explicitly determined in \cite{ludewig-roos}. In terms of their notation, that we do not recall here, one has 
 \[
 \omega_{ijkl}^{LR}=\frac{\bigwedge^{\mathrm{top}}\rho_{ijkl}}{\det(\rho_{ijkl}^*\rho_{ijkl})},
\]
see \cite[Theorem 2.18]{ludewig-roos}. A crucial property of the Ludewig--Roos phase that we will use in the following section, is that it is a coboundary:
\begin{equation}\label{eq:coboundary}
 \omega_{ijkl}^{LR}=\tau_{jkl}^{LR}(\tau_{ikl}^{LR})^{-1}\tau_{ijl}^{LR}(\tau_{ijk}^{LR})^{-1}
\end{equation}
for a suitable $\mathbb{K}^\ast$-valued 3-cochain $\tau^{LR}$; see \cite[Equation (47)]{ludewig-roos}.
}
\end{remark}

\subsection{The category $\mathbf{LagrSpans}_\K$} 

We now introduce the category of Lagrangian spans. It plays for Lagrangian correspondences the same role that linear spans play for linear correspondences. First we need to introduce \emph{generalized Lagrangians}.
\begin{definition}
   Let $(H,\alpha)$ be antiunvolutive Hilbert space. A generalized Lagrangian is a linear map $\eta\colon W\to H$ such that $\ker(\eta)$ is finite dimensional and $\mathrm{Im}(\eta)$ is a Lagrangian subspace of $(H,\alpha)$. 
\end{definition}
\begin{remark}
 Clearly the condition on the finite dimensionality of $\ker(\eta)$ is trivial when one restricts to finite dimensional Hilbert spaces, whereas it is a nontrivial condition in the general case.  
\end{remark}
\begin{definition}
Let $(H_i,\alpha_i)$ and $(H_j,\alpha_j)$ be antiunvolutive Hilbert spaces. A Lagrangian span between  $(H_i,\alpha_i)$ and $(H_j,\alpha_j)$ is a linear span
\[
\begin{tikzcd}
&W_{ij}\arrow[dr,"\pi_j"]\arrow[dl,"\pi_i"']\\
H_i                 && H_j
\end{tikzcd}
\]
between $H_i$ and $H_j$
such that $\pi_{0,1}\colon W_{ij}\to H_i\oplus H_j$ is a generalized Lagrangian for $(H_i,-\alpha_i)\oplus(H_j,\alpha_j)$.
\end{definition}
\begin{lemma}\label{lemma:span-to-corr}
   Let $(H_i,\alpha_i)$, $(H_j,\alpha_j)$ and $(H_k,\alpha_k)$ be antiunvolutive Hilbert spaces, and let $\eta_{ij}\colon W_{ij}\to H_i\oplus H_j $ and $\eta_{jk}\colon W_{jk}\to H_j\oplus H_k $ be Lagrangian spans between $(H_i,\alpha_i)$, $(H_j,\alpha_j)$ and between $(H_j,\alpha_j)$, $(H_k,\alpha_k)$, respectively. Then $\pi_{0,2}\colon W_{ij}\times_{H_j} W_{jk}\to H_i\oplus H_k$ is a Lagrangian span. Moreover
   \[
   \pi_{i,k}(W_{ij}\times_{H_j} W_{jk})=\eta_{jk}(W_{jk})\circ \eta_{ij}(W_{ij}),
   \]
   where on the right we have the composition of Lagrangian correspondences.
   \end{lemma}
\begin{proof}
    Let $L_{ij}=\eta_{ij}(W_{ij})$. The projections $\pi_i,\pi_j\colon W_{ij}\to H_i,H_j$ factor through the projections $\pi_i,\pi_j\colon L_{ij}\to H_i,H_j$. This gives
    \[
    \pi_{i,k}(W_{ij}\times_{H_j} W_{jk})=\pi_{i,k}(L_{ij}\times_{H_j}L_{jk}),
    \]
    i.e.,
    \[
    \pi_{i,k}(W_{ij}\times_{H_j} W_{jk})=L_{jk}\circ L_{ij}.
    \]
    In particular, $\eta_{ik}=\eta_{jk}\circ_{\mathrm{spans}}\eta_{ij}=\pi_{i,k}\colon W_{ij}\times_{H_j} W_{jk}\to H_i\oplus H_k$ is a Lagrangian span. 
\end{proof}
We can then define the category of Lagrangian spans as follows.
\begin{definition}
Let $\K=\mathbb{R}$ or $\K=\mathbb{C}$. The (2,1)-category $\mathbf{LagrSpans}_\K$ of Lagrangian correspondences over $\K$ is the 2-category having antiunvolutive Hilbert spaces as objects, Lagrangian correspondences as 1-morphisms and isomorphisms of Lagrangian correspondences as 2-morphisms. The composition is given by fiber product, and the identity morphism is given by the diagonal.
\end{definition}

\begin{remark}
Since $\pi_{i,k}\colon W_{ij}\times_{H_j} W_{jk}\to H_i\oplus H_k$ is the composition $\eta_{jk}\circ_{\mathrm{spans}}\eta_{ij}$ of Lagrangian spans, Lemma \ref{lemma:span-to-corr} tells us that
 we have a natural commutative diagram of functors, with identity filler, 
 \begin{equation*}
            \begin{tikzcd}{\mathbf{LagrSpans}_\K}
                \arrow[r]\arrow[d]&{\mathbf{LinSpans}_\K}
                \arrow[d]\\
                {\mathbf{LagrCorr}_\K}\arrow[ur,Leftarrow,shorten >=2mm, shorten <=2mm]
                \arrow[r]&{\mathbf{LinCorr}_\K}
            \end{tikzcd}
        \end{equation*}
where the vertical arrows maps a span $\eta\colon W\to H_i\oplus H_j$ to its image $\eta(W)$ and the horizontal arrows forget the antiunvolutive Hilbert space structures and only retain the linear structures.        
\end{remark}

\begin{remark}
Like $\mathbf{LagrCorr}_\K$, the category $\mathbf{LagrSpans}_\K$ is symmetric monoidal. In particular it is naturally pointed with base point the 0-dimensional Hilbert space. The based loop category $\Omega\mathbf{LagrSpans}_\K$ of $\mathbf{LagrSpans}_\K$ is (equivalent to) the monoidal category $(\mathrm{Vect}_\K,\oplus)$ of finite dimensional $\K$-vector spaces with direct sum as tensor product. This gives a natural embedding 
\[
\mathbf{B}(\mathrm{Vect}_\K,\oplus)\to \mathbf{LagrSpans}_\K.
\]
By composing this with the projection $p\colon \mathbf{LagrSpans}_\K\to \mathbf{LagrCorr}_\K$ we obtain a {lax} homotopy commutative diagram
\begin{equation}\label{eq:BVect}
   \begin{tikzcd}
            {\mathbf{B}(\mathrm{Vect}_\K},\oplus)
                \arrow[r]\arrow[d]&\ast\arrow[d]\arrow[dl,Rightarrow,shorten >=2mm, shorten <=2mm]\\
            {\mathbf{LagrSpans}_\K}
                \arrow[r,"p"]&{\mathbf{LagrCorr}_\K}
   \end{tikzcd}             
\end{equation}
with trivial 2-cell filler (the zero Hilbert space has only trivial{, i.e., identity,} endomorphisms in $\mathbf{LagrCorr}_\K$).
\end{remark}

The reason why we are interested in Lagrangian spans is that they provide a trivialization of the Pfaffian 2-cocycle. We begin by noticing that on $\mathbf{LagrSpans}_\K$ we can naturally define \emph{two} invertible wannabe functors. One is the pullback $p^\ast\mathrm{Fock}$ to  $\mathbf{LagrSpans}_\K$ of the Clifford/Fock wannabe functor considered in the previous section. To give the other one we introduce the following definition.
\begin{definition}
 Let $(H,\alpha)$ be an antiunvolutive Hilbert space, and let $\eta\colon W\to H$ be a generalized Lagrangian. The \emph{twisting line} of $\eta$ is the line
 \[
 \beta(\eta)=\det(\ker(\eta)).
 \]
 \end{definition}
Then we can define an invertible wannabe functor, denoted by the same symbol $\beta$ as follows.
\begin{definition}
 The invertible wannabe functor
 \[
 \beta\colon \mathbf{LagrSpans}_\K\to {2\mathrm{sVect}_\K}
 \]
 is defined by the associations
 \begin{align*}
 (H,\alpha)&\mapsto \K\\
 (\eta_{ij}\colon W_{ij}\to (H_i,-\alpha_i)\oplus (H_j,\alpha_j))&\mapsto \beta(\eta_{ij})=\det(\ker(\eta_{ij})).
 \end{align*}
 We denote by $l^\beta$ the $\mathrm{Pic}({\mathrm{sVect}_\K})$-valued 2-cocycle for the projective representation (again denoted by the same symbol $\beta$) associated with the invertible wannabe
 functor $\beta$.
 \end{definition}
\begin{lemma}
    Let $(H_i,\alpha_i)$, $(H_j,\alpha_j)$ and $(H_k,\alpha_k)$ be antiunvolutive Hilbert spaces, and let $\eta_{ij}\colon W_{ij}\to H_i\oplus H_j $ and $\eta_{jk}\colon W_{jk}\to H_j\oplus H_k $ be Lagrangian spans between $(H_i,\alpha_i)$, $(H_j,\alpha_j)$ and between $(H_j,\alpha_j)$, $(H_k,\alpha_k)$, respectively. Then
    \[
    l^\beta_{\eta_{ij},\eta_{jk}}=\beta(\eta_{ik})^{-1}\otimes \beta(\eta_{jk})\otimes \beta(\eta_{ij}).
    \]
\end{lemma}
\begin{proof}
  The wannabe functor $\beta$ maps every object in $\mathbf{LagrSpans}_\K$ to the unit object $\K$ of ${2\mathrm{sVect}_\K}$, and the trace $\mathrm{tr}\colon \mathrm{End}(\mathbf{1})\to \mathrm{End}(\mathbf{1})$ is the identity.  
\end{proof}
Then we  have the following.
\begin{proposition}\label{prop:beta-to-the-minus-one}
One has a natural isomorphism of 2-cocycles
\[
(l^\beta)^{-1}\cong p^*\mathrm{Pf}{\colon  \mathbf{LagrSpans}_\K\to \mathbf{B}^2\mathrm{Pic}(\mathrm{sVect}_\K)},
\]
where $p\colon \mathbf{LagrSpans}_\K\to \mathbf{LagrCorr}_\K$ is the projection.
\end{proposition}
\begin{proof}
We have to show that, given antiunvolutive  Hilbert spaces $(H_i,\alpha_i)$, $(H_j,\alpha_j)$ and $(H_k,\alpha_k)$ and Lagrangian spans $\eta_{ij}\colon W_{ij}\to H_i\oplus H_j $ and $\eta_{jk}\colon W_{jk}\to H_j\oplus H_k $ between $(H_i,\alpha_i)$, $(H_j,\alpha_j)$ and between $(H_j,\alpha_j)$, $(H_k,\alpha_k)$, respectively, we have a canonical isomorphism
\begin{equation}\label{eq:beta-trivializes}
\psi_{ijk}\colon  \beta(\eta_{jk})\otimes\beta(\eta_{ij})\otimes \mathrm{Pf}(L_{ij},L_{jk})\xrightarrow{\sim} \beta(\eta_{ik}),
\end{equation}
where $\eta_{ik}=\eta_{jk}\circ_{\mathrm{spans}}\eta_{ij}$ and $L_{ij}=\mathrm{Im}(\eta_{ij})$,
and that these isomorphisms make the diagrams
\begin{equation}\label{eq:coherence-beta}
\hskip -1 cm            \begin{tikzcd}
            &\beta_{il}\\
          \beta_{jl}\otimes \beta_{ij}\otimes \mathrm{Pf}(L_{ij},L_{jl})
          \ar[ur,"\psi_{ijl}"] && \beta_{kl}\otimes \beta_{ik}\otimes \mathrm{Pf}(L_{ik},L_{kl})\ar[ul,"\psi_{ikl}"']\\
           \beta_{kl}\otimes\beta_{jk}\otimes \mathrm{Pf}(L_{jk},L_{kl})\otimes\beta_{ij}\otimes  \mathrm{Pf}(L_{ij},L_{jl})\ar[u," \psi_{jkl}\otimes \mathrm{id}"]&&{}
          \\
          \beta_{kl}\otimes\beta_{jk}\otimes\beta_{ij}\otimes \mathrm{Pf}(L_{jk},L_{kl})\otimes  \mathrm{Pf}(L_{ij},L_{jl})
          \ar[u,"\wr"]
          \ar[rr," \mathrm{id}\otimes \phi_{ijkl}" ]&&
          \beta_{kl}\otimes\beta_{jk}\otimes\beta_{ij}\otimes \mathrm{Pf}(L_{ij},L_{jk})\otimes \mathrm{Pf}(L_{ik},L_{kl})\ar[uu,"  \mathrm{id}\otimes \psi_{ijk}\otimes \mathrm{id}"']
            \end{tikzcd}
        \end{equation} 
commute, where we wrote $\beta_{ij}=\beta(\eta_{ij})$.  {Again, the proof is secretly based on the use of associahedra in homotopy associative algebras.}      
\par
Let us consider the maps $\pi_j\circ \eta_{ij}\colon W_{ij}\to H_j$ and $\pi_j\circ \eta_{jk}\colon W_{jk}\to H_j$, and let us form the fiber product $W_{jk}\times_{H_j}W_{ij}$. Then we have a commutative diagram
\[
\begin{tikzcd}
{W_{jk}\times_{H_j}W_{ij}}\arrow[d,"{(\eta_{jk},\eta_{ij})}"'] \arrow[dr,"\eta_{jk}\circ_{\mathrm{spans}}\eta_{ij}"]&
\\
{L_{jk}\times_{H_j}L_{ij}}\arrow[r,"\pi_{0,2}"']&{H_i\oplus H_k},
 \end{tikzcd}
\]
where the vertical arrow is surjective. This induces the short exact sequence
\[
0\to \ker((\eta_{jk},\eta_{ij})\bigr\vert_{\ker(\eta_{jk}\circ_{\mathrm{spans}}\eta_{ij})})\to \ker(\eta_{jk}\circ_{\mathrm{spans}}\eta_{ij}) \xrightarrow{(\eta_{jk},\eta_{ij})} \ker(\pi_{i,k}) \to 0.
\]
The space $\ker((\eta_{jk},\eta_{ij})\bigr\vert_{\ker(\eta_{ik})})$ consists of those elements $(w_{jk},w_{ij})$ in $W_{jk}\oplus W_{ij}$ such that $\eta_{jk}(w_{jk})=0$, $\eta_{ij}(w_{ij})=0$ and $\pi_j(\eta_{jk}(w_{jk})=\pi_j(\eta_{ij}(w_{ij})$. The third condition is manifestly redundant, so that
\[
\ker((\eta_{jk},\eta_{ij})\bigr\vert_{\ker(\eta_{ik})})=\ker(\eta_{jk})\oplus\ker(\eta_{ij}).
\]
We therefore have the short exact sequence
\[
0\to \ker(\eta_{jk})\oplus\ker(\eta_{ij})\to \ker(\eta_{ik}) \xrightarrow{(\eta_{jk},\eta_{ij})}\ker(\pi_{i,k}) \to 0,
\]
and so a canonical isomorphism
\[
\det(\ker(\eta_{ik}))
\xleftarrow{\sim}  \det(\ker(\eta_{jk})\oplus\ker(\eta_{ij}))\otimes \det(\ker(\pi_{i,k}))
\cong  \det(\ker(\eta_{jk}))\otimes\det(\ker(\eta_{ij}))\otimes \det(\ker(\pi_{i,k})).
\]
{
Recalling the definition of $\beta$, let us write
\[
\psi^{\det}_{ijk}\colon \beta(\eta_{jk})\otimes\beta(\eta_{ij})\otimes \det(K_{ijk})\xrightarrow{\sim} \beta(\eta_{ik}).
\]
Recall from Remark \ref{rem:pf-as-det} that we have a canonical isomorphism $\mathrm{Pf}(L_{ij},L_{jk})\cong\det(\ker(\pi_{i,k}))$ and a distinguished 3-cochain $\tau^{LR}$. Using these we define a distinguished isomorphism
\[
\psi_{ijk}:=\tau^{LR}_{ijk}\psi_{ijk}^{\det}\colon \beta(\eta_{jk})\otimes\beta(\eta_{ij})\otimes \mathrm{Pf}(L_{ij},L_{jk})\xrightarrow{\sim} \beta(\eta_{ik}).
\]
To prove the commutativity of \eqref{eq:coherence-beta}, we need to recall from Remark \ref{rem:pf-as-det}
how the isomorphism $\phi_{ijkl}$ is described in terms of the identification $\mathrm{Pf}(L_{ij},L_{jk})\cong \det(\ker(\pi_{i,k}))$. We have $\phi_{ijkl}=\omega_{ijkl}^{LR}\phi_{ijkl}^{\det}$}
Let us now consider the commutative diagram
\[
\begin{tikzcd}
{W_{kl}\times_{H_k} W_{jk}\times_{H_j}W_{ij}}\arrow[d,"{(\eta_{kl},\eta_{jk},\eta_{ij})}"'] \arrow[dr,"\eta_{il}"]&
\\
{L_{kl}\times_{H_k} L_{jk}\times_{H_j}L_{ij}}\arrow[r,"\pi_{0,3}"']&{H_i\oplus H_l},
 \end{tikzcd}
\]
where $\eta_{il}=\eta_{kl}\circ_{\mathrm{spans}}\eta_{jk}\circ_{\mathrm{spans}}\eta_{ij}$
The vertical arrow in the above diagram is surjective. This induces the short exact sequence
\[
0\to \ker((\eta_{kl},\eta_{jk},\eta_{ij})\bigr\vert_{\ker(\eta_{il})})\to \ker(\eta_{il}) \xrightarrow{(\eta_{jk},\eta_{ij})} \ker(\pi_{i,k}) \to 0.
\]
The space $\ker((\eta_{kl},\eta_{jk},\eta_{ij})\bigr\vert_{\ker(\eta_{il})})$ consists of those elements $(w_{kl},w_{jk},w_{ij})$ in $W_{kl}\oplus W_{jk}\oplus W_{ij}$ such that $\eta_{kl}(w_{kl})=\eta_{jk}(w_{jk})=\eta_{ij}(w_{ij})=0$ and 
\[
\begin{cases}\pi_j(\eta_{jk}(w_{jk})=\pi_j(\eta_{ij}(w_{ij})\\
\pi_k(\eta_{jk}(w_{jk})=\pi_k(\eta_{kl}(w_{kl}).
\end{cases}
\]
The second condition is manifestly redundant, so that
\[
\ker((\eta_{kl},\eta_{jk},\eta_{ij})\bigr\vert_{\ker(\eta_{il})})=\ker(\eta_{kl})\oplus\ker(\eta_{jk})\oplus\ker(\eta_{ij}).
\]
We therefore have the short exact sequence
\[
0\to \ker(\eta_{kl})\oplus \ker(\eta_{jk})\oplus\ker(\eta_{ij})\to \ker(\eta_{il}) \xrightarrow{(\eta_{kl},\eta_{jk},\eta_{ij})}\ker(\pi_{il}) \to 0,
\]
and so a canonical isomorphism
\[
\det(\ker(\eta_{il}))
\xleftarrow{\sim}   \det(\ker(\eta_{kl})\otimes  \det(\ker(\eta_{jk})\otimes\det(\ker(\eta_{ij}))\otimes \det(\ker(\pi_{il})),
\]
that is a canonical isomorphism
\[
 \beta_{kl}\otimes  \beta_{jk}\otimes\beta_{ij}\otimes \det(M_{ijkl})\xrightarrow{\sim} 
\beta_{il}.
\]
Let us consider the diagram
{
\begin{equation}\label{eq:coherence-beta2}
       \hskip -2 cm     \begin{tikzcd}[column sep=-.1em]
            &\beta_{il}\\
          \beta_{jl}\otimes \beta_{ij}\otimes \det(K_{ijl})
          \ar[ur,"\psi_{ijl}^{\det}"] && \beta_{kl}\otimes \beta_{ik}\otimes  \det(K_{ikl})\ar[ul,"\psi_{ikl}^{\det}"']\\
           \beta_{kl}\otimes\beta_{jk}\otimes \det(K_{jkl})\otimes\beta_{ij}\otimes  \det(K_{ijl})\phantom{mmm}\ar[u," \psi_{jkl}^{\det}\otimes \mathrm{id}"]&\beta_{kl}\otimes\beta_{jk}\otimes\beta_{ij}\otimes \det(M_{ijkl})\ar[uu]&{}
          \\
          \beta_{kl}\otimes\beta_{jk}\otimes\beta_{ij}\otimes \det(K_{jkl})\otimes  \det(K_{ijl})
          \ar[u,"\wr"]
          \ar[rr," \mathrm{id}\otimes \phi_{ijkl}^{\det}" ]\ar[ur]&&
          \beta_{kl}\otimes\beta_{jk}\otimes\beta_{ij}\otimes \det(K_{ijk})\otimes \det(K_{ikl})\ar[uu,"  \mathrm{id}\otimes \psi_{ijk}^{\det}\otimes \mathrm{id}"']
          \ar[ul]
            \end{tikzcd}
        \end{equation} 
 where we used notation from Remark  \ref{rem:pf-as-det}.      }
The triangular subdiagram commutes by definition, so we are reduced to checking the commutativity of the quadrangular subdiagrams. We check the commutativity of the left subdiagram, the one of the right subdiagram being completely analogous.
Recall that $\mathrm{Pf}(L_{ij},L_{jk})=\det(K_{ijk})$, and that the canonical isomorphism
\[
\det(A)\otimes \det(C)\xrightarrow{\sim}\det(B)
\]
induced by a short exact sequence 
\[
0\to A\to B\to C\to 0
\]
of vector spaces is obtained by looking at elements of $A$ as elements of $B$ via the inclusion and by lifting elements of $C$ to $B$:
\[
(a_1\wedge\cdots \wedge a_{\dim A})\otimes (c_1\wedge\cdots\wedge c_{\dim C})\mapsto a_1\wedge a_{\dim A}\wedge \tilde{c}_1\wedge\cdots \wedge \tilde{c}_{\dim C};
\]
the term on the right hand side is independent of the choice of the lifts $\tilde{c}_i$.
 Let $(v_{ij},v_{jl})$ be an element in $K_{ijl}$, and let $(v_{jk},v_{kl})$ be an element in $L_{jk}\times_{H_k} L_{kl}$ with $\pi_{jk}(v_{jk},v_{kl})=v_{jl}$. Let us pick lifts $w_{ij}$, $w_{jk}$ and $w_{kl}$ of $v_{ij}$, $v_{jk}$ and $v_{kl}$, respectively. Then $(w_{ij},w_{jk},w_{kl})$ is an element in $\ker(\eta_{il})$. 
Following the clockwise route an element $(v_{ij},v_{jl})$ in $K_{ijl}$ can be lifted to the element $(w_{ij},w_{jk}, w_{kl})$ in $\ker(\eta_{il})$. Following the anticlockwise route, the element $(v_{ij},v_{jl})$ in $K_{ijl}$ can be first lifted to the element $(v_{ij},v_{jk},v_{jl})$ in $M_{ijkl}$ and then this can be lifted again to $(w_{ij},w_{jk}, w_{kl})$ in $\ker(\eta_{il})$. If we consider an element $(v_{jk},v_{kl})$ in $K_{jkl}$, instead, in the clockwise route it can be first lifted to $(w_{jk},w_{kl})$ in $\ker(\eta_{jl})$ and then to $(0,w_{jk},w_{kl})$ in $\ker(\eta_{il})$. Following the clockwise route, $(v_{jk},v_{kl})$ can be first lifted to $(0,v_{jk},v_{kl})$ in $M_{ijkl}$ and then this can be  lifted to $(0,w_{jk},w_{kl})$ in $\ker(\eta_{il})$. Finally, elements $w_{ij}, w_{jk}$ and $w_{kl}$ in $\ker(\eta_{ij})$, $\ker(\eta_{jk})$ and $\ker(\eta_{kl})$ are mapped to the elements $(w_{ij},0,0)$, $(0,w_{jk},0)$ and $(0,0,w_{jl})$ of $\ker(\eta_{il})$, respectively, both along the clockwise and along the anticlockwise route. So we have choices of lifts making the clockwise and the anticlockwise isomorphisms identical. Since the isomorphisms in the diagram are independent on the choices of lifts, this proves the commutativity of the left quadrilateral subdiagram. 
{Taking the outer diagram of \eqref{eq:coherence-beta2}, and using the identification $\det(K_{ijk})\cong \mathrm{Pf}(L_{ij},L_{ik})$, we obtain the commutative diagram
\[
       \hskip -2 cm     \begin{tikzcd}[column sep=.6em]
            &\beta_{il}\\
          \beta_{jl}\otimes \beta_{ij}\otimes \mathrm{Pf}(L_{ij},L_{jl})
          \ar[ur,"\psi_{ijl}^{\det}"] && \beta_{kl}\otimes \beta_{ik}\otimes \mathrm{Pf}(L_{ik},L_{kl})\ar[ul,"\psi_{ikl}^{\det}"']\\
           \beta_{kl}\otimes\beta_{jk}\otimes \mathrm{Pf}(L_{jk},L_{kl})\otimes\beta_{ij}\otimes  \mathrm{Pf}(L_{ij},L_{jl})\ar[u," \psi_{jkl}^{\det}\otimes \mathrm{id}"]&&{}
          \\
          \beta_{kl}\otimes\beta_{jk}\otimes\beta_{ij}\otimes \mathrm{Pf}(L_{jk},L_{kl})\otimes  \mathrm{Pf}(L_{ij},L_{jl})
          \ar[u,"\wr"]
          \ar[rr," \mathrm{id}\otimes \phi_{ijkl}^{\det}" ]&&
          \beta_{kl}\otimes\beta_{jk}\otimes\beta_{ij}\otimes \mathrm{Pf}(L_{ij},L_{jk})\otimes \mathrm{Pf}(L_{ik},L_{kl})\ar[uu,"  \mathrm{id}\otimes \psi_{ijk}^{\det}\otimes \mathrm{id}"']
            \end{tikzcd}
\]
By the coboundary equation \eqref{eq:coboundary}, this gives the commutative diagram
\[
       \hskip -2 cm     \begin{tikzcd}[column sep=1.2em]
            &\beta_{il}\\
          \beta_{jl}\otimes \beta_{ij}\otimes \mathrm{Pf}(L_{ij},L_{jl})
          \ar[ur,"\tau_{ijl}^{LR}\psi_{ijl}^{\det}"] && \beta_{kl}\otimes \beta_{ik}\otimes \mathrm{Pf}(L_{ik},L_{kl})\ar[ul,"\tau_{ikl}^{LR}\psi_{ikl}^{\det}"']\\
           \beta_{kl}\otimes\beta_{jk}\otimes \mathrm{Pf}(L_{jk},L_{kl})\otimes\beta_{ij}\otimes  \mathrm{Pf}(L_{ij},L_{jl})\ar[u," \tau_{jkl}^{LR}\psi_{jkl}^{\det}\otimes \mathrm{id}"]&&{}
          \\
          \beta_{kl}\otimes\beta_{jk}\otimes\beta_{ij}\otimes \mathrm{Pf}(L_{jk},L_{kl})\otimes  \mathrm{Pf}(L_{ij},L_{jl})
          \ar[u,"\wr"]
          \ar[rr," \mathrm{id}\otimes \omega_{ijkl}^{LR}\phi_{ijkl}^{\det}" ]&&
          \beta_{kl}\otimes\beta_{jk}\otimes\beta_{ij}\otimes \mathrm{Pf}(L_{ij},L_{jk})\otimes \mathrm{Pf}(L_{ik},L_{kl})\ar[uu,"  \mathrm{id}\otimes \tau_{ijk}^{LR}\psi_{ijk}^{\det}\otimes \mathrm{id}"']
            \end{tikzcd}
\]
 that is the commutativity of diagram   \eqref{eq:coherence-beta}.
}
\end{proof}

\begin{remark}
When the generalized Lagrangians $\eta_{ij}$ and $\eta_{jk}$ are ordinary Lagrangians, i.e., $\eta_{ij}$ is the inclusion $L_{ij}\hookrightarrow H_i\oplus H_j$ of a Lagrangian subspace, equation \eqref{eq:beta-trivializes} reduces to
\[
\beta(\eta_{ij}\times_{H_j}\eta_{jk}){\cong} \mathrm{Pf}(L_{ij},L_{jk}),
\]
i.e., to
\[
\det(\ker(\pi_{i,k}\colon L_{jk}\circ_{\mathrm{spans}}L_{ij}\to H_i\oplus H_k){\cong} \mathrm{Pf}(L_{ij},L_{ik}),
\]
and this is precisely the definition of the Pfaffian line $\mathrm{Pf}(L_{ij},L_{ik})$ as a determinantal line.
\end{remark}
\begin{corollary}
The invertible projective 2-representation
\[
p^\ast \mathrm{Fock}\otimes \beta\colon \mathbf{LagrSpans}_\K \to {2\mathrm{sVect}_\K/\!/\mathbf{B}\mathrm{Pic}(\mathrm{sVect}_\K)}
\]
is a linear 2-representation, i.e. factors through ${2\mathrm{sVect}_\K}$. 
\end{corollary}
\begin{proof}
The 2-cocycle for the projective 2-representation  $p^\ast \mathrm{Fock}\otimes \beta$ is $p^\ast\mathrm{Pf}\otimes l^\beta$, and this is trivial by Proposition \ref{prop:beta-to-the-minus-one}.  
\end{proof}
For emphasis, let us write explicitly how the 2-representation 
\[
p^\ast \mathrm{Fock}\otimes \beta\colon \mathbf{LagrSpans}_\K \to \mathrm{Pic}({2\mathrm{sVect}_\K})\hookrightarrow {2\mathrm{sVect}_\K}
\]
acts on objects and morphisms of $\mathbf{LagrSpans}_\K$. We have
\[
(H,\alpha) \mapsto \mathrm{Cl}(H,\alpha)
\]
and
\[
\eta_{ij}\mapsto F(\mathrm{Im}(\eta_{ij}))\otimes \det(\ker(\eta_{ij})).
\]
A direct proof that this assignment is a linear 2-representation  $\mathbf{LagrSpans}_\K \to {2\mathrm{sVect}_\K}$ has been recently given by Ludewig \cite{ludewig}. The statement had already appeared as Gluing Lemma 2.3.14 in \cite{what-is-an-elliptic-object}.

\begin{remark}\label{rem:von-neumann}
For applications to supersymmetric Euclidean or conformal quantum field theories, e.g., in the Stolz--Teichner approach to elliptic cohomology \cite{what-is-an-elliptic-object} , the use of infinite dimensional (super-)Hilbert spaces is unavoidable. The structures considered in this Section for finite dimensional  (super-)Hilbert spaces need to be properly refined in order to make the constructions presented there continue to work. For instance, we already mentioned that in the general setting of possibly infinite dimensional antiunvolutive Hilbert spaces one requires Lagrangian subspaces to be closed. This requirement alone is however not sufficient in order to have a well defined category of Lagrangian correspondences: the linear composition of two Lagrangian submanifolds is not necessarily closed (see \cite[Example 2.7]{ludewig-roos} and \cite[Example 2.18]{ludewig}). A solution to this and other issues, leading to a well defined category of Lagrangian correspondences between possibly infinite-dimensional (super-)Hilbert spaces, consists in considering \emph{polarized} (super-)Hilbert spaces, where a polarization of an antiunvolutive (super-)Hilbert space $(H,\alpha)$ is the choice of an equivalence class of subLagrangian subspaces $[U]$,\footnote{With respect to a suitable notion of equivalence, called `closeness'.} and with Lagrangian correspondences that are compatible with the polarizations on the source and on the target, see \cite{ludewig} for details.  Next, one has to extend the wannabe functor $\mathrm{Fock}$ to this setting. Here again problems arise from infinite-dimensionality: both the Clifford algebra $\mathrm{Cl}(H,\alpha)$
for an infinite dimensional Hilbert space and the Fock bimodule $\bigwedge^\bullet L$ will be infinite-dimensional when $H$ is an infinite dimensional Hilbert space, and statements of the previous Section do not hold true anymore. In particular, $\mathrm{Cl}(H,\alpha)$ will not be an invertible algebra, and $\bigwedge^\bullet L$ will not be an invertible bimodule. An expected solution comes from adding topology to the picture. More precisely, the following strategy is currently being worked out in detail by Raphael Schmidpeter in his PhD Thesis. One completes the algebraic Fock module $F(L)=\bigwedge^\bullet L$ to a super-Hilbert space $F_{\mathrm{Hilb}}(L)=(\bigwedge^\bullet L)_{\mathrm{Hilb}}$ and the Clifford superalgebra $\mathrm{Cl}(H,\alpha)$ to a $\ast$-(super-)algebra $\mathrm{Cl}(H,\alpha,[U])_{\mathrm{vN}}$ of bounded operators on $(\bigwedge^\bullet L)_{\mathrm{Hilb}}$. That is, the 2-category $2\mathrm{sVect}_{\mathbb{K}}$ of finite dimensional $\mathbb{K}$-superalgebras, superbimodules and intertwiners is extended to 2-category $2\mathrm{sHilb}_{\mathbb{K}}$ of super von Neumann algebras over $\mathbb{K}$, Hilbert superbimodules, and continuous intertwiners. The algebraic tensor product of (super-)bimodules is replaced by Connes' fusion product of Hilbert bi(super-)modules over von Neumann (super-)algebras. Doing so, the Fock invertible wannabe functor $\mathrm{Fock}\colon \mathbf{LagrCorr}_\K\dashrightarrow 2\mathrm{sVect}_\K$ extends to an invertible wannabe functor
\[
\mathrm{Fock}\colon \mathbf{LagrCorr}^{\mathrm{plrzd}}_\K\dashrightarrow 2\mathrm{sHilb}_\K,
\]
and the constructions of this Section apply.
\end{remark}

\bibliographystyle{alphaurl}

\end{document}